\documentclass[12pt,reqno]{amsart} 
\usepackage{amssymb} 
\usepackage{enumerate}
\usepackage{mathrsfs}

\newcommand{\bmid}{\mathrel{\big|}}

\newcommand{\G}[1][p]{\mathscr{G}_{#1}}
\newcommand{\GG}[1][p]{\mathscr{G}_{#1}^\wedge}
\newcommand{\GGp}[1]{\mathscr{G}_{#1}^\wedge}

\let\D=\displaystyle

\let\endpf=\endproof

\newcommand{\Sm}[1]{\textup{\begin{Small}#1\end{Small}}}
\newcommand{\pile}[3]{\begin{smallmatrix}#2\\#3\end{smallmatrix}}
\newcommand{\Pile}[3]{{\renewcommand{\arraystretch}{#1}\begin{matrix}#2\\#3\end{matrix}}}
\newcommand{\halfup}[2][2.2]{\raisebox{#1ex}[0pt]{$#2$}}
\newcommand{\up}[2][2.2]{\raisebox{#1ex}[0pt]{$#2$}}

\newcommand{\UUU}{\mathbf{U}}
\newcommand{\widebar}[1]{\overset{\mskip2mu\hrulefill\mskip2mu}{#1}
		\vphantom{#1}}
\newcommand{\4}[1]{\widebar{#1}}
\newcommand{\5}[1]{\widehat{#1}}

\newcommand{\7}{^\vee}
\newcommand{\9}[1]{{}^{#1}\!}   %% left conjugation
\newcommand{\EE}[1]{\mathbf{E}_{#1}}
\newcommand{\autff}{\autf\7}
\newcommand{\outff}{\outf\7}
\newcommand{\xx}{\mathbf{x}}
\newcommand{\xa}{\mathbf{a}}
\newcommand{\xb}{\mathbf{b}}

\newcommand{\scal}{_{\textup{sc}}}

\newlength{\short}
\newcommand{\boldd}[1]{{\mathversion{bold}\textbf{#1}}}

\newcommand{\lie}[3]{\def\test{#2}\def\tst{G}\ifx\test\tst{{}^{#1}#2_{#3}}
\else{{}^{#1}\!#2_{#3}}\fi}

\let\oldcirc=\circ
\renewcommand{\circ}{\mathchoice
    {\mathbin{\scriptstyle\oldcirc}}{\mathbin{\scriptstyle\oldcirc}}
    {\mathbin{\scriptscriptstyle\oldcirc}}
    {\mathbin{\scriptscriptstyle\oldcirc}}}

\newlength{\upto}\newlength{\dnto}
\newcommand{\I}[2]{\addtolength{\upto}{#1pt}\addtolength{\dnto}{#2pt}%
{\vrule height\upto depth\dnto width 0pt}}

\numberwithin{equation}{section}

\let\endpf=\endproof
\renewcommand{\endproof}{\endpf\setcounter{equation}{0}}

\mathcode`\:="003A
\mathchardef\cdot="0201
\def\beq#1\eeq{\begin{equation*}#1\end{equation*}}
\def\beqq#1\eeqq{\begin{equation}#1\end{equation}}

\let\emptyset=\varnothing
\renewcommand{\:}{\colon}   %% as in f:X-->Y
\newcommand{\longline}{\bigskip\centerline{\hbox to 5cm{\hrulefill}}\bigskip}

\newcommand{\mxtwo}[4]{\left(\begin{smallmatrix}#1&#2\\#3&#4
\end{smallmatrix}\right)}

\newcommand{\Mxtwo}[4]{\begin{pmatrix}#1&#2\\#3&#4\end{pmatrix}}

\newcommand{\mxfourb}[8]{#1&#2&#3&#4\\#5&#6&#7&#8\end{smallmatrix}\right)}
\newcommand{\mxfoura}[8]{\left(\begin{smallmatrix}#1&#2&#3&#4\\#5&#6&#7&#8\\}
\DeclareMathAlphabet\EuR{U}{eur}{m}{n}
\SetMathAlphabet\EuR{bold}{U}{eur}{b}{n}

\newcommand{\higherlim}[2]{\displaystyle\setbox1=\hbox{\rm lim}
	\setbox2=\hbox to \wd1{\leftarrowfill} \ht2=0pt \dp2=-1pt
	\setbox3=\hbox{$\scriptstyle{#1}$}
	\def\test{#1}\ifx\test\empty
	\mathop{\mathop{\vtop{\baselineskip=5pt\box1\box2}}}\nolimits^{#2}
	\else
	\ifdim\wd1<\wd3
	\mathop{\hphantom{^{#2}}\vtop{\baselineskip=5pt\box1\box2}^{#2}}_{#1}
	\else
	\mathop{\mathop{\vtop{\baselineskip=5pt\box1\box2}}_{#1}}%
	\nolimits^{#2}
	\fi\fi}
\newcommand{\higherlimm}[2]{\setbox1=\hbox{\rm lim}
	\setbox2=\hbox to \wd1{\leftarrowfill} \ht2=0pt \dp2=-1pt
	\mathop{\mathop{\vtop{\baselineskip=5pt\box1\box2}}}\limits_{#1}
	\nolimits^{#2}}

\newcounter{let} 
\loop\stepcounter{let}
\expandafter\edef\csname cal\alph{let}\endcsname%
{\noexpand\mathcal{\Alph{let}}}
\ifnum\thelet<26\repeat

\newcommand{\tdef}[2][]{\expandafter\newcommand\csname#2\endcsname%
{#1\textup{#2}}}
\tdef{Iso}   \tdef{Aut}    \tdef{Out}    \tdef{Inn}    \tdef{Hom}   
\tdef{Ext}   \tdef{Irr}    \tdef{rad}    \tdef{Sym}    \tdef{ord}
\tdef{End}   \tdef{Inj}    \tdef{map}    \tdef{Ker}    \tdef{Ob}    
\tdef{Mor}   \tdef{Res}    \tdef{Spin}   \tdef{Id}     \tdef{Fr}
\tdef{rk}    \tdef{conj}   \tdef{incl}   \tdef{proj}   \tdef{diag} 
\tdef{trf}   \tdef{Sol}     \tdef{cj}    \tdef{Outdiag} \tdef{lcm}
\tdef{free}  \tdef{supp}   \tdef{soc}    \tdef{Ind}
\tdef[_]{typ}   \tdef[^]{op}    \tdef[^]{ab}

\newcommand{\fdef}[1]{\expandafter\newcommand\csname#1\endcsname%
{\mathfrak{#1}}}
\fdef{X}  \fdef{red}  \fdef{foc}  \fdef{hyp}

\newcommand{\qq}{\mathfrak{q}}

\newcommand{\bbdef}[1]{\expandafter\newcommand% 
\csname#1\endcsname{\mathbb{#1}}}
\bbdef{C} \bbdef{F} \bbdef{R} \bbdef{Z} \bbdef{Q}

\renewcommand{\gg}{\mathbb{G}}

\newcommand{\itdef}[1]{\expandafter\newcommand\csname#1\endcsname%
{\textit{#1}}}
\itdef{PSL}  \itdef{PSU}  \itdef{SL}  \itdef{SU}  \itdef{GL} \itdef{GU}
\itdef{UT}   \itdef{Sp}   \itdef{PSp} \itdef{PGL} \itdef{GO} \itdef{Co}
\itdef{SO}   \itdef{SD}   \itdef{Th}  \itdef{He}  \itdef{Sz} \itdef{AGL}
\itdef{Suz}  \itdef{McL}

\newcommand{\SSL}{\textit{$\varSigma$L}}

\newcommand{\PSSL}{\textit{P$\varSigma$L}}
\newcommand{\PGGL}{\textit{P$\varGamma$L}}

\newcommand{\gee}{\varepsilon}

\newcommand{\gen}[1]{\langle{#1}\rangle}
\newcommand{\Gen}[1]{\bigl\langle{#1}\bigr\rangle}

\let\nsg=\normal
\let\nnsg=\ntrianglelefteq

\newcommand{\syl}[2]{\textup{Syl}_{#1}(#2)}
\newcommand{\sylp}[1]{\syl{p}{#1}}
\newcommand{\autf}{\Aut_{\calf}}

\newcommand{\outf}{\Out_{\calf}}

\newcommand{\homf}{\Hom_{\calf}}

\newcommand{\sminus}{\smallsetminus}
\newcommand{\defeq}{\overset{\textup{def}}{=}}

\renewcommand{\Im}{\textup{Im}}

\newcommand{\sd}[1]{\overset{{#1}}{\rtimes}}

\let\til=\widetilde

\newcommand{\longleft}[1]{\;{\leftarrow%
\count255=0 \loop \mathrel{\mkern-6mu}%
    \relbar\advance\count255 by1\ifnum\count255<#1\repeat}\;}
\newcommand{\longright}[1]{\;{\count255=0 \loop \relbar\mathrel{\mkern-6mu}%
    \advance\count255 by1\ifnum\count255<#1\repeat\rightarrow}\;}
\newcommand{\Right}[2]{\overset{#2}{\longright#1}}
\newcommand{\RIGHT}[3]{\mathrel{\mathop{\kern0pt\longright#1}
	\limits^{#2}_{#3}}}

\newcommand{\LEFT}[3]{\mathrel{\mathop{\kern0pt\longleft#1}\limits^{#2}_{#3}}
}
\newcommand{\longleftright}[1]{\;{\leftarrow\mathrel{\mkern-6mu}%
    \count255=0\loop\relbar\mathrel{\mkern-6mu}% 
    \advance\count255 by1\ifnum\count255<#1\repeat\rightarrow}\;} 
 
\newcommand{\onto}[1]{\;{\count255=0 \loop \relbar\joinrel
    \advance\count255 by1
    \ifnum\count255<#1 \repeat \twoheadrightarrow}\;}
\newcommand{\Onto}[2]{\overset{#2}{\onto#1}}
\newcommand{\RLEFT}[3]{\mathrel{%
   \mathop{\vcenter{\baselineskip=0pt\hbox{$\kern0pt\longright#1$}%
   \hbox{$\kern0pt\longleft#1$}}}\limits^{#2}_{#3}}}

\numberwithin{table}{section}

\renewenvironment{enumerate}[1][]
{\begin{enumerat}[#1]\setlength{\itemsep}{6pt}}{\end{enumerat}}

\renewenvironment{itemize}
{\begin{itemiz}\setlength{\itemsep}{6pt}\setlength{\itemindent}{-15pt}}
{\end{itemiz}}

\newenvironment{enuma}{\begin{enumerate}[{\rm(a) }]}{\end{enumerate}}
\newenvironment{enumi}{\begin{enumerate}[{\rm(i) }]}{\end{enumerate}}
\newenvironment{enum1}{\begin{enumerate}[{\rm(1) }]}{\end{enumerate}}

\newtheorem{Thm}{Theorem}[section]
\newtheorem{Prop}[Thm]{Proposition}
\newtheorem{Cor}[Thm]{Corollary}
\newtheorem{Lem}[Thm]{Lemma}

\newtheorem{Not}[Thm]{Notation}

\newtheorem{prop}[Thm]{Proposition}
\newtheorem{thm}[Thm]{Theorem}

\newtheorem{lem}[Thm]{Lemma}

\theoremstyle{definition}
\newtheorem{Defi}[Thm]{Definition} 
\newtheorem{defn}[Thm]{Definition}

\newtheorem{Ex}[Thm]{Example}

\theoremstyle{remark}

\title{Reduced fusion systems over $p$-groups with 
abelian subgroup of index $p$: II}

\author{David A.\ Craven}
\address{School of Mathematics, University of Birmingham, Edgbaston, 
Birmingham, B15 2TT, United Kingdom}
\email{d.a.craven@bham.ac.uk}
\thanks{D.\ Craven is financially supported by a Royal Society University 
Research Fellowship}

\author{Bob Oliver}
\address{Universit\'e Paris 13, Sorbonne Paris Cit\'e, LAGA, UMR 7539 du CNRS, 
99, Av. J.-B. Cl\'ement, 93430 Villetaneuse, France.}
\email{bobol@math.univ-paris13.fr}
\thanks{B. Oliver is partially supported by UMR 7539 of the CNRS}

\author{Jason Semeraro}
\address{Heilbronn Institute for Mathematical Research, University of 
Leicester, Leicester, United Kingdom}
\email{js13525@bristol.ac.uk}
\thanks{J. Semeraro was an EPSRC-funded DPhil student while some of this 
research was carried out.}
\date{\today}


\begin{document}

\begin{abstract} 
Let $p$ be an odd prime, and let $S$ be a $p$-group with a unique 
elementary abelian subgroup $A$ of index $p$. We classify the simple fusion 
systems over all such groups $S$ in which $A$ is essential. The resulting 
list, which depends on the classification of finite simple groups, includes 
a large variety of new, exotic simple fusion systems.
\end{abstract}

\maketitle

A saturated fusion system $\calf$ over a finite $p$-group $S$ is a category 
whose objects are the subgroups of $S$, whose morphisms are monomorphisms 
between subgroups, and whose morphism sets satisfy certain axioms, 
formulated originally by Puig and motivated by properties of conjugacy 
relations between $p$-subgroups of a finite group. For example, to each 
finite group $G$ and each Sylow $p$-subgroup $S$ of $G$, one associates the 
saturated fusion system $\calf_S(G)$ over $S$ whose morphisms are those 
homomorphisms induced by conjugation in $G$. We refer to \cite{AKO} or 
\cite{Craven} for a detailed introduction to the theory of saturated fusion 
systems, and to the beginning of Section \ref{s:background} for a little more 
detail about these definitions.

A saturated fusion system $\calf$ is \textit{simple} if it contains no 
non-trivial proper normal subsystems (see \cite[Definition I.6.1]{AKO} or 
\cite[Sections 5.4 and 8.1]{Craven} for the precise definition of a normal 
subsystem). In this paper, we continue the study, started in \cite{indp1}, 
of simple fusion systems $\calf$ over non-abelian $p$-groups $S$ which have 
an abelian subgroup of index $p$. For $p=2$, this was handled in 
\cite[Proposition 5.2(a)]{AOV2}: $S$ must be dihedral, semidihedral or 
wreathed, and $\calf$ must be the $2$-fusion system of $\PSL_2(q)$ for 
$q\equiv\pm1$ (mod $8$) or of $\PSL_3(q)$ for $q$ odd. So assume $p$ is an 
odd prime. If $S$ has more than one such subgroup, then $|S|=p^3$ by 
\cite[Theorem 2.1]{indp1}, and this case was dealt with earlier in 
\cite{RV}. Thus we assume that $S$ has a \emph{unique} abelian subgroup of 
index $p$, which we denote $A$. If $A$ is not $\calf$-essential (see 
Definition \ref{d:subgroups}), then we are in the situation of 
\cite{indp1}, and \cite[Theorem 2.8]{indp1} gives a complete 
characterization of simple fusion systems on $S$. In contrast to the 
situation when $p=2$, most of the fusion systems found in \cite{indp1} are 
\emph{exotic}; i.e., they are not fusion systems of finite groups. 

In this paper, we handle the case when $A$ is an $\calf$-essential subgroup 
and has exponent $p$, and again find a very large variety of exotic fusion 
systems. Our main tool is Theorem \ref{t3:s/a}, which gives precise details 
concerning the way in which the structure of $\calf$ is controlled by the 
action of $\Aut_\calf(A)$ on $A$. Indeed, Theorem \ref{t3:s/a} and its 
Corollary \ref{cor:s/a} reduce the problem of classifying fusion systems on 
$S$ to that of determining all pairs $(G,A)$, where $G$ is a finite group 
(the candidate for $\autf(A)$) and $A$ is an $\F_pG$-module that satisfy a 
certain list of conditions.

The bulk of our analysis is thus centred on classifying modules satisfying 
the required conditions. After some preliminary results in Section 3, the 
main results are summarized in Theorem \ref{t:alm.simple.reps} and Table 
\ref{tbl:alm.simple.reps}. Certain cases are then covered in more detail in 
Propositions \ref{SL2(p)-summary}, \ref{Ap-summary}, and 
\ref{p:reduce2simple-x}. By combining these results with Theorem 
\ref{t3:s/a}, one can get a complete list of all simple fusion systems of 
the type described. (A few explicit examples are worked out at the end of 
Section \ref{s:reps-summary}.) Most of the results involving lists 
of modules depend on the classification of finite simple groups 
(CFSG), which is thus assumed throughout Sections 
\ref{s:reps-summary} and \ref{sec:nonsimple} and also in Lemma 
\ref{list-simp}.

Our strategy for listing modules is based on Aschbacher's classification 
\cite{Asch} of subgroups $G<\GL_n(p)$. This splits into two cases, 
according to whether or not the image of $G$ in $\PGL_n(p)$ is almost 
simple. If it is not almost simple, Aschbacher gives a short list of 
possibilities, which we make more explicit (in our situation) in Section 
\ref{sec:nonsimple} (Propositions \ref{p:wreathcase} and 
\ref{p:reduce2simple}). When $G/Z(G)$ is almost simple, we use CFSG to 
check each of the possibilities for $G$ in Sections 
\ref{sec:gl2p}--\ref{s:exceptional}.

We are still left with the case where $S$ has a unique abelian subgroup of 
index $p$ which is $\calf$-essential and of exponent greater than $p$. The 
second author plans to handle this in a later paper with Albert Ruiz.

Our notation is mostly standard. For example, $p^{1+2k}_\pm$ denotes an 
extraspecial group of order $p^{2k+1}$, where (for odd $p$) $p^{1+2k}_+$ 
has exponent $p$ and $p^{1+2k}_-$ exponent $p^2$. Also, $\9xg=xgx^{-1}$, 
$g^x=x^{-1}gx$, $G'$ is the derived (commutator) subgroup of $G$, and 
$A\circ B$ denotes a central product of groups $A$ and $B$.

As usual, $F(G)$ denotes the Fitting subgroup of the finite group $G$: the 
largest normal nilpotent subgroup of $G$ (i.e., the product of the subgroups 
$O_p(G)$ for all $p$). Also, $E(G)$ denotes the layer: the central product 
of the components of $G$ (the subnormal quasisimple subgroups). Thus 
$F^*(G)=F(G)E(G)$ is the generalized Fitting subgroup.

The following definition will be useful in Sections 
\ref{sec:gl2p}--\ref{s:exceptional}.

\begin{defn}\label{d:type}
Let $H$ be a finite simple group. A finite group G is \emph{of type $H$} if 
$Z(G)$ is cyclic and $F^*(G)/Z(G)\cong H$.
\end{defn}

\section{Background}
\label{s:background}

We begin by recalling some definitions. 
When $G$ is a finite group and $S\in\sylp{G}$, the \emph{$p$-fusion system of 
$G$} is the category $\calf_S(G)$ whose objects are the subgroups of $S$, 
and where for $P,Q\le S$,
	\begin{align*} 
	\Hom_{\calf_S(G)}&(P,Q) = \Hom_G(P,Q) \\ 
	&\defeq \bigl\{ \varphi\in\Hom(P,Q) \bmid 
	\textup{$\varphi=c_g=(x\mapsto\9gx)$ for some 
	$g\in G$ such that $\9gP\le Q$} \bigr\}. 
	\end{align*}
More generally, a \emph{saturated fusion system} over a $p$-group $S$ is a 
category $\calf$ whose objects are the subgroups of $S$, where  
$\homf(P,Q)$ is a set of injective homomorphisms from $P$ to $Q$ for each 
$P,Q\le S$, and which satisfies certain axioms, due originally to Puig, and 
listed (in the form we use them) in \cite[Definition I.2.2]{AKO} and 
\cite[\S\,4.1]{Craven}. 

The following terminology is used to describe certain subgroups in a fusion 
system. 

\begin{Defi} \label{d:subgroups}
Fix a prime $p$, a finite $p$-group $S$, and a saturated fusion system $\calf$
over $S$.  Let $P\le{}S$ be any subgroup.
\begin{itemize} 
\item $P^\calf$ denotes the set of subgroups of $S$ which are 
$\calf$-conjugate (isomorphic in $\calf$) to $P$.  Also, $g^\calf$ 
denotes the $\calf$-conjugacy class of an element $g\in{}S$ (the set of 
images of $g$ under morphisms in $\calf$).

\item $P$ is \emph{fully normalized} in $\calf$ (\emph{fully centralized} 
in $\calf$) if $|N_S(P)|\ge|N_S(Q)|$ ($|C_S(P)|\ge|C_S(Q)|$) for each $Q\in 
P^\calf$. 

\item $P$ is \emph{$\calf$-centric} if $C_S(Q)=Z(Q)$ for each 
$Q\in{}P^\calf$.  

\item $P$ is \emph{$\calf$-essential} if $P<S$, $P$ is $\calf$-centric and 
fully normalized in $\calf$, and $\outf(P)\defeq\autf(P)/\Inn(P)$ contains 
a strongly $p$-embedded subgroup. Here, a proper subgroup $H<G$ of a finite 
group $G$ is \emph{strongly $p$-embedded} if $p\bmid|H|$, and 
$p\nmid|H\cap gHg^{-1}|$ for each $g\in G\sminus H$. \\Let $\EE\calf$ denote the set of all $\calf$-essential subgroups of $S$.  

\item $P$ is \emph{normal in $\calf$} if each morphism $\varphi\in\homf(Q,R)$ in $\calf$ extends to a morphism $\widebar{\varphi}\in\homf(PQ,PR)$ such that $\widebar{\varphi}(P)=P$. The 
maximal normal $p$-subgroup of a saturated fusion system $\calf$ is denoted 
$O_p(\calf)$.  

\item $P$ is \emph{strongly closed in $\calf$} if for each $g\in{}P$, 
$g^\calf\subseteq P$. 

\item $\foc(\calf)=\Gen{gh^{-1}\,\big|\,g\in{}S,~h\in{}g^\calf}$.

\end{itemize}
\end{Defi}

The above definition of $\calf$-essential subgroups is motivated by the 
following version for fusion systems, due to Puig, of the 
Alperin--Goldschmidt fusion theorem.

\begin{Thm}[{\cite[Theorem I.3.5]{AKO}, \cite[Theorem 4.51]{Craven}}] \label{AFT}
For each saturated fusion system $\calf$ over a finite $p$-group $S$, each 
morphism in $\calf$ is a composite of restrictions of $\calf$-automorphisms 
of $S$ and of $\calf$-essential subgroups of $S$.
\end{Thm}

Let $O^p(\calf)$ and $O^{p'}(\calf)$ denote the smallest normal fusion 
subsystems of $p$-power index, and of index prime to $p$, respectively.  
Such normal subsystems are defined by analogy with finite groups, and we 
refer to \cite[\S\,I.7]{AKO} or \cite[\S\,7.5]{Craven} for their precise 
definitions and properties.

\begin{Defi} \label{d:red-simple}
For any saturated fusion system $\calf$, 
\begin{itemize} 

\item $\calf$ is \emph{reduced} if $O_p(\calf)=1$, $O^p(\calf)=\calf$, and 
$O^{p'}(\calf)=\calf$;

\item $\calf$ is \emph{simple} if it contains no non-trivial proper normal 
fusion subsystems, in the sense of \cite[Definition I.6.1]{AKO} or 
\cite[\S\S\,5.4 \& 8.1]{Craven}; and 

\item $\calf$ is \emph{realizable} if $\calf=\calf_S(G)$ for some finite 
group $G$ with Sylow $p$-subgroup $S$.

\end{itemize}
\end{Defi}

For any saturated fusion system $\calf$ over $S$, $O^p(\calf)$, 
$O^{p'}(\calf)$, and $\calf_{O_p(\calf)}(O_p(\calf))$ are all normal 
subsystems of $\calf$.  Hence $\calf$ is reduced if it is simple.  If 
$\cale\nsg\calf$ is a normal subsystem over the subgroup $T\nsg{}S$, then 
by definition of normality, $T$ is strongly closed in $\calf$.  Thus a 
reduced fusion system is simple if it has no proper non-trivial strongly 
closed subgroups.

The next proposition gives some very general conditions for a fusion 
system to be reduced.

\begin{Prop} \label{Q<|F}
The following hold for a saturated fusion system $\calf$ over a finite 
$p$-group $S$. 
\begin{enuma}  

\item For each $Q\nsg{}S$, $Q\nsg\calf$ if and only if for each 
$P\in\EE\calf\cup\{S\}$, $Q\le{}P$ and $Q$ is $\autf(P)$-invariant. 

\item We have $\foc(\calf) = \Gen{[P,\autf(P)] \,\big|\, 
P\in\EE\calf\cup\{S\} }$.

\item In all cases, $O^p(\calf)=\calf$ if and only if $\foc(\calf)=S$. 

\item If each $P\in\EE\calf$ is minimal in the set of all $\calf$-centric 
subgroups, then $O^{p'}(\calf)=\calf$ if and only if 
$\autf(S)=\Gen{\Inn(S),\autf^{(P)}(S)\,\big|\,P\in\EE\calf}$, where for 
$P\le S$,
	\[ \autf^{(P)}(S) = \bigl\{ \alpha\in\autf(S) \,\big|\, 
	\alpha(P)=P,~ \alpha|_P\in O^{p'}(\autf(P)) \bigr\}\,. \]

\end{enuma}
\end{Prop}

\begin{proof}  Point (a) is shown in \cite[Proposition I.4.5]{AKO}, and 
point (c) in \cite[Corollary I.7.5]{AKO}. Point (b) follows from the 
definition and Alperin's fusion theorem (Theorem \ref{AFT}). Point (d) is 
shown in \cite[Lemma 1.4]{indp1}. 
\end{proof}

In order to be able to identify which of the fusion systems we construct 
are realizable, it will be helpful to know that when realizable, they can 
be realized by finite simple groups.

\begin{Lem}[{\cite{DRV}, \cite[Lemma 1.5]{indp1}}] \label{red->simple}
Let $\calf$ be a reduced fusion system over a $p$-group $S$.  Assume, for 
each strongly $\calf$-closed subgroup $1\ne{}P\nsg{}S$, that $P$ is centric 
in $S$, is not elementary abelian, and does not factor as a product of two 
or more subgroups which are permuted transitively by $\autf(P)$. Under 
these conditions, if $\calf$ is realizable, then it is the fusion system of 
a finite simple group.
\end{Lem}

The next lemma gives a very simple, necessary condition for a $p$-group $S$ 
to have an abelian subgroup of index $p$.

\newcommand{\ordp}{\textup{ord}_p}

\begin{Lem} \label{l:|A|}
Let $p$ be any prime, and let $S$ be a non-abelian $p$-group which contains 
an abelian subgroup of index $p$. Then $|Z(S)|\cdot|[S,S]|=\frac1p|S|$.
\end{Lem}

\begin{proof} Let $A\nsg S$ be an abelian subgroup of index $p$, fix 
$x\in{}S\sminus A$, and let $\varphi\in\End(A)$ be the homomorphism 
$\varphi(a)=[a,x]$. Then $Z(S)=C_A(x)=\Ker(\varphi)$, and 
$[S,S]=[x,A]=\Im(\varphi)$.
\end{proof}

Lemma \ref{red->simple} motivates the next lemma: a list of all finite 
simple groups whose fusion systems are of the type we are studying.

\begin{Lem} \label{list-simp}
Fix an odd prime $p$. Let $G$ be a known finite simple group such that 
$S\in\sylp{G}$ is nonabelian, and contains a unique abelian subgroup $A$ of 
index $p$. Set $\calf=\calf_S(G)$, and assume $\calf$ is reduced. Then 
$\calf$ is isomorphic to the fusion system of one of the following simple 
groups:
\begin{enuma}  
\item $A_{pn}$, where $p\le{}n<2p$;
\item $\Sp_4(p)$;
\item $\PSL_n(q)$, where $p|(q-1)$ and $p\le{}n<2p$;
\item $P\varOmega_{2n}^+(q)$, where $p|(q-1)$ and $p\le{}n<2p$;
\item $\lie3D4(q)$ or $\lie2F4(q)$, where $p=3$ and $q$ is prime to $3$;
\item $E_n(q)$, where $p|(q-1)$, $p=5$ if $n=6,7$, and $p=7$ if $n=7,8$;
\item $E_8(q)$, where $p=5$ and $q\equiv\pm2$ (mod $5$); or 
\item $\Co_1$, where $p=5$.
\end{enuma}
\end{Lem}

\iffalse
Assume the classification of finite simple groups. Let 
$p$ be an odd prime, and let $\calf$ be a reduced fusion system over 
a finite $p$-group $S$ which contains a unique abelian subgroup $A$ of 
index $p$. If $\calf$ is realizable, then it is isomorphic to the fusion 
system of one of the following simple groups:
\fi

%%\item $\PSL_n(q)$ with $p\nmid q$ and $p\le{}n/\ordp(q)<2p$ 
%%\item $P\varOmega_{2n}^\varepsilon(q)$, where $2|\ordp(q)|2n$, 
%%$p\le2n/\ordp(q)<2p$, and $(-1)^{2n/\ordp(q)}=\varepsilon$.

\begin{proof}  Since $A$ is the unique abelian subgroup of index $p$, 
$|S|\ge p^4$. 

%%By Lemma \ref{red->simple}, $\calf$ is the fusion system of 
%%a finite simple group $G$.  

\smallskip

\noindent\textbf{Case 1: } If $G$ is an alternating group, then 
$G\cong{}A_n$ for some $p^2\le{}n<2p^2$.  Since $\calf=\calf_S(G)$ is 
reduced, it is also the fusion system of $A_m$, where $m=\sup\{k\in 
p\Z\,|\,k\le n\}$ (see \cite[16.5]{A-genfit}). So we can choose $n=m$ to be 
a multiple of $p$.

%%we can choose $n$ to be a multiple of $p$.  

\smallskip

\noindent\textbf{Case 2: } Assume that $G$ is of Lie type in 
characteristic $p$, and fix $S\in\sylp{G}$. We prove that the only 
possibility is $G=\PSp_4(p)$, by showing that in all other cases, either 
$S$ is abelian, or it has more than one abelian subgroup of index $p$, or 
it has none at all.

Assume first that $G$ is one of the following groups:
	\beqq \PSL_4(p)\cong \varOmega_6^+(p), \quad \PSU_4(p)\cong 
	\varOmega_6^-(p),\quad \PSp_6(p),\quad 
	G_2(p), \quad \lie2G2(3^k). 
	\label{G(p)-elim} \eeqq
We will show that $S$ has no abelian group of index $p$ in any of these 
cases. If $H<\PSL_4(p)$ is the stabilizer of a projective plane and 
a point in the plane, then $O_p(H)\cong p^{1+4}_+$. If $G=\PSU_4(p)$ or 
$\PSp_6(p)$, and $H<G$ is the stabilizer of an isotropic point or a point, 
respectively, then $O_p(H)\cong p^{1+4}_+$. Thus if $G$ is one of the groups 
$\PSL_4(p)$, $\PSU_4(p)$, or $\PSp_6(p)$, then $S$ contains an extraspecial 
subgroup of order $p^5$, and hence contains no abelian subgroup of index $p$. 

For $p\ge5$, $G_2(p)$ also contains an extraspecial $p$-group of order 
$p^5$ (see \cite[p.127]{wilsonbook}). If $G\cong G_2(3)$, then its 
parabolic subgroups have the form $(C_3^2\times 3^{1+2}_+)\rtimes 
\GL_2(3)$. If $S$ is contained in this group, then each abelian subgroup of 
$S$ must intersect the subgroup $(C_3^2\times 3^{1+2}_+)$ with index at 
least $3$, so that $\SL_2(3)$ must act trivially in particular on the 
$C_3^2$ factor, while it actually acts as on the natural module (see 
\cite[p.125]{wilsonbook}). So we eliminate these cases.

Now assume that $G\cong\lie2G2(3^k)$ for $k\ge3$. By the main theorem in 
\cite{ward}, $|S|=(p^k)^3=|Z(S)|\cdot|[S,S]|$. So by Lemma \ref{l:|A|}, $S$ 
does not contain an abelian subgroup of index $p$.

Thus $S$ has no abelian subgroup of index $p$ if $G$ is one of the groups 
in \eqref{G(p)-elim}, or any group which contains one of them. In this way, 
we can eliminate all larger classical groups, as well as 
$\lie3D4(p)>G_2(p)$, $E_n(p)>F_4(p)>\Spin_9(p)$ ($n=6,7,8$), and 
$\lie2E6(p)>F_4(p)$ (see, e.g., \cite[Chapter 4]{wilsonbook} for 
descriptions of these inclusions), and also the groups of the same type 
over larger fields of characteristic $p$. 

Since $\PSL_2(p^k)$ has abelian Sylow $p$-subgroups, it remains to consider 
the groups 
	\beqq \PSL_3(p^k), \quad \PSU_3(p^k), \quad
	\PSp_4(p^k)\cong\varOmega_5(p^k),
	\label{G(p)-small} \eeqq
for $k\ge1$. The Sylow $p$-subgroups of 
$\PSL_3(p)$ and $\PSU_3(p)$ are extraspecial of order $p^3$, hence have 
more than one abelian subgroup of index $p$, while those of $\PSp_4(p)$ 
do have a unique such subgroup. By \cite[Theorem 3.3.1.a]{GLS3}, if $G$ is 
one of the groups in \eqref{G(p)-small}, then $|S|$, 
$|Z(S)|$, and $|[S,S]|$ are all powers of $p^k$, and so if $k>1$, $S$ 
contains no abelian subgroup of index $p$ by Lemma \ref{l:|A|}.

\smallskip

\noindent\textbf{Case 3: } Assume that $G={}^r\gg(q)$ is a group of Lie type in 
characteristic different from $p$. By \cite[Lemma 6.9]{BMO2}, and since the 
Sylow $p$-subgroups of $G$ are non-abelian, $G$ has a $p$-fusion system 
isomorphic to that of one of the following groups:
\begin{enumi} 
\item $\PSL_n(q)$ for some $n\ge p$; or
\item $P\varOmega_{2n}^\gee(q)$, where $n\ge p$, $\gee=\pm1$, $q^n\equiv\gee$ 
(mod $p$), and $\gee=+1$ if $n$ is odd; or
\item $\lie3D4(q)$ or $\lie2F4(q)$, where $p=3$ and $q$ is a power of $2$; 
or
\item $G_2(q)$, $F_4(q)$, $E_6(q)$, $E_7(q)$, or $E_8(q)$ where $p|(q-1)$; 
or
\item $E_8(q)$ where $p=5$ and $q\equiv\pm2$ (mod $5$). 
\end{enumi}

Assume that $G=\PSL_n(q)$. If $e=\ord_p(q)>1$, then by \cite[Theorem 
B]{Ruiz}, the $p$-fusion system of $G$ has a proper normal subsystem of 
index $e$, and hence is not reduced. Thus $p|(q-1)$, and $p\le n<2p$ since 
the Sylow $p$-subgroups have abelian subgroups of index $p$. We are thus in 
the situation of (c).

Assume that $G=P\varOmega_{2n}^\gee(q)$ is as in (ii). Set $e=\ord_p(q)$. 
If $e$ is even, then by \cite[Proposition A.3]{BMO1}, the $p$-fusion system of 
$\GO_{2n}^\gee(q)$ is isomorphic to that of $\SL_{2n}(q)$. So $\calf$
is normal of  index $2$ in the fusion system of 
$\SL_{2n}(q)$; this contains a normal subsystem of index $e$ by 
\cite[Theorem B]{Ruiz} again, and hence $\calf$ has a normal subsystem of index 
$e/2$. So $\calf$ is reduced only if $e=2$. If $e$ is odd, then by 
\cite[Theorem A(a,b)]{BMO1}, we can assume that $q$ is a square, 
$\GO_{2n}^\gee(q)$ and $\SL_{2n}(\sqrt{q})$ have isomorphic $p$-fusion systems 
by  \cite[Proposition A.3]{BMO1} again, and so $\calf$ is reduced only if 
$e=1$.

We can thus assume that $q\equiv\pm1$ (mod $p$), and $p|(q-1)$ if 
$n$ is odd (since $\gee=+1$). If $n$ is even and $q\equiv-1$ (mod $p$), 
then by \cite[Theorem A(b)]{BMO1}, $G$ has the same $p$-fusion system as 
$P\varOmega_{2n}^\gee(q^*)$ for some $q^*\equiv1$ (mod $p$). So we can 
assume that $p|(q-1)$ in all cases, and are in the situation of (d). 

Cases (iii) and (v) correspond to (e) and (g), respectively. In case (iv), 
if $G=\gg(q)$ and $p|(q-1)$, then the order of the Weyl group of $\gg$ must 
be a multiple of $p$ but not of $p^2$, and so $(\gg,p)$ is one of the pairs 
$(G_2,3)$, $(E_6,5)$, $(E_7,5)$, $(E_7,7)$, or $(E_8,7)$. The $3$-fusion 
system of $G_2(q)$  is not reduced, since it is the fusion system of 
$\PSL_3(q):2$ or $\PSU_3(q):2$ \cite[(16.11)]{A-genfit}. So we are in the 
situation of (f). 

\smallskip

\noindent\textbf{Case 4: } If $G$ is a sporadic group, then by the tables 
in \cite[\S\,1.5]{GL} or \cite[\S\,5.3]{GLS3}, in almost all cases, either 
$|S|\le p^3$, or $S$ is abelian, or $S$ contains an extraspecial group of 
type $p^{1+2k}$ for $k\ge2$, or $S$ contains a special group of type 
$3^{2+4}$.  The exceptions are $(G,p)=(J_3,3)$, $(\Co_1,5)$, and $(\Th,3)$.  
When $S\in\syl3{J_3}$, $|S|=3^5$, $Z(S)\cong C_3^2$ and $[S,S]\cong C_3^3$ 
(see \cite[\S\,3]{FR}), so by Lemma \ref{l:|A|}, $S$ does not contain an 
abelian subgroup of index $3$. Since $\Th$ contains a subgroup isomorphic to 
$G_2(3)$ \cite[(3.12)]{Parrott}, whose Sylow 3-subgroups were already shown 
not to have abelian 
subgroups of index $3$, the same holds for $\Th$. Thus $p=5$ and $G=\Co_1$.  
\end{proof}

We finish the section with some miscellaneous group-theoretic results that 
will be needed later.

\begin{Lem} \label{mod-Fr}
Fix a prime $p$, a finite $p$-group $P$, and a group $G\le\Aut(P)$ of 
automorphisms of $P$.  Let $\Fr(P)=P_0\nsg{}P_1\nsg\cdots\nsg{}P_m=P$ be a 
sequence of subgroups, all normal in $P$ and normalized by $G$.  Let 
$H\le{}G$ be the subgroup of those $g\in{}G$ which act via the identity on 
$P_i/P_{i-1}$ for each $1\le i\le m$.  Then $H$ is a normal $p$-subgroup of 
$G$, and hence $H\le O_p(G)$.
\end{Lem}

\begin{proof}  See, e.g., \cite[Theorems 5.3.2 \& 5.1.4]{Gorenstein}.  
\end{proof}

\begin{Lem} \label{l0:xx}
Fix a finite abelian $p$-group $A$ and a subgroup $G\le\Aut(A)$, and choose 
$\UUU\in\sylp{G}$.  Then
	\[ C_A(\UUU)\le[G,A] \ \Longleftrightarrow\  C_A(G)\le[G,A] 
	\ \Longleftrightarrow\ C_A(G)\le[\UUU,A]\,. \]
\end{Lem}

\begin{proof}  Each of the first and third inequalities clearly implies 
the second, so it suffices to show that the second implies each of the 
other two.  Assume $C_A(G)\le[G,A]$.

If $C_A(\UUU)\nleq[G,A]$, then choose $z\in{}C_A(\UUU)\sminus[G,A]$.  Let $X$ 
be the $G$-orbit of $z$ and set $m=|X|$; $p{\nmid}m$ since $z$ is fixed by 
$\UUU$.  Let $\widehat{z}$ be the product of the elements in $X$.  Then 
$\widehat{z}\in{}C_A(G)$, and $\widehat{z}\in{}z^m{\cdot}[G,A]$.  Thus 
$\widehat{z}\in{}C_A(G)\sminus[G,A]$, contradicting our assumption.

Now assume $C_A(G)\nleq[\UUU,A]$.  Since $\Z/p^\infty$ is injective as an 
abelian group, there is a homomorphism $\varphi\in\Hom(A,\Z/p^\infty)$ such 
that $[\UUU,A]\le\Ker(\varphi)$ but $C_A(G)\nleq\Ker(\varphi)$.  
Let $X$ be the $G$-orbit of $\varphi$ under the action of $G$ on 
$\Hom(A,\Z/p^\infty)$, and set $m=|X|$.  Then $p{\nmid}m$ since $\varphi$ 
is fixed by $\UUU$.  Let $\widehat{\varphi}$ be the product 
of the elements of $X$.  Then $\widehat{\varphi}$ is $G$-invariant, so 
$[G,A]\le\Ker(\widehat{\varphi})$.  Fix $z\in{}C_A(G)$ such that 
$\varphi(z)\ne0$; then $\widehat{\varphi}(z)=m{\cdot}\varphi(z)\ne0$.  
Thus $z\in{}C_A(G)\sminus[G,A]$, again contradicting our assumption.
\end{proof}

\begin{Lem}[{\cite[Lemma 1.11]{indp1}}] \label{GonA}
Fix a finite abelian $p$-group $A$ and a subgroup $G\le\Aut(A)$.  Assume the 
following.
\begin{enumi}  
\item Each Sylow $p$-subgroup of $G$ has order $p$ and is not normal in 
$G$.
\item For each $x\in{}G$ of order $p$, $[x,A]$ has order $p$, and hence 
$C_A(x)$ has index $p$.  
\end{enumi}
Set $A_1=C_A(O^{p'}(G))$ and $A_2=[O^{p'}(G),A]$. Then $G$ normalizes $A_1$ 
and $A_2$, $A=A_1\times A_2$, and $O^{p'}(G)\cong\SL_2(p)$ acts faithfully 
on $A_2\cong C_p^2$.
\end{Lem}

\begin{Lem} \label{l:A/Z-filter}
Let $V$ be a finite abelian $p$-group (written additively), and fix a subgroup 
$G\le\Aut(V)$ with $\UUU\in\sylp{G}$ of order $p$. Assume also that 
$\dim\bigl(C_V(\UUU)\cap[\UUU,V]\bigr)=1$. Set $V_0=C_V(\UUU)[\UUU,V]$. Define 
inductively $W_1>W_2>\cdots$ by setting $W_1=[\UUU,V]$, and 
$W_{n+1}=[\UUU,W_n]$ for $n\ge1$. Let $m$ be the smallest integer 
such that $W_m=0$. Then the following hold.
\begin{enuma} 
\item For each $1\le i\le m-1$, $|W_i/W_{i+1}|=p=|V/V_0|$.
\item Fix $g\in N_G(\UUU)$. Let $r,t\in(\Z/p)^\times$ be such that 
$\9gu=u^r$ for all $u\in\UUU$, and $g$ induces multiplication by 
$t$ on $V/V_0$. Then for each $1\le i\le m-1$, $g$ induces multiplication 
by $tr^i$ on $W_i/W_{i+1}$.

\end{enuma}
\end{Lem}

\begin{proof} \textbf{(a) } Fix a generator $u\in\UUU$, and let 
$\varphi\:V\Right2{}V$ be the homomorphism $\varphi(v)=[u,v]$.  For each 
$0\ne{}W\le V$ normalized by $\UUU$, $\Ker(\varphi|_W)=C_W(\UUU)$, 
$\Im(\varphi|_W)=[\UUU,W]$, and thus $|W/[\UUU,W]|=|C_W(\UUU)|>1$.  In 
particular, if $0\ne W\le W_1=[\UUU,V]$, then $0\ne{}C_W(\UUU)\le 
C_{W_1}(\UUU)$, with equality since $C_{W_1}(\UUU)=C_V(\UUU)\cap[\UUU,V]$ 
has order $p$ by assumption. Hence $[\UUU,W]$ has index $p$ in $W$.

Thus $|W_i/W_{i+1}|=p$ for each $1\le i<m$. Also, since 
$|C_V(\UUU)|\cdot|[\UUU,V]|=|V|$, $|V/V_0|=|C_V(\UUU)\cap[\UUU,V]|=p$.

\smallskip

\noindent\textbf{(b) } Fix $g\in N_G(\UUU)$, and let $r,t\in(\Z/p)^\times$ 
be as above. For each $x=[u,v]\in W_1$, $g(x)=[\9gu,g(v)]\equiv[u^r,tv]\equiv 
rt[u,v]$ modulo $[\UUU,V_0]=[\UUU,W_1]=W_2$. This proves the result when 
$i=1$, and the other cases follow inductively.
\end{proof}

%%\newpage

\section{Reduced fusion systems over non-abelian $p$-groups with index $p$ 
abelian subgroup}
\label{s:s/a}

Throughout this section, $p$ is an odd prime.  We want to describe all 
reduced fusion systems over non-abelian $p$-groups which contain an abelian 
subgroup of index $p$. If $S$ has more than one abelian subgroup of index 
$p$, then by \cite[Theorem 2.1]{indp1}, either $S$ is extraspecial of order 
$p^3$ and exponent $p$ (the case already handled by Ruiz and Viruel in 
\cite{RV}), or there are no reduced fusion systems over $S$. In 
\cite[Theorem 2.8]{indp1}, the second author handled the case where $S$ 
contains a unique abelian subgroup of index $p$ and that subgroup is not 
essential. We now look at the more complicated case: that where $S$ 
contains a unique abelian subgroup of index $p$ and it is essential 
(Theorem \ref{t3:s/a} below).

\begin{Not} \label{n:not1}
Fix a $p$-group $S$ with unique abelian subgroup $A$ of index $p$, and a 
saturated fusion system $\calf$ over $S$.  
Define 
	\[ S'=[S,S]\,,\quad Z=Z(S)\,,\quad Z_0=Z\cap S'\,,\quad Z_2=Z_2(S)\,,
	\quad A_0=Z{\cdot}S' \,. \]
Thus $Z_0\le Z\le A_0\le A$ and $Z_0\le S'\le A_0$.  Also, set
	\[ \calh = \bigl\{ Z\gen{x} \,\big|\, x\in S\sminus A \bigr\} 
	\qquad\textup{and}\qquad
%%	\quad\textup{(when $Z_2\le A_0$)}\quad
	\calb = \bigl\{ Z_2\gen{x} \,\big|\, x\in S\sminus A \bigr\}\,. \]
\end{Not}

\begin{Lem} \label{l1:s/a}
Assume the notation and hypotheses of \ref{n:not1}. Then 
the following hold. 
\begin{enuma} 

\item  For each $P\in\EE\calf$, either $P=A$, or $P$ is abelian and 
$P\in\calh$, or $P$ is non-abelian and $P\in\calb$.  In all cases, 
$|N_S(P)/P|=p$.

\item If $Z_2\gen{x}\in\EE\calf$ for some $x\in{}S\sminus A$, then 
$Z\gen{x}$ is not $\calf$-centric and $Z\gen{x}\notin\EE\calf$.

\item If $A\nnsg\calf$ (equivalently, if $\EE\calf\not\subseteq\{A\}$), then 
$|Z_0|=p$, $Z_2\le A$, and $S/Z$ is nonabelian.

%%, and hence the conclusions of (c) all hold. 

\item If $|Z_0|=p$, then $|A/A_0|=p$, $|Z_2/Z|=p$, $Z_2\le A_0$, and $Z_2\cap 
S'\cong C_p^2$. Also, there are elements $\xx\in{}S\sminus A$ and 
$\xa\in{}A\sminus A_0$ such that $A_0\gen{\xx}$ and $S'\gen{\xa}$ are 
normalized by $\autf(S)$.  If some element of $S\sminus A$ has order $p$, 
then we can choose $\xx$ to have order $p$.

\item For each $P\in\EE\calf$ and each $\alpha\in N_{\autf(P)}(\Aut_S(P))$, 
$\alpha$ extends to some $\4\alpha\in\autf(S)$.

\item For each $x\in{}S\sminus A$ and each $g\in{}A_0$, $Z\gen{x}$ is 
$S$-conjugate to $Z\gen{gx}$, and $Z_2\gen{x}$ is $S$-conjugate to 
$Z_2\gen{gx}$.
\end{enuma}
\end{Lem}

\begin{proof} \textbf{(a,b,c,e) } See points (a), (c), (b), and 
(e), respectively, in \cite[Lemma 2.3]{indp1}. Note, in (c), that 
$S/Z$ is nonabelian since $Z(S/Z)=Z_2/Z\le A/Z$.

\smallskip

\textbf{(d) } Assume that $|Z_0|=p$. Fix a generator $\alpha\in\Aut_S(A)\cong 
C_p$. Let $f\:A\Right2{}A$ be the homomorphism $f(x)=x^{-1}\alpha(x)$. Then 
$\Ker(f)=C_A(\alpha)=Z$ and $\Im(f)=[\alpha,A]=S'$, so $|Z|\cdot|S'|=|A|$, 
and $|A/A_0|=|A/ZS'|=|Z\cap{}S'|=p$. Also, $|S'|>p$ since $A$ is the unique 
abelian subgroup of index $p$ in $S$, so $S/Z$ is non-abelian, and 
$Z_2/Z=Z(S/Z)=C_{A/Z}(\alpha)$. Let $\4f\:A/Z\Right2{}A/Z$ be the 
homomorphism induced by $f$ on the quotient; then 
$|Z_2/Z|=|\Ker(\4f)|=|(A/Z)/\Im(\4f)|=|A/ZS'|=p$. 

Since $Z_2/Z=Z(S/Z)$ has order $p$ and $S/Z$ is nonabelian, $Z_2/Z$ must 
be contained in $[S/Z,S/Z]=S'Z/Z=A_0/Z$. Thus $Z_2\le A_0$, so $Z_2S'=ZS'$, 
and hence $|Z_2\cap{}S'|=|Z_2/Z|\cdot|Z\cap{}S'|=p^2$. It remains to show 
that $Z_2\cap{}S'$ is not cyclic.

For each $x\in{}S'=[\alpha,A]$, $x=y^{-1}\alpha(y)$ for some $y\in{}A$, so 
$\prod_{i=0}^{p-1}\alpha^i(x)=1$. Hence if $Z_2\cap{}S'\cong C_{p^2}$ is 
generated by $x$, then $\alpha(x)=x^{1+kp}$ for some $k$ such that $p\nmid 
k$, and $\sum_{i=0}^{p-1}(1+kp)^i=\bigl((1+kp)^p-1\bigr)\big/kp\equiv0$ 
(mod $p^2$). Since $(1+kp)^p\equiv1+kp^2$ (mod $p^3$), this is impossible.

%%\smallskip
%%
%%\noindent\textbf{(d) } If $A\nnsg\calf$, then $|Z_0|=p$ by \cite[Lemma 
%%2.3(b)]{indp1}.

Let $B\le A$ be minimal among subgroups which are 
normalized by $\autf(S)$ such that $B\ge S'$ and $BA_0=A$. The natural 
surjection $B/S'\Fr(B)\Onto2{}A/A_0$ of $\F_p[\outf(S)]$-modules is split 
since $p\nmid|\outf(S)|$. Hence $B/S'\Fr(B)\cong A/A_0$ by the minimality 
of $B$, and $B/S'$ is cyclic since $B/S'\Fr(B)$ is cyclic. For any 
generator $\xa S'$ of $B/S'$, $\xa\in{}A\sminus A_0$, and $B=S'\gen{\xa}$ 
is normalized by $\autf(S)$.  

%%By \cite[Lemma 2.3(d)]{indp1}, there is $x\in{}S\sminus A$ such that 
%%$A_0\gen{x}$ is normalized by $\autf(S)$. 

For each $x\in S\sminus A$, $x^p\in C_A(x)=Z\le A_0$. Hence $S/A_0\cong 
C_p^2$. Since $A/A_0$ is an $\F_p[\outf(S)]$-submodule of $S/A_0$ (and since 
$p\nmid|\outf(S)|$), $S/A_0$ splits as a product 
$S/A_0=(A/A_0)\times(R/A_0)$ where $R$ is $\autf(S)$-invariant. Let 
$\xx$ be any element of $R\sminus A\subseteq S\sminus A$; then 
$R=A_0\gen{\xx}$.

Assume that there is $y\in{}S\sminus A$ such that $y^p=1$. Upon 
replacing $y$ by some other generator of $\gen{y}$, we can assume that 
$y\in\xx A$. If $A_0\gen{y}$ is 
normalized by $\autf(S)$, we are done.  Otherwise, there is 
$y'\in{}yA$ such that $y'{}^p=1$ and $y^{-1}y'\notin{}A_0$.  Set 
$g=y^{-1}y'$; then 
	\[ 1 = y'{}^p = (yg)^p = (ygy^{-1})(y^2gy^{-2})\cdots(y^pgy^{-p}) 
	y^p \,, \]
so $\prod_{i=1}^p(y^igy^{-i})=1$.  Then 
$(yg^j)^p=\prod_{i=1}^p(y^ig^jy^{-i})y^p=1$ for all $j\in\Z$ by 
a similar computation.  Since $g\in{}A\sminus A_0$, there is $j$ such that 
$yg^j\in\xx A_0$. Upon replacing $\xx$ by $yg^j$, we can arrange that 
$\xx^p=1$.

\smallskip

\noindent\textbf{(f) } Fix $x\in{}S\sminus A$ and $g\in{}A_0=S'Z$, and 
choose $z\in{}Z$ and $g'\in{}S'$ such that $g=g'z$.  Then there is 
$h\in{}A$ such that $g'=h^{-1}(\9xh)$, $gx=zh^{-1}xh$, and so 
$Z\gen{gx}=Z\gen{x^h}=Z\gen{x}^h$ and 
$Z_2\gen{gx}=Z_2\gen{x^h}=Z_2\gen{x}^h$.
\end{proof}

\begin{Lem} \label{l2:s/a}
Let $A\nsg{}S$, $\calf$, $\calh$, $\calb$, etc., be as in Notation 
\ref{n:not1}.  Assume that $A\nnsg\calf$.
\begin{enuma} 
\item If $P\in\calb\cap\EE\calf$, then $P=ZP^*$ for some unique 
$\autf(P)$-invariant extraspecial subgroup $P^*$ of order $p^3$ and 
exponent $p$, with centre $Z(P^*)=Z_0=Z\cap{}P^*$.  Also, 
$O^{p'}(\outf(P))\cong\SL_2(p)$, and this group 
acts faithfully on $P^*/Z_0\cong P/Z\cong C_p^2$ and acts trivially on $Z$.  

\item If $P\in\calh\cap\EE\calf$, then there is a unique subgroup $Z^*<Z$ 
which is normalized by $\autf(Z)$ and such that $Z=Z_0\times Z^*$. Also, 
$Z^*$ is normalized by $\autf(P)$, $P=Z^*\times{}P^*$ for some unique 
$\autf(P)$-invariant subgroup $P^*\cong{}C_p^2$ which contains $Z_0$, 
$O^{p'}(\autf(P))\cong\SL_2(p)$, and this group acts faithfully on 
$P^*$ and acts trivially on $Z^*$. 

\item There is $x\in S\sminus A$ of order $p$; i.e., $S$ splits over $A$.
\end{enuma}
\end{Lem}

\begin{proof} \textbf{(a) } Assume that $P\in\calb\cap\EE\calf$.   Then 
$|P/Z|=|P/Z_2|\cdot|Z_2/Z|=p^2$ (Lemma \ref{l1:s/a}(c,d)), so $P/Z$ is 
abelian, and $[P,P]\le Z\cap S'=Z_0$. Thus $[P,P]=Z_0$ since $|Z_0|=p$ and 
$P$ is non-abelian.

By Lemma \ref{l1:s/a}(a), $\Out_S(P)\in\sylp{\outf(P)}$ has order $p$. 
Also, $[N_S(P),P]\le{}P\cap S'=Z_2\cap S'$, and $Z_2\cap S'\cong C_p^2$ by 
Lemma \ref{l1:s/a}(c,d). Thus $[\Aut_S(P),P/Z_0]\le(Z_2\cap S')/Z_0\cong 
C_p$, with equality since $P$ is essential and hence 
$\Aut_S(P)\in\sylp{\autf(P)}$ cannot act trivially on $P/[P,P]$ 
(recall $[P,P]\le Z_0$).

By Lemma \ref{GonA} applied to the $\outf(P)$-action on $P/Z_0$, 
$P=P_1P_2$, where $\autf(P)$ normalizes $P_1$ and $P_2$, 
$P_1/Z_0=C_{P/Z_0}(O^{p'}(\outf(P)))$, $P_1\cap P_2=Z_0$, and 
$|P_2/Z_0|=p^2$. Thus 
$P_1/Z_0$ is the intersection of the subgroups in the $\autf(P)$-orbit of 
$C_{P/Z_0}(\Aut_S(P))=Z_2/Z_0$, hence contains $Z(P)/Z_0=Z/Z_0$, with 
equality since $|P/Z|=p^2=|P_2/Z_0|$. Thus $P_1=Z$, and we set $P^*=P_2$. 
Since $P^*/Z_0\cong C_p^2$ by Lemma \ref{GonA} again, $P^*$ is extraspecial 
of order $p^3$, and has exponent $p$ since otherwise its automorphism group 
would be a $p$-group. The last statement follows immediately from Lemma 
\ref{GonA}.

\smallskip

\noindent\textbf{(b) }  Assume that $P\in\calh\cap\EE\calf$. Thus $P$ is 
abelian. By Lemma \ref{l1:s/a}(a), $\Aut_S(P)\in\sylp{\autf(P)}$ has order 
$p$.  Also, for each $g\in{}N_S(P)\sminus P$, $1\ne[g,P]\le{}P\cap S'=Z_0$, 
so $[g,P]=Z_0$ since $|Z_0|=p$ by Lemma \ref{l1:s/a}(c). Hence by Lemma \ref{GonA}, 
$P=Z^*\times{}P^*$, where $Z^*\le Z$ and $P^*\ge Z(P)\cap S'=Z_0$ are both 
$\autf(P)$-invariant. Also, by the same lemma, 
$O^{p'}(\autf(P))\cong\SL_2(p)$, and this subgroup acts faithfully on 
$P^*\cong C_p^2$ and trivially on $Z^*$.  

In particular, there is a subgroup $H\le{}N_{O^{p'}(\autf(P))}(\Aut_S(P))$ of 
order $p-1$ which acts as the full group of automorphisms of $Z_0$ and of 
$P^*/Z_0$, and acts trivially on $Z^*$.  Since $H$ restricts to a subgroup 
of $\autf(Z)$, this shows that $Z=Z^*\times{}Z_0$ is the unique 
$\autf(Z)$-invariant splitting of $Z$ with one factor $Z_0$. 

\smallskip

\noindent\textbf{(c) } Since $A\nnsg\calf$, there must be $P\in\EE\calf$ in 
$\calh\cup\calb$. So there is $x\in S\sminus A$ of order $p$ by 
the descriptions of $P$ in (a) and (b). 
\end{proof}

We now need to fix some more notation.

\begin{Not} \label{n:not2}
Assume Notation \ref{n:not1}. Assume also that $|Z_0|=p$, and hence that 
$|A/A_0|=p$.  Fix $\xa\in{}A\sminus A_0$ and $\xx\in{}S\sminus A$, chosen 
such that $A_0\gen{\xx}$ and $S'\gen{\xa}$ are each normalized by 
$\autf(S)$, and such that $\xx^p=1$ if any element of $S\sminus A$ has 
order $p$ (Lemma \ref{l1:s/a}(d)).  For each $i=0,1,\dots,p-1$, define 
	\[ H_i = Z\gen{\xx\xa^i}\in\calh
	\qquad\textup{and}\qquad 
	B_i = Z_2\gen{\xx\xa^i}\in\calb \,. \]
Let $\calh_i$ and $\calb_i$ denote the $S$-conjugacy classes of $H_i$ 
and $B_i$, respectively, and set
	\[ \calh_*=\calh_1\cup\cdots\cup\calh_{p-1} \quad\textup{and}\quad
	\calb_*=\calb_1\cup\cdots\cup\calb_{p-1}. \]
Thus $\calh=\calh_0\cup\calh_*$ and $\calb=\calb_0\cup\calb_*$ by Lemma 
\ref{l1:s/a}(f) and since $|A/A_0|=p$.

Set 
	\[ \Delta = (\Z/p)^\times \times (\Z/p)^\times\,,
	\qquad\textup{and}\qquad
	\Delta_i=\{(r,r^i)\,|\,r\in(\Z/p)^\times\}\le\Delta
	\quad\textup{(for $i\in\Z$).} \]
Define 
	\[ \mu\:\Aut(S)\Right5{}\Delta \qquad\textup{and}\qquad
	\5\mu\:\Out(S)\Right5{}\Delta \]
by setting, for $\alpha\in\Aut(S)$, 
	\[ \mu(\alpha) = \5\mu([\alpha])=(r,s) \quad\textup{if}\quad
	\begin{cases} 
	\alpha(x)\in x^rA & \textup{for $x\in{}S\sminus A$} \\
	\alpha(g)=g^s & \textup{for $g\in{}Z_0$\,.}
	\end{cases} \]
Finally, set
	\begin{align*} 
	\autff(S)&=\bigl\{\alpha\in\autf(S) \,\big|\,
	[\alpha,Z]\le Z_0\bigr\}, \\
	\outff(S)&=\autff(S)/\Inn(S), \\
	\autff(A)&=\bigl\{\alpha|_A \,\big|\, \alpha\in\autff(S) \bigr\}
	= \bigl\{\beta\in N_{\autf(A)}(\Aut_S(A)) \,\big|\, [\beta,Z]\le 
	Z_0 \bigr\} , \\
	\autf^{(P)}(S) &= \bigl\{ \alpha\in\autf(S) \,\big|\,
	\alpha(P)=P,~ \alpha|_P\in{}O^{p'}(\autf(P)) \bigr\} 
	\qquad\textup{(all $P\le S$).}
	\end{align*}
\end{Not}

\begin{Lem} \label{l3:s/a}
Let $S$ be a finite $p$-group with a unique abelian subgroup $A\nsg{}S$ of 
index $p$, and let $\calf$ be a saturated fusion system over $S$. Assume that 
$|Z_0|=p$, and use Notation \ref{n:not1} and \ref{n:not2}. Let $m\ge3$ be 
such that $|A/Z|=p^{m-1}$.  Then the following hold.
\begin{enuma}
\item $\5\mu|_{\outff(S)}$ is injective.

\item Fix $\alpha\in\Aut(S)$, set $(r,s)=\mu(\alpha)$, and let $t$ be such 
that $\alpha(g)\in{}g^tA_0$ for each $g\in{}A\sminus A_0$.  Then $s\equiv 
tr^{m-1}$ (mod $p$).

\item  For $\alpha\in\autf(S)$, either $\mu(\alpha)\in\Delta_m$, and 
$\alpha$ normalizes each of the $S$-conjugacy classes $\calh_i$ and 
$\calb_i$ ($0\le i\le p-1$); or $\mu(\alpha)\notin\Delta_m$, and $\alpha$ 
normalizes only the classes $\calh_0$ and $\calb_0$. Also, $\alpha$ acts 
via the identity on $A/A_0$ if and only if 
$\mu(\alpha)\in\Delta_{m-1}$.

\item Assume that $\xx^p=1$, and set 
$\sigma=\prod_{i=0}^{p-1}\9{\xx^i}(\xa)=(\xa\xx)^p\xx^{-p}$. For 
each $P\in\calh_0\cup\calb_0$, $P$ splits over $P\cap{}A$. For 
each $P\in\calh_*\cup\calb_*$, $P$ splits over $P\cap{}A$ if and only if 
$\sigma\in\Fr(Z)$. 

%%\item $\Ker(\5\mu)\cap\outff(S)=1$.
\end{enuma}
\end{Lem}

\begin{proof} \textbf{(b) } This follows from Lemma \ref{l:A/Z-filter}(b), 
applied with $A$, $A_0$, $S'$, and $Z_0$ in the role of $V$, $V_0$, $W_1$, 
and $W_{m-1}$. Note that $p^{m-1}=|A/Z|=|S'|$, $|W_1|=p^{m-1}$ by Lemma 
\ref{l:A/Z-filter}(a), and thus $m$ plays the same role here as in Lemma 
\ref{l:A/Z-filter}.

\smallskip

\noindent\textbf{(a) }  Assume that $\alpha\in\autff(S)$ has order prime to $p$, 
and $\mu(\alpha)=1$. Then $\alpha$ induces the identity on $S/A$, on 
$A/A_0$ by (b), and on $A_0/S'$ since $[\alpha,ZS']\le Z_0S'=S'$ by 
definition of $\autff(S)$. So $\alpha=\Id_S$ by Lemma \ref{mod-Fr}. Thus 
$\autff(S)\cap\Ker(\mu)$ is a $p$-group, and hence equal to $\Inn(S)$.

%%, on $A/A_0$ by (b), and also on $ZS'/S'=A_0/S'$ since 
%%$\alpha(S')=S'$. 

\smallskip

\noindent\textbf{(c) }  Fix $\alpha\in\autf(S)$, and let $r,s,t$ be as in 
(b). Then $\alpha$ acts via the identity on $A/A_0$ if and only if $t=1$; 
equivalently (by (b)) if and only if $\5\mu([\alpha])=(r,s)\in\Delta_{m-1}$.  

Since $\alpha(A_0\gen{\xx})=A_0\gen{\xx}$ by assumption (see Notation 
\ref{n:not2}), $\alpha(\xx)\in{}A_0\xx^r$. Hence for each $0\le 
i<p$, $\alpha(A_0\gen{\xx\xa^i})=A_0\gen{\xx^r\xa^{it}}$. Thus 
$\alpha(\calh_0)=\calh_0$ and $\alpha(\calb_0)=\calb_0$ in all cases, while 
for $0<j<p$, $\alpha(\calh_j)=\calh_j$ or $\alpha(\calb_j)=\calb_j$ only if 
$r\equiv t$ (mod $p-1$). By (b), this holds if and only if $s\equiv r^m$ 
(mod $p-1$); i.e., if and only if $\mu(\alpha)\in\Delta_m$. 

\smallskip

\noindent\textbf{(d) } Since $\xx^p=1$, $P$ splits over $P\cap A$ for all 
$P\in\calh_0\cup\calb_0$. For each $1\le i\le p-1$, 
$(\xa^i\xx)^p=\sigma^i$, so there is $z\in{}Z$ with $\xa^i\xx z$ of order 
$p$ exactly when $\sigma\in\Fr(Z)$. This proves the claim for 
$P\in\calh_*$. 

Assume that $P\in\calb_i$ for $1\le i\le p-1$. If $P$ splits over $P\cap 
A$, then $Z_2\gen{\xa^i\xx}$ splits (it is $S$-conjugate to $P$), and hence 
$(\xa^i\xx y)^p=1$ for some $y\in Z_2$. Also, 
$[\xa^i\xx,y]=[\xx,y]\in[S,Z_2]=Z_0$, so $1=(\xa^i\xx 
y)^p=(\xa^i\xx)^py^p=\sigma^iy^p$, and hence $\sigma\in\Fr(Z_2)$. Also, 
since $Z_2=Z(Z_2\cap S')$, where $Z_2\cap S'\cong C_p^2$ by Lemma 
\ref{l1:s/a}(d), we have $Z_2\cong Z\times C_p$, and hence 
$\Fr(Z_2)=\Fr(Z)$. The converse is clear: if $\sigma=z^p$ for some $z\in 
Z$, then $Z_2\gen{\xa^i\xx}$ is split over $Z_2$ by $\gen{\xa^i\xx 
z^{-i}}$. 
\end{proof}

\begin{Lem} \label{l4:s/a}
Let $S$ be a finite $p$-group with a unique abelian subgroup $A\nsg{}S$ of 
index $p$, and let $\calf$ be a saturated fusion system over $S$. We use 
Notation \ref{n:not1}, and let $m$ be such 
that $|A/Z|=p^{m-1}$. Fix $P\in\calh\cup\calb$. Set $t=-1$ if $P\in\calh$ or set $t=0$ if $P\in\calb$. 
\begin{enuma}

\item If $P\in\EE\calf$, then in the notation of \ref{n:not2}, 
$\autf^{(P)}(S)\le\autff(S)$ and $\mu(\autf^{(P)}(S))=\Delta_t$.  If in 
addition, $P\in\calh_*\cup\calb_*$, then $m\equiv t$ (mod $p-1$).  

\item Conversely, assume that $\mu\bigl(N_{\autff(S)}(P)\bigr)\ge\Delta_t$, 
and also that $P$ splits over $P\cap{}A$. 
Then there is a unique subgroup $\Theta\le\Aut(P)$ such that \smallskip
\begin{enumi} 
\item $\Aut_S(P)\in\sylp\Theta$,
\item $\Theta\ge\Inn(P)$ and 
$O^{p'}(\Theta)/\Inn(P)\cong\SL_2(p)$, 

\item $[\alpha,Z]\le Z_0$ for each $\alpha\in 
N_{O^{p'}(\Theta)}(\Aut_S(P))$, and 

\item $N_\Theta(\Aut_S(P))= \bigl\{\alpha|_P\,\big|\,\alpha\in 
N_{\autf(S)}(P)\bigr\}$. 

\end{enumi}

\end{enuma}
\end{Lem}

\begin{proof} \textbf{(a) } Assume that $P\in\EE\calf$. By Lemma 
\ref{l2:s/a}(a,b), whether $P\in\calh$ or $P\in\calb$, 
$O^{p'}(\outf(P))\cong\SL_2(p)$, and acts trivially on $Z/Z_0$. Thus 
$\autf^{(P)}(S)\le\autff(S)$. 

If $P\in\calh$, then by Lemma \ref{l2:s/a}(b), there is $P^*\le P$ such 
that $P^*\cap{}A=Z_0$, and $O^{p'}(\autf(P))\cong\SL_2(p)$ acts faithfully 
on $P^*\cong C_p^2$. Each element of the normalizer 
$N_{O^{p'}(\autf(P))}(\Aut_S(P))\cong C_p\sd{}C_{p-1}$ extends to an 
element of $\autf^{(P)}(S)$ by Lemma \ref{l1:s/a}(e); and since this 
normalizer contains all diagonal matrices $\mxtwo{u}00{u^{-1}}$ for 
$u\in(\Z/p)^\times$, $\mu(\autf^{(P)}(S))$ is the set
$\{(u,u^{-1})\,|\,u\in(\Z/p)^\times\} =\Delta_{-1}$. 

If $P\in\calb$, then by Lemma \ref{l2:s/a}(a), there is $P^*\le P$ such 
that $Z_0\le P^*\cap{}A\le Z_2$, $P^*$ is extraspecial of order $p^3$ and 
exponent $p$, and $O^{p'}(\outf(P))\cong\SL_2(p)$ acts faithfully on 
$P^*/Z_0\cong C_p^2$. Each element of $N_{O^{p'}(\outf(P))}(\Out_S(P))\cong 
C_p\sd{}C_{p-1}$ extends to an element of $\autf^{(P)}(S)$ by Lemma 
\ref{l1:s/a}(e); this normalizer acts trivially on $Z_0$, and hence 
$\mu(\autf^{(P)}(S))=\{(u,1)\,|\,u\in(\Z/p)^\times\}=\Delta_{0}$.

If $P\in\calh_*\cup\calb_*$, then $\mu(\autf^{(P)}(S))\le\Delta_m$ by Lemma 
\ref{l3:s/a}(c), since each element in $\autf^{(P)}(S)$ normalizes 
$P$. So $\Delta_m=\Delta_t$, and $m\equiv t$ (mod $p-1$).  

\smallskip

\noindent\textbf{(b) } Set 
	\[ \Lambda_P=\bigl\{\alpha|_P\,\big|\,\alpha\in N_{\autf(S)}(P)\bigr\}
	\qquad\textup{and}\qquad \5P = P\cap S' = \begin{cases} 
	Z\cap S' = Z_0 & \textup{if $P\in\calh$} \\
	Z_2\cap S'\cong C_p^2 & \textup{if $P\in\calb$.}
	\end{cases} \]
Then $\Lambda_P$ normalizes $\5P$, and the induced action of 
$\Lambda_P$ on $P/\5P$ factors through the quotient group 
$\Lambda_P/\Aut_S(P)$ of order prime to $p$ and normalizes 
$(A\cap{}P)/\5P$. Since $P$ splits over $P\cap{}A$, there is a subgroup 
$P^*/\5P\le P/\5P$ of order $p$, also normalized by $\Lambda_P$, such that 
$P/\5P=((P\cap A)/\5P)\times(P^*/\5P)$. Since $Z_2\le A_0=ZS'$ by Lemma 
\ref{l1:s/a}(d), $P=P^*Z$ and $P^*\cap{}Z=Z_0$. Also, $|P^*|=p^2$ 
if $P\in\calh$, and $P^*$ is extraspecial of order $p^3$ if 
$P\in\calb$.

Assume first that $P\in\calh$. Choose $\alpha\in{}N_{\autff(S)}(P)$ such 
that $\mu(\alpha)$ generates $\Delta_{-1}$.  Then $\alpha(P^*)=P^*$ since 
$P^*$ is normalized by $\Lambda_P$, $\alpha$ acts non-trivially on $Z_0$ and 
on $P^*/Z_0$, and induces the identity on $Z/Z_0$.  Set $Z^*=C_Z(\alpha)$; 
thus $Z=Z_0\times Z^*$. Also, $P^*=[\alpha,P]$, and $P^*\cong C_p^2$ since 
$P=Z^*\times P^*$ splits over $Z$. Since $\autf(Z)$ has order prime to $p$, 
and normalizes $Z_0$ which is a direct factor in $Z$, there is an 
$\autf(Z)$-invariant subgroup of $Z$ complementary to $Z_0$, and this can 
only be $Z^*=C_Z(\alpha)$. In particular, $Z^*$ is normalized by 
$\autf(S)$, and hence by $\Lambda_P$.

Now assume that $P\in\calb$. Choose $\alpha\in N_{\autff(S)}(P)$ such that 
$\mu(\alpha)$ generates $\Delta_{0}$.  Then $\alpha(P^*)=P^*$ since $P^*$ 
is normalized by $\Lambda_P$, $\alpha$ acts non-trivially on 
$P^*/(P^*\cap{}A)$ and on $(P^*\cap{}A)/Z_0$ and trivially on $Z_0$, and 
hence $P^*$ has exponent $p$. Also, since $P^*/Z_0=[\alpha,P/Z_0]$, the 
choice of $P^*$ was unique.

We have now shown that $O^{p'}(\Out(P^*))\cong\SL_2(p)$ in both cases 
($P\in\calh$ or $P\in\calb$). Define
	\[ \Theta_0 = \begin{cases} 
	\bigl\{\alpha\in\Aut(P) \,\big|\,
	\alpha|_{P^*}\in O^{p'}(\Aut(P^*)),~ \alpha|_{Z^*}=\Id 
	\bigr\} & \textup{if $P\in\calh$} \\
	\bigl\{\alpha\in\Aut(P) \,\big|\,
	\alpha|_{P^*}\in O^{p'}(\Aut(P^*)),~ \alpha|_Z=\Id 
	\bigr\} \vphantom{\overset{X}{Y}} & \textup{if $P\in\calb$} 
	\end{cases} \]
and set $\Theta=\Lambda_P\Theta_0$. Since $\Lambda_P$ normalizes 
$\Theta_0$, $\Theta$ is a subgroup of $\Aut(P)$. Also, 
$\Theta\ge\Theta_0\ge\Inn(P)$, and $\Theta_0/\Inn(P)\cong\SL_2(p)$.

\noindent\underbar{Proof of (i)--(iii)} \
Since $\Theta=\Lambda_P\Theta_0$ where $\Lambda_P$ normalizes $\Theta_0$, 
$|\Theta/\Theta_0|=|\Lambda_P/(\Lambda_P\cap\Theta_0)|$, and this is prime 
to $p$ since $\Lambda_P\cap\Theta_0\ge\Aut_S(P)\in\sylp{\Lambda_P}$. Thus 
$\Theta_0=O^{p'}(\Theta)$, and $\Aut_S(P)\in\sylp{\Theta}$. So (i) and (ii) 
hold. Also, (iii) holds since $[\alpha,Z]\le Z_0$ for all 
$\alpha\in N_{\Theta_0}(\Aut_S(P))$ and $\Theta_0=O^{p'}(\Theta)$.

\noindent\underbar{Proof of (iv)} \
By construction, $\Lambda_P\le N_\Theta(\Aut_S(P))$. So to prove (iv), it 
suffices to show that $\Lambda_P\ge N_{\Theta_0}(\Aut_S(P))$. Fix a generator 
$u\in(\Z/p)^\times$. By assumption, there is $\alpha\in N_{\autff(S)}(P)$ 
such that $\mu(\alpha)=(u,u^t)$. Thus $\alpha|_P\in\Lambda_P$, 
and $[\alpha,Z]\le Z_0$. Upon replacing $\alpha$ 
by an appropriate power, if necessary, we can assume that $|\alpha|$ is 
prime to $p$.

%%N_{\autf(P)}(\Aut_S(P))$, 
 
If $P\in\calh$, then $\alpha|_{Z^*}=\Id$ (recall that 
$\alpha|_Z\in\autf(Z)$ normalizes $Z^*$), $t=-1$, so $\alpha$ acts on 
$P^*\cong C_p^2$ via $\diag(u,u^{-1})$. If $P\in\calb$, then $t=0$, 
$\alpha|_{Z}=\Id$ since $\alpha$ induces the identity on $Z_0$ and on 
$Z/Z_0$, and $\alpha$ acts on $P^*/Z_0\cong C_p^2$ via $\diag(u,u^{-1})$. 
Thus $\alpha|_P\in\Theta_0$, and 
$N_{\Theta_0}(\Aut_S(P))=\Aut_S(P)\gen{\alpha}\le\Lambda_P$. This finishes 
the proof of (iv).

\noindent\underbar{Proof of uniqueness} \
Let $\Theta^*\le\Aut(P)$ be another subgroup which satisfies (i)--(iv). Since 
$\Aut_S(P)\in\sylp{\Theta^*}$ by (i), 
$\Theta^*=N_{\Theta^*}(\Aut_S(P))O^{p'}(\Theta^*)$ by the Frattini argument, 
and so $\Theta^*=\Lambda_PO^{p'}(\Theta^*)$ by (iv). 

It remains to prove that $O^{p'}(\Theta^*)=\Theta_0$. Since 
$O^{p'}(\Theta^*)/\Inn(P)\cong\SL_2(p)\cong\Theta_0/\Inn(P)$ by (ii), it 
suffices to show that $O^{p'}(\Theta^*)\le\Theta_0$. 

If $P\in\calb$, then set $P^+=[O^{p'}(\Theta^*),P]$. By Lemma 
\ref{GonA}, applied with $P/Z_0$ and $\Theta^*$ in the roles of $A$ and 
$G$, $P^+/Z_0\cong C_p^2$ and is complementary to 
$C_{P/Z_0}(O^{p'}(\Theta^*))=Z/Z_0$. Hence $P^+\ge\5P$, 
$P/\5P=(Z_2/\5P)\times(P^+/\5P)$, and $P^+$ is normalized by the action of 
$\Lambda_P\le\Theta^*$; so $P^+=P^*$ by the uniqueness of $P^*$ shown 
above. Also, $\Aut_S(P)\in\sylp{\Theta^*}$ acts trivially on $Z$ and 
$Z=Z(P)$ is characteristic in $P$, so $O^{p'}(\Theta^*)$ also acts 
trivially on $Z$. Since $\Theta^*$ normalizes 
$P^+=[O^{p'}(\Theta^*),P]$, we have $O^{p'}(\Theta^*)\le\Theta_0$ by 
definition of $\Theta_0$. 

\iffalse
It remains to show that $O^{p'}(\Theta^*)=\Theta_0$. Since 
$O^{p'}(\Theta^*)/\Inn(P)\cong\SL_2(p)\cong\Theta_0/\Inn(P)$ by (ii), it 
suffices to show that for $\alpha\in O^{p'}(\Theta^*)$, $\alpha|_{Z^*}=\Id$ 
if $P\in\calh$ and $\alpha|_Z=\Id$ if $P\in\calb$. If $P\in\calb$, then 
$\Aut_S(P)\in\sylp{\Theta^*}$ acts trivially on $Z=Z(P)$, so 
$O^{p'}(\Theta^*)$ also acts trivially on $Z$. 
\fi

If $P\in\calh$, then by Lemma \ref{GonA}, $P=Z\7\times P\7$ where 
$O^{p'}(\Theta^*)\cong\SL_2(p)$ acts trivially on $Z\7$ and acts faithfully 
on $P\7\cong C_p^2$. We showed above that there is $\alpha\in 
N_{\autff(S)}(P)$ such that $Z^*=C_P(\alpha)$ and $P^*=[\alpha,P]$, and 
$\alpha|_P\in\Lambda_P\le\Theta^*$ by (iv). Hence $Z^*=Z\7$ and $P^*=P\7$, and so 
$O^{p'}(\Theta^*)\le\Theta_0$. 
\end{proof}

In the next lemma, we describe the conditions for a saturated fusion system 
$\calf$ over $S$ to be reduced.  This requires some more precise 
information about the $\calf$-essential subgroups and their automorphisms 
in this situation. 

In the proofs of the next lemma and theorem, we refer several times to the 
\emph{extension axiom} for saturated fusion systems. This axiom states 
that in a saturated fusion system $\calf$ over a $p$-group $S$, if 
$\varphi\in\homf(P,Q)$ is such that $Q=\varphi(P)$ is fully centralized in 
$\calf$, and if $\4P\ge P$ is such that $P\nsg\4P$ and 
$\varphi\Aut_{\4P}(P)\varphi^{-1}\le\Aut_S(Q)$, then there is 
$\4\varphi\in\homf(\4P,S)$ which extends $\varphi$. We refer to 
\cite[Proposition I.2.5]{AKO} for more detail, including a description of 
how this axiom can be used to characterize saturated fusion systems.

\begin{Lem} \label{l8:s/a}
Fix an odd prime $p$, and a $p$-group $S$ which contains a unique abelian 
subgroup $A\nsg{}S$ of index $p$.  Let $\calf$ be a saturated fusion system 
over $S$, and assume that $A\nnsg\calf$.  We use the notation of \ref{n:not1} 
and \ref{n:not2}. Assume also that $A\in\EE\calf$, and set $G=\autf(A)$ and 
$\UUU=\Aut_S(A)\in\sylp{G}$.  Then the following hold.
\begin{enuma} 

\item $O_p(\calf)=1$ if and only if either
\begin{enumi} 
\smallskip

\item there are no non-trivial $G$-invariant subgroups of $Z$; or 

\item $\EE\calf\cap\calh\ne\emptyset$ and $Z_0$ is the only non-trivial 
$G$-invariant subgroup of $Z$.
\end{enumi}

\item If $O_p(\calf)=1$, then $O^p(\calf)=\calf$ if and only if $[G,A]=A$.

\item $O^{p'}(\calf)=\calf$ if and only if one of the following holds:  
either 
\smallskip

\begin{enumi} 
\item $\EE\calf\sminus\{A\}=\calh_0\cup\calb_*$, and 
%%$m\equiv0$ (mod $p-1$), $\mu(\autff(S))=\Delta$, and \\
$\autf(S)=\Gen{\autff(S),\autf^{(A)}(S)}$; \ or

\item $\EE\calf\sminus\{A\}=\calb_0\cup\calh_*$, and 
%%$m\equiv-1$ (mod $p-1$), $\mu(\autff(S))=\Delta$, and \\
$\autf(S)=\Gen{\autff(S),\autf^{(A)}(S)}$; \ or 

\item $\EE\calf\sminus\{A\}\subseteq\calh$ and 
%%$\mu(\autff(S))\ge\Delta_{-1}$, and \\
$\autf(S)=\Gen{(\autff(S)\cap\mu^{-1}(\Delta_{-1})), \autf^{(A)}(S)}$; 
\ or 

\item $\EE\calf\sminus\{A\}\subseteq\calb$ and 
%%$\mu(\autff(S))\ge\Delta_{0}$, and \\
$\autf(S)=\Gen{(\autff(S)\cap\mu^{-1}(\Delta_{0})), \autf^{(A)}(S)}$.
\end{enumi}

\end{enuma}
\end{Lem}

\begin{proof}  \noindent\textbf{(a) } Upon replacing the statements by 
their negatives, we must show that $O_p(\calf)\ne1$ if and only if
	\beqq \parbox{\short}{there is a non-trivial subgroup $1\ne Q\le Z$ 
	normalized by $G$, such that either $Q\ne Z_0$ or 
	$\EE\calf\cap\calh=\emptyset$.} \label{e:xxx} \eeqq

Assume first that $O_p(\calf)\ne1$, and set $Q=O_p(\calf)$. Since 
$A\in\EE\calf$, $Q\le A$ and is $G$-invariant by Proposition \ref{Q<|F}(a).  
Since $A\nnsg\calf$, there is some 
$P\in\EE\calf\sminus\{A\}\subseteq\calb\cup\calh$. By Proposition 
\ref{Q<|F}(a) again, if $P\in\calh$, then $Q\le P\cap{}A=Z$, while if 
$P\in\calb$, then $Q\le\bigcap_{\alpha\in\autf(P)}\alpha(P\cap{}A)=Z$ (see 
Lemma \ref{l2:s/a}(a)). 
Thus $Q$ is a non-trivial $G$-invariant subgroup of $Z$. If $Q=Z_0$, then 
$\EE\calf\cap\calh=\emptyset$, since for $P\in\EE\calf\cap\calh$, $Z_0$ is 
not normalized by $\autf(P)$ ($Z_0<P^*\cong C_p^2$ in the notation of Lemma 
\ref{l2:s/a}(b)). Thus \eqref{e:xxx} holds in this case.

Conversely, assume that \eqref{e:xxx} holds. In particular, $1\ne{}Q\le Z$ is 
$G$-invariant. For each $\alpha\in\autf(S)$, $\alpha(A)=A$ since $A$ is the 
unique abelian subgroup of index $p$, so $\alpha|_A\in{}G$, and thus 
$\alpha(Q)=Q$. Since each element of $\autf(Z)$ extends to $S$ by the 
extension axiom, $Q$ is also normalized by $\autf(Z)$. Also, for each 
$P\in\EE\calf\cap\calb$, $Z=Z(P)$ is characteristic in $P$ and so $Q$ is 
also normalized by $\autf(P)$. In particular, if 
$\EE\calf\cap\calh=\emptyset$, then $Q\nsg\calf$ by Proposition 
\ref{Q<|F}(a), so $O_p(\calf)\ne1$. 

Now assume that $\EE\calf\cap\calh\ne\emptyset$, and hence by assumption 
that $Q\ne{}Z_0$. By Lemma \ref{l2:s/a}(b), there is a 
unique $\autf(Z)$-invariant splitting $Z=Z_0\times{}Z^*$.  Set 
$Q^*=Q\cap{}Z^*$. If $Q\ge Z_0$, then $Q=Q^*\times Z_0$. Otherwise, $Q\cap 
Z_0=1$ (recall $|Z_0|=p$), and since $Q$ is $\autf(Z)$-invariant, the 
uniqueness of the splitting implies that $Q\le Z^*$ and hence $Q=Q^*$. 
Since $Q\ne{}Z_0$, we have $Q^*\ne1$ in either case. 

For each $\varphi\in\autf(A)=G$, $\varphi(Q^*)\le Q\le Z$, so by the 
extension axiom, $\varphi|_{Q^*}$ extends to some $\4\varphi\in\autf(S)$, 
and $\varphi(Q^*)=\4\varphi(Q^*)=Q^*$ since $Q^*$ is $\autf(Z)$-invariant. 
So by the same arguments as those applied above to $Q$, $Q^*$ is normalized 
by $\autf(P)$ for each $P\in(\{S\}\cup\EE\calf)\sminus\calh$. If 
$P\in\EE\calf\cap\calh$, then for each $\alpha\in\autf(P)$, 
$\alpha(Z^*)=Z^*$ by Lemma \ref{l2:s/a}(b), so $\alpha|_{Z^*}$ extends to 
an element of $\autf(S)$ and hence of $\autf(Z)$, and in particular, 
$\alpha(Q^*)=Q^*$. Thus $1\ne{}Q^*\nsg\calf$, and hence $O_p(\calf)\ne1$. 

\smallskip

\noindent\textbf{(b) }  Assume that $O_p(\calf)=1$ and $[G,A]<A$.  Since 
$C_A(G)\le{}Z_0\le[\UUU,A]$ by (a), $[G,A]\ge{}C_A(\UUU)=Z$ by Lemma 
\ref{l0:xx}.  Hence $[G,A]\ge ZS'=A_0$, with equality since 
$|A/A_0|=p$. The $G$-action on $A/A_0$ is thus trivial, and hence 
$\widehat{\mu}(\outf(S))\le\Delta_{m-1}$ by Lemma \ref{l3:s/a}(c), where 
$p^{m-1}=|A/Z|$.

By Lemma \ref{l4:s/a}(a), for each $1\le{}i\le{}p-1$, 
$\calh_i\subseteq\EE\calf$ implies 
$\widehat{\mu}(\outf(S))\ge\Delta_{-1}=\Delta_m$, while 
$\calb_i\subseteq\EE\calf$ implies 
$\widehat{\mu}(\outf(S))\ge\Delta_0=\Delta_m$.  Thus 
$\EE\calf\sminus\{A\}\subseteq\calh_0\cup\calb_0$, so 
$\foc(\calf)\le\Gen{[G,A],\calh_0,\calb_0}= A_0\gen{\xx}<S$ by 
Proposition \ref{Q<|F}(b), and $O^p(\calf)\ne\calf$ by Proposition 
\ref{Q<|F}(c). 

Conversely, if $[G,A]=A$, then $\foc(\calf)\ge A$.  Since $A\nnsg\calf$, 
$\EE\calf\supsetneqq\{A\}$, and hence $\foc(\calf)=S$. So 
$O^p(\calf)=\calf$ by Proposition \ref{Q<|F}(c).

\smallskip

\noindent\textbf{(c) }  By Proposition \ref{Q<|F}(d), and since 
$\autf^{(A)}(S)\ge\Inn(S)$, $O^{p'}(\calf)=\calf$ if and only if $\autf(S)$ 
is generated by the subgroups $\autf^{(P)}(S)$ for all $P\in\EE\calf$.  By 
Lemma \ref{l4:s/a}(a), if $P\in\EE\calf\sminus\{A\}$, then 
$\autf^{(P)}(S)=\autff(S)\cap\mu^{-1}(\Delta_{t})$, where $t=-1$ if 
$P\in\calh$ and $t=0$ if $P\in\calb$. Hence the conditions in (c.iii) are 
necessary and sufficient if $\EE\calf\sminus\{A\}\subseteq\calh$, and 
those in (c.iv) are necessary and sufficient if 
$\EE\calf\sminus\{A\}\subseteq\calb$.

If $\EE\calf\sminus\{A\}$ contains subgroups in both $\calh$ 
and $\calb$, then $\mu(\autff(S))\ge\Delta_0\Delta_{-1}=\Delta$ by Lemma 
\ref{l4:s/a}(a), and $\autf(S)=\gen{\autf^{(A)}(S),\autff(S)}$ by the above 
remarks. Also, by Lemma \ref{l1:s/a}(b), some $H_i$ or $B_i$ must be 
essential for $0<i<p$, and by Lemma \ref{l4:s/a}(a), $m\equiv-1$ (mod 
$p-1$) if $H_i\in\EE\calf$, while $m\equiv0$ (mod $p-1$) if 
$B_i\in\EE\calf$. Also, all subgroups in $\calh_*$ and in $\calb_*$ are 
$\calf$-conjugate by Lemma \ref{l3:s/a}(c) and since 
$\mu(\autff(S))=\Delta$. Hence $\EE\calf\sminus\{A\}=\calh_0\cup\calb_*$ 
or $\calb_0\cup\calh_*$, and we are in the situation of (c.i) or (c.ii). 
\end{proof}

We are now ready to describe the reduced fusion systems over non-abelian 
$p$-groups $S$ which contain a unique abelian subgroup of index $p$ which 
is essential. Recall that we defined 
	\[ \autff(A) = \bigl\{\alpha\in{}N_G(\UUU)\,\big|\,
	[\alpha,Z]\le Z_0 \bigr\} 
	= \bigl\{\alpha|_A\,\big|\, 
	\alpha\in\autff(S) \bigr\}\,. \]
Let 
	\[ \mu_A \: \autff(A) \Right4{} \Delta \]
be the homomorphism defined by setting $\mu_A(\alpha)=\mu(\4\alpha)$ for 
some $\4\alpha\in\autf(S)$ which extends $\alpha$. Since $\alpha\in 
N_G(\UUU)=N_{\autf(A)}(\Aut_S(A))$, it does extend to some 
$\4\alpha\in\autf(S)$ by the extension axiom. If $\beta$ is another 
extension, then $\beta^{-1}\4\alpha\in\autf(S)$ induces the identity on $A$ 
and (since $C_S(A)=A$) on $S/A$, and hence by definition of $\mu$ lies in 
$\Ker(\mu)$. Thus $\mu_A(\alpha)=\mu(\4\alpha)$ is independent of the 
choice of $\4\alpha$.

\begin{Thm} \label{t3:s/a}
Fix an odd prime $p$, and a $p$-group $S$ which contains a unique abelian 
subgroup $A\nsg{}S$ of index $p$.  Let $\calf$ be a reduced fusion system 
over $S$ for which $A$ is $\calf$-essential.  We use the notation of 
\ref{n:not1} and \ref{n:not2}, and also set $\EE0=\EE\calf\sminus\{A\}$ and 
$G=\autf(A)$ and $G\7=\autff(A)$, so that 
$\UUU=\Aut_S(A)\in\sylp{G}$.  Let $m\ge3$ be such that $|A/Z|=p^{m-1}$. Set 
$\sigma=\prod_{i=0}^{p-1}\9{\xx^i}(\xa)=(\xa\xx)^p\xx^{-p}\in Z$. 
Then the following hold:
\begin{enuma}  

\item $Z_0=C_A(\UUU)\cap[\UUU,A]$ has order $p$, and hence 
$A_0=C_A(\UUU)[\UUU,A]$ has index $p$ in $A$.  

\item There are no non-trivial $G$-invariant subgroups of $Z=C_A(\UUU)$, 
aside (possibly) from $Z_0$.  

\item $[G,A]=A$.

\item One of the conditions (i)--(iv) holds, described in 
Table \ref{tbl:(d)}, where $I$ is always some nonempty subset 
of $\{0,1,\dots,p-1\}$. 
\begin{table}[ht]
\[ \renewcommand{\arraystretch}{1.5}
\begin{array}{|c|c|c|c|c|c|} \hline
 & \mu_A(G\7) & G=O^{p'}(G)X~ \textup{ where} & 
\textup{$m$ (mod ~ ${p-1}$)} & \sigma & \EE0  \\ \hline\hline
\textup{(i)} & \Delta & X=G\7 & \equiv0 & \sigma\in\Fr(Z) &
\calh_0\cup\calb_*  \\\hline

\textup{(ii)} & \Delta & X=G\7 & \equiv-1 & \sigma\in\Fr(Z) &
\calb_0\cup\calh_*  \\\hline

&  & & \equiv-1 & \sigma\in\Fr(Z) & \bigcup_{i\in I}\calh_i
\\\cline{4-6}

\halfup{\textup{(iii)}} & \halfup{\ge\Delta_{-1}} & 
\halfup{X=\mu_A^{-1}(\Delta_{-1})} 
& - & - & \calh_0 \\\hline

&  &  {X=\mu_A^{-1}(\Delta_{0})}
& \equiv0 & \sigma\in\Fr(Z) & \bigcup_{i\in I}\calb_i
\\\cline{4-6}

\halfup{\textup{(iv)}} & \halfup{\ge\Delta_{0}} & 
\textup{$Z_0$ not $G$-invariant} & - & - & \calb_0 \\\hline

\hline
\end{array}
\]
\caption{} \label{tbl:(d)}
\end{table}
\end{enuma}

Conversely, for each $G$, $A$, $\UUU\in\sylp{G}$, and 
$\EE0\subseteq\calh\cup\calb$ which satisfy conditions (a)--(d), where 
$|\UUU|=p$, $\UUU\nnsg{}G$, $G\7=\{\alpha\in N_G(\UUU)\,|\,[\alpha,Z]\le 
Z_0\}$, and $\calb$ and $\calh$ are defined as in 
Notation \ref{n:not1} for $S=A\rtimes\UUU$, there is a reduced fusion 
system $\calf$ over $A\sd{}\UUU$ with $\autf(A)=G$ and 
$\EE\calf=\EE0\cup\{A\}$, unique up to isomorphism. All such fusion systems 
are simple. All such fusion systems are exotic, except for the fusion 
systems of the simple groups listed in Table \ref{tbl:type3}. \\ 
Such a fusion system $\calf$ has a proper strongly closed subgroup if and 
only if $A_0=C_A(\UUU)[\UUU,A]$ is $G$-invariant, and $\EE0=\calh_i$ or 
$\calb_i$ for some $i=0,\dots,p-1$, in which case $A_0H_i=A_0B_i$ is strongly closed. 
\end{Thm}

In Table \ref{tbl:type3}, $e$ is such that $p^e$ is the exponent of $A$.  In 
all cases except when $\Gamma\cong\PSL_p(q)$ and $e>1$, $A$ is homocyclic.

\begin{table}[ht]
\begin{small} 
\[ \renewcommand{\arraystretch}{1.5}
\begin{array}{|c|c|c|c|c|c|c|c|} \hline
\Gamma & p & \textup{conditions} & \rk(A) & e & m &
G=\Aut_\Gamma(A) & \EE0 \\ \hline\hline
A_{pn} & p & p\le n<2p & n & 1 & p & \frac12 
C_{p-1}\wr S_n & \calh_0 \\ \hline

\Sp_4(p) & p & \raisebox{3pt}{\hbox to 2cm{\hrulefill}} & 3 & 1 & 3 & 
\GL_2(p)/\{\pm I\} & \calb_0 \\ \hline

\PSL_p(q) & p & v_p(q{-}1)=1,\ p>3 & p{-}2 & 1 & p-2 & S_p & 
\calh_0\cup\calh_* \\ \hline

\PSL_p(q) & p & p^2|(q{-}1),\ p>3 & p{-}1 & v_p(q{-}1) & e(p{-}1)-1 & 
S_p & \calh_0\cup\calh_* \\ \hline

\PSL_n(q) & p & p|(q{-}1),\ p{<}n{<}2p & n{-}1 & v_p(q{-}1) & e(p{-}1)+1 & 
S_n & \calb_0 \\ \hline

P\varOmega_{2n}^+(q) & p & p|(q{-}1),\ p{\le} n{<}2p & n & v_p(q{-}1) & 
e(p{-}1)+1 & C_2^{n-1}\rtimes S_n & \calb_0 \\ \hline

\lie2F4(q) & 3 & q\ge8 & 2 & v_3(q{+}1) 
& 2e & \GL_2(3) & \calb_0\cup\calb_* \\ \hline 

E_n(q) & 5 & n=6,7,\ p|(q{-}1) & n & v_p(q{-}1) & 4e+1 & W(E_n) & \calb_0 
\\ \hline

E_n(q) & 7 & n=7,8,\ p|(q{-}1) & n & v_p(q{-}1) & 6e+1 & W(E_n) & \calb_0 
\\ \hline

E_8(q) & 5 & q\equiv\pm2 \pmod5 & 4 & v_5(q^4{-}1) & 4e & 
(C_4\circ2^{1+4}).S_6 & \calh_0\cup\calb_* \\ \hline

\Co_1 & 5 & \raisebox{3pt}{\hbox to 2cm{\hrulefill}} & 3 & 1 & 3 & 
4\times S_5 & \calb_0\cup\calh_* \\ \hline

\hline
\end{array}
\]
\end{small}
\caption{} \label{tbl:type3}
\end{table}

\begin{proof} We prove in Step 1 that conditions (a)--(d) are necessary, 
and in Step 2 that they are sufficient for the existence of a reduced 
fusion system. We also prove the uniqueness of the fusion system in 
Step 2. In Step 3, we list all strongly closed subgroups for the 
fusion systems constructed in Step 2, and then prove in Step 4 that they 
are all simple. In Step 5, we handle the question of which of these fusion 
systems are realizable.

\smallskip

\noindent\textbf{Step 1: } Assume that $\calf$ is a reduced fusion system 
over $S$. We must show that conditions (a)--(d) hold.

\smallskip

\noindent\textbf{(a) }  By Lemma \ref{l1:s/a}(c,d), $|Z_0|=p=|A/A_0|$.  

\smallskip

\noindent\textbf{(b,c) } Since $\calf$ is reduced, $O_p(\calf)=1$ and 
$O^p(\calf)=\calf$. So these claims follow from points (a) and (b), 
respectively, in Lemma \ref{l8:s/a}.  

\smallskip

\noindent\textbf{(d) } Cases (i)--(iv) here correspond exactly to cases 
(i)--(iv) of Lemma \ref{l8:s/a}(c). The conditions on $\EE0$ follow 
immediately from that lemma, while the conditions on 
$\mu_A(\autff(A))=\mu(\autff(S))$ and those on $m$ (mod $p-1$) follow from 
Lemma \ref{l4:s/a}(a). By Lemma \ref{l3:s/a}(d), for 
$P\in\calh_*\cup\calb_*$, $\sigma\in\Fr(Z)$ if and only if $P$ splits over 
$P\cap{}A$, and this is a necessary condition to have $P\in\EE\calf$ by 
Lemma \ref{l2:s/a}(a,b). 

To see why the conditions on $G$ hold, let  
	\[ R\:\autf(S) \Right5{} N_{\autf(A)}(\Aut_S(A))=N_G(\UUU) \]
be the homomorphism induced by restriction. By the extension axiom and by 
definition, $R$ sends subgroups of $\autf(S)$ as follows:
	\beqq \renewcommand{\arraystretch}{1.5}
	\begin{array}{|c|c||c|c|c|c|} \hline
	X & R^{-1}(Y) & \autf(S) & \autff(S) & \autf^{(A)}(S) 
	& \autff(S)\cap\mu^{-1}(\Delta_r) \\\hline
	R(X) & Y & N_G(\UUU) & \autff(A) & N_{O^{p'}(G)}(\UUU) 
	& \mu_A^{-1}(\Delta_r) \\\hline
	\end{array}
	\label{e:R(X)} \eeqq
In other words, the groups in the second row are the images of those in the 
first, and those in the first are the inverse images of those in the second.
By the Frattini argument, $G = O^{p'}(G)\cdot N_G(\UUU)
= O^{p'}(G)\cdot R(\autf(S))$. The claims in (d) about generators for $G$ 
now follow from the formulas for $\autf(S)$ in Lemma \ref{l8:s/a}(c) (and since 
$R(\autf^{(A)}(S))\le O^{p'}(G)$). 

\smallskip

\noindent\textbf{Step 2: } Now assume that $A$, $G$, and $\EE0$ are as 
above and satisfy (a)--(d). We will show that they are realized by a unique 
reduced fusion system $\calf$. 

Set $\Gamma=A\sd{}G$, and identify $S=A\sd{}\UUU\in\sylp{\Gamma}$. Choose a 
generator $\xx\in\UUU<S$.  Choose $\xa\in{}A\sminus A_0$ so that 
$S'\gen{\xa}$ is normalized by $N_G(\UUU)$.  Set $Z=Z(S)$, $Z_2=Z_2(S)$, 
$H_i=Z\gen{\xx\xa^i}$, and $B_i=Z_2\gen{\xx\xa^i}$, as in Notation 
\ref{n:not1} and \ref{n:not2}.

Set $\calf_0=\calf_S(\Gamma)$. We will apply Lemmas \ref{l3:s/a} 
and \ref{l4:s/a}(b) here with $\calf_0$ in the role of $\calf$. Note that 
$\Aut_{\calf_0}\7(S)$ is the group of all $\alpha\in\Aut_\Gamma(S)$ that 
induce the identity on $Z/Z_0$.

%%, and 
%%$\mu(\Aut_{\calf_0}\7(S))=\mu_A(\autff(A))$.} 

Let $Q_1,\dots,Q_k\in\EE0\cap\{B_i,H_i\,|\,0\le i\le p-1\}$ be a set of 
representatives for the $\Gamma$-conjugacy classes containing 
subgroups in $\EE0$. For each $i\le k$, set $K_i=\Aut_\Gamma(Q_i)$. 
Set $t_i=-1$ if $Q_i\in\calh$, and $t_i=0$ if $Q_i\in\calb$. Since 
$\sigma\in\Fr(Z)$ whenever $Q_i\in\calh_*\cup\calb_*$, $Q_i$ splits over 
$Q_i\cap{}A$ in all cases by Lemma \ref{l3:s/a}(d). By the 
assumptions in (d), there is $\alpha\in \Aut_{\calf_0}\7(S)$ such that $\mu(\alpha)$ 
generates $\Delta_t$. By Lemma \ref{l3:s/a}(c) and the assumptions on $m$, 
$\alpha(Q_i)$ is $S$-conjugate to $Q_i$. So upon composing $\alpha$ with an 
inner automorphism (hence in $\Ker(\mu)$), we can arrange that $\alpha\in 
N_{\Aut_{\calf_0}\7(S)}(Q_i)$. 

\iffalse
We want to apply Lemma \ref{l4:s/a}(b), with $\calf_0$ in the role 
of $\calf$ and $Q_i$ in the role of $P$. Thus 
$\Aut_{\calf_0}\7(S)$ is the group of all 
$\alpha\in\Aut_\Gamma(S)$ that induce the identity on $Z/Z_0$, and 
$\mu(\Aut_{\calf_0}\7(S))=\mu_A(\autff(A))$. 
\fi

Each element of $K_i=\Aut_\Gamma(Q_i)$ extends to an element of 
$\Aut_\Gamma(S)$ since $N_\Gamma(Q_i)\le N_\Gamma(AQ_i)$, and so 
$K_i=\bigl\{\beta|_{Q_i}\,\big|\,\beta\in 
N_{\Aut_\Gamma(S)}(Q_i)\bigr\}$. By Lemma \ref{l4:s/a}(b), there is a 
unique subgroup $\Theta_i\le\Aut(Q_i)$ such that $\Theta_i\ge\Inn(Q_i)$, 
$\Aut_S(Q_i)\in\sylp{\Theta_i}$, $N_{\Theta_i}(\Aut_S(Q_i))=K_i$, 
$[\beta,Z]\le Z_0$ for $\beta\in N_{O^{p'}(\Theta_i)}(\Aut_S(Q_i))$, and 
$O^{p'}(\Theta_i)/\Inn(Q_i)\cong\SL_2(p)$.

Set $\calf=\gen{\calf_0,\Theta_1,\dots,\Theta_k}$: the smallest fusion 
system over $S$ which contains $\calf_0$, and such that 
$\autf(Q_i)\ge\Theta_i$ for each $i$. Note in particular that 
$\autf(S)=\Aut_{\calf_0}(S)=\Aut_\Gamma(S)$. By \cite[Proposition 
5.1]{BLO4}, to see that $\calf$ is saturated, it suffices to check that the 
following conditions hold. 
\begin{enum1} 

\item \boldd{For $i\ne{}j$, $Q_i$ is not $\Gamma$-conjugate to a subgroup 
of $Q_j$.}  By assumption, $Q_i$ and $Q_j$ are not $\Gamma$-conjugate, and 
hence are not $\calf$-conjugate since there is no larger group in the 
generating set which could conjugate the one into the other. If $Q_i=H_k$ 
and $Q_j=B_\ell$ for some $k,\ell\in\{0,1,\dots,p-1\}$, then by (d), 
either $k=0$ and $\ell\ne0$ or vice versa, so $B_k$ and $B_\ell$ are not 
$\Aut_\Gamma(S)$-conjugate (hence not $\calf_0$-conjugate, hence not 
$\calf$-conjugate) by Lemma \ref{l3:s/a}(c), applied with 
$\calf_0$ in the role of $\calf$. 

\item \boldd{For each $i$, $Q_i$ is $p$-centric in $\Gamma$, but no proper 
subgroup $P<Q_i$ is $\calf$-centric nor an essential $p$-subgroup of 
$\Gamma$.} Here, a $p$-subgroup $Q\le\Gamma$ is essential if it is 
$p$-centric in $\Gamma$ and $\Out_\Gamma(Q)$ has a strongly $p$-embedded 
subgroup \cite[Definition 3.2(b)]{BLO4}. In all cases, $C_\Gamma(Q_i)\cap 
A=Z=Z(Q_i)\cap A$. Also, $Q_i\in\calh$ implies $C_\Gamma(Q_i)\le 
A\sd{}C_G(Z)$ where $\UUU\in\sylp{C_G(Z)}$, while $Q_i\in\calb$ implies 
$C_\Gamma(Q_i)\le A\sd{}C_G(Z_2)$ where $C_G(Z_2)$ has order prime to $p$. 
Thus $Z(Q_i)\in\sylp{C_\Gamma(Q_i)}$ in both cases, so $Q_i$ is $p$-centric 
in $\Gamma$ by definition.

Each proper subgroup of $Q_i$ either does not contain $Z$ or is in the 
$\Theta_i$-orbit of (hence $\calf$-conjugate to) a proper subgroup of $A$, 
and in either case, is not $\calf$-centric. Since $A\nsg\Gamma$ (so 
$A\nsg\calf_0=\calf_S(\Gamma)$), each essential $p$-subgroup of $\Gamma$ contains 
$A$ by Proposition \ref{Q<|F}(a).

\item \boldd{For each $i$, $p\nmid[\Theta_i:K_i]$ and $K_i/\Inn(Q_i)$ is 
strongly $p$-embedded in $\Theta_i/\Inn(Q_i)$.} By assumption, 
$\Aut_S(Q_i)\in\sylp{\Theta_i}$, and $K_i=N_{\Theta_i}(\Aut_S(Q_i))$. Thus 
$\Out_S(Q_i)\in\sylp{\Theta_i/\Inn(Q_i)}$ has order $p$ (and is not 
normal), and so its normalizer has index prime to $p$ and is strongly 
$p$-embedded in $\Theta_i/\Inn(Q_i)$. 

\end{enum1}

It remains to prove that $\calf$ is reduced. By (b), there are no 
non-trivial $G$-invariant subgroups of $Z$ except possibly for $Z_0$, 
$\EE\calf\cap\calh\ne\emptyset$ in cases (d.i)--(d.iii), and $Z_0$ is 
not $G$-invariant in case (d.iv). Hence $O_p(\calf)=1$ by Lemma 
\ref{l8:s/a}(a). Also, $O^p(\calf)=\calf$ by Lemma \ref{l8:s/a}(b), and 
since $[G,A]=A$ by (c). 

As for showing that $O^{p'}(\calf)=\calf$, we claim that the 
necessary conditions in Lemma \ref{l8:s/a}(c.i--c.iv) follow from 
(d.i--d.iv). The conditions on $\EE0=\EE\calf\sminus\{A\}$ are clear. As for the 
conditions on $\autf(S)$, these follow since if $G=O^p(G)X$ by (d) for $X\le 
N_{\Aut(A)}(\UUU)$, and $R\:\Aut(S)\Right2{}N_{\Aut(A)}(\UUU)$ is as in 
Step 1, then
	\[ \autf(S) = R^{-1}(N_G(\UUU)) = R^{-1}(N_{O^{p'}(G)}(\UUU)) \cdot 
	R^{-1}(X) = \autf^{(A)}(S)\cdot R^{-1}(X), \]
where $R^{-1}(X)$ is described in \eqref{e:R(X)}.

The uniqueness of $\calf$ follows from the uniqueness of the $\Theta_i$ 
(Lemma \ref{l4:s/a}(b)). 

\smallskip

\noindent\textbf{Step 3: } We next list the proper non-trivial strongly 
closed subgroups in the reduced fusion system $\calf$ constructed in Step 
2.

Assume that $1\ne{}Q<S$ is strongly closed in $\calf$.  If $Q\le Z$, then $Q$ is 
contained in all $\calf$-essential subgroups, so $Q\nsg\calf$ by 
Proposition \ref{Q<|F}(a), which is impossible since $O_p(\calf)=1$.  
Thus $Q\nleq{}Z$.  

Now, $(QZ/Z)\cap Z(S/Z)\ne1$ since $Q\nsg{}S$, so 
$Q\cap{}Z_2\nleq Z$.  Fix $g\in(Q\cap Z_2)\sminus Z$.  Then 
$Q\ge[g,S]=Z_0$ since $Q\nsg S$.

Since $A\nnsg\calf$, there is $P\in\EE0$.  If $P\in\calb$, then the 
$\autf(P)$-orbit of $g\in{}Z_2\le P$ is not contained in $A$.  If 
$P\in\calh$, then the $\autf(P)$-orbit of $Z_0\le P$ is not contained in 
$A$.  So in either case, $Q\nleq{}A$.  Hence 
$Q\ge[Q,S]\ge[\UUU,A]=S'$.  

Set $G_0=O^{p'}(G)$. Since $Q\cap{}A$ is normalized by the action of 
$G=\autf(A)$, and contains $[\UUU,A]$ where $\UUU\in\sylp{G}$, 
$Q\ge[G_0,A]$.  Since $[G,A]=A$ by (c), the group $G/G_0$ of order 
prime to $p$ acts on $A/[G_0,A]$ with $C_{A/[G_0,A]}(G/G_0)=1$. Also, 
$G=G_0\autff(A)$ by (d), so $C_{A/[G_0,A]}(\autff(A))=1$. Since $\autff(A)$ 
acts trivially on $Z/Z_0$ (and $Z_0\le S'\le[G_0,A]$), the natural 
homomorphism $Z/Z_0\Right2{}A/[G_0,A]$ has trivial image. Hence 
$A_0=ZS'\le[G_0,A]\le Q$.

Since $Q<S$ by assumption, this proves that $Q=A_0\gen{\xx\xa^i}=A_0H_i$ 
for some $i$, and has index $p$ in $S$. If $A_0H_i$ is strongly closed, 
then neither $H_j$ nor $B_j$ can be $\calf$-essential for any $j\ne{}i$ 
($0\le j\le p-1$), so $\EE0=\calh_i$ or $\calb_i$. 

Conversely, if $\EE0=\calh_i$ or $\calb_i$ for some $0\le i\le p-1$ and 
$A_0$ is normalized by $G$, then $A_0H_i$ is normalized by $\autf(S)$ and 
$\autf(A)$ and contains all $\calf$-essential subgroups other than $A$, 
hence is strongly closed in $\calf$.

\smallskip

\noindent\textbf{Step 4: } We now prove that $\calf$ is simple. Assume 
otherwise. Then there is a proper normal subsystem $\cale\nsg\calf$ 
over a nontrivial
strongly closed subgroup $1\ne{}T\le S$, and $T<S$ since 
$O^{p'}(\calf)=\calf$. We just saw that this implies 
$T=A_0H_i$ for some $i$, and $\EE0=\calh_i$ or $\calb_i$. 

Assume first that $m\ge4$. Thus $|S'|=|A/Z|=p^{m-1}\ge p^3$. For 
$P\in\EE0$, we have $P<T$, so $N_S(P)\le A_0P=T$. Thus 
$\Aut_S(P)\le\Aut_{\cale}(P)$ and $\Aut_{\cale}(P)\nsg \autf(P)$, which 
imply that $P\in\EE\cale$. (Since $\Aut_{\cale}(P)$ has Sylow $p$-subgroups 
which are non-normal and have order $p$, it has strongly $p$-embedded 
subgroups.) Also, $A_0=A\cap{}T$ is the unique abelian subgroup of index 
$p$ in $T$ since $|[T,T]|=\frac1p|S'|\ge p^2$, so by Steps 1 and 2 applied 
to $\cale$, we have $\EE\cale=\{A_0,\calh_i\}$ or $\{A_0,\calb_i\}$. Here, 
$\calh_i$ or $\calb_i$ plays the role for $\cale$ that $\calh$ or $\calb$ 
plays for $\calf$. So by Lemma \ref{l4:s/a}(a) (or by Table \ref{tbl:(d)} 
case (iii) or (iv)), $m-1\equiv t$ (mod $p-1$), 
where as usual, $t=-1$ if $\EE0=\calh_i$ and $t=0$ if $\EE0=\calb_i$. Also, 
$G=O^{p'}(G)\mu_A^{-1}(\Delta_t)$, and by Lemma \ref{l3:s/a}(c), elements 
in $\mu_A^{-1}(\Delta_t)=\mu_A^{-1}(\Delta_{m-1})$ act trivially on 
$A/A_0$. Since $T$ is strongly closed, $A_0=A\cap T$ is $G$-invariant, 
and so $O^{p'}(G)$ acts trivially on $A/A_0\cong C_p$. Thus 
$[G,A]\le A_0$, which contradicts (c). We conclude that $\calf$ is 
simple.

%%From Table 
%%\ref{tbl:(d)}, we now see that $i=0$ (the condition on $m$ does not hold 
%%for $\calf$). 

Now assume that $m=3$. Thus $|[\UUU,A_0]|=|A_0/Z|=p^{m-2}=p$, and $A_0$ is 
$G$-invariant. By Lemma \ref{GonA} applied to the $G$-action on $A_0$, 
$A_0=C_{A_0}(O^{p'}(G))\times[O^{p'}(G),A_0]$, the first factor is 
$G$-invariant and contained in $Z$, and hence is trivial by (b). So by the 
same lemma, $A_0\cong C_p^2$ and $O^{p'}(G)\cong\SL_2(p)$. If $\alpha\in 
N_{O^{p'}(G)}(\UUU)$ has order $p-1$, then for some generator 
$u\in(\Z/p)^\times$, $\alpha(z)=z^u$ for $z\in{}Z_0$, and $\alpha(z)\in 
z^{u^{-1}}Z_0$ for $z\in{}Z_2\sminus Z_0=A_0\sminus Z_0$. In the notation 
of Lemma \ref{l:A/Z-filter}(b) (applied with $A$ in the role of $V$), 
$tr=u^{-1}$ and $tr^2=u$. So $t=u^{-3}$, and thus $\alpha(a)\in 
a^{u^{-3}}A_0$ for $a\in{}A\sminus A_0$. Since $O^{p'}(G)$ acts trivially 
on $A/A_0$, this implies that $u^{-3}\equiv1$ (mod $p$), so $(p-1)|3$, 
which is impossible. Thus there are no strongly closed subgroups 
in this case, and so $\calf$ is simple.

\smallskip

\noindent\textbf{Step 5: } By Lemma \ref{red->simple}, if $\calf$ is 
realizable, it is isomorphic to the fusion system of a finite simple group. 
(The hypotheses of Lemma \ref{red->simple} hold by the description 
in Step 3 of the strongly $\calf$-closed subgroups.)
Hence by Lemma \ref{list-simp}, it is isomorphic to the fusion system of 
one of the simple groups listed in Table \ref{tbl:type3}.  
\end{proof}

\iffalse
We finish this section with a corollary which states the conditions for the 
existence of a simple fusion system solely in terms of $A$ as a 
$G$-representation. Before stating the corollary, we must define the map 
$\mu$ in this situation.
\fi

Recall that when $\UUU\cong C_p$, then $\F_p\UUU\cong\F_p[X]/(X^p)$, and 
hence the indecomposable $\F_p\UUU$-modules are those of the form 
$\F_p[X]/(X^i)$ for $1\le i\le p$. These are the ``Jordan blocks'' of an 
$\F_p\UUU$-module. 

\begin{Not} \label{n:not3}
Assume that $G$ is a finite group, with $\UUU\in\sylp{G}$ of order $p$. Let $M$ 
be an $\F_pG$-module, and set $Z=C_M(\UUU)$ and $Z_0=Z\cap[\UUU,M]$. 
Assume that $\dim(Z_0)=1$;
equivalently, that $M|_\UUU$ has just one non-trivial Jordan block. 
In this situation, we set 
	\[ G\7=\bigl\{\alpha\in{}N_G(\UUU)\,\big|\,
	[\alpha,Z]\le Z_0 \bigr\}\,,  \]
and define $\mu_A\:G\7\Right2{}\Delta$ by setting $\mu_A(g)=(r,s)$ if 
$\9gu=u^r$ and $g(z)=z^s$ for all $u\in\UUU$ and $z\in Z_0$. 
\end{Not}

Using this notation, we now get as an immediate consequence of Theorem 
\ref{t3:s/a} the following list of necessary and sufficient conditions for 
an $\F_pG$-representation $A$ to give rise to a simple fusion system in the 
way described by the theorem.

\begin{Cor} \label{cor:s/a}
Fix an odd prime $p$, a finite group $G$, and a finite dimensional, 
faithful $\F_pG$-module $A$. 
Fix $\UUU\in\sylp{G}$, and assume that $|\UUU|=p$ and $\UUU\nnsg G$. 
Set $Z=C_A(\UUU)$, $Z_0=Z\cap[\UUU,A]$, and $m=\dim(A/Z)+1$, 
and assume that $m\ge3$. 
Then there is a simple fusion system $\calf$ over a finite $p$-group $S$ 
which contains $A$ as its unique abelian subgroup of index $p$, where 
$G=\autf(A)$, if and only if the following conditions hold:
\begin{enuma}  

%%\item $|N_G(\UUU)/C_G(\UUU)|=p-1$; thus $\Aut_G(\UUU)=\Aut(\UUU)$.

\item $|Z_0|=p$;

\item there are no non-trivial $G$-invariant subgroups of $Z$, 
aside (possibly) from $Z_0$;

\item $[G,A]=A$; and 

\item one of the following holds: either \medskip
\begin{enumerate}[\rm(d.1) ]

\item $\mu_A(G\7)=\Delta$, $G=O^{p'}(G)G\7$, $m\equiv0,-1$ (mod 
$p-1$), and $\dim(A)\le p-1$; or 

\iffalse
\item[\rm(d.1$'$) ] $\mu_A(G\7)=\Delta$, $G=O^{p'}(G)G\7$, $m=p-1$ or $p-2$ 
or $m=p=3$, and $\dim(A)\le p-1$; or 
\fi

\item $\mu_A(G\7)\ge\Delta_{-1}$ and 
$G=O^{p'}(G)\cdot\mu_A^{-1}(\Delta_{-1})$; or 

\item $\mu_A(G\7)\ge\Delta_0$, 
$G=O^{p'}(G)\cdot\mu_A^{-1}(\Delta_{0})$, and $Z_0$ is not 
$G$-invariant.  
\end{enumerate}

\end{enuma}
\end{Cor}

Condition (d.1) in Corollary \ref{cor:s/a} corresponds to the cases in 
Table \ref{tbl:(d)} where $\sigma\in\Fr(Z)$. Since $A$ has exponent $p$, 
this means that $\sigma=1$, and hence that $A|_\UUU$ has no indecomposable 
summand of dimension $p$. (See the definition of $\sigma$ in Theorem 
\ref{t3:s/a}.) So by Lemma \ref{prop:collatedresults}(a), $\dim(A)\le p-1$ 
in these cases.

We will see in Section \ref{sec:repprelims} (Propositions 
\ref{cor=>minimally active} and \ref{prop:collatedresults}) that 
$m=\min(\rk(A),p)$ in the situation of Corollary \ref{cor:s/a}.

We finish the section with some examples which show that proper strongly closed 
subgroups can be found in simple fusion systems of the type constructed in 
Theorem \ref{t3:s/a}.

\begin{Ex} \label{ex:str.cl.}
Fix a prime $p\ge5$. Set $\Gamma=S_p\times\F_p^\times$, and let 
$M\cong\F_p^p$ be the $\F_p\Gamma$-module where $S_p$ acts by 
permuting the coordinates, and where $a\in\F_p^\times$ acts via $a\cdot\Id_M$. 
Fix $\UUU\in\sylp{\Gamma}$. 
\begin{enuma} 

\item Set $A=M$ and $S=A\rtimes\UUU$. Let 
$\mu_A\:N_\Gamma(\UUU)\Right2{}\Delta$ be as in Notation \ref{n:not3}. Set 
$G=O^p(\Gamma)\mu^{-1}(\Delta_{-1})$. By Theorem \ref{t3:s/a} case (d.iii), 
there is a simple fusion system $\calf$ over $S$ such that 
$\autf(A)=\Aut_G(A)$, and such that $\EE\calf=\{A\}\cup\calh_0$ in the 
notation of \ref{n:not2}. Furthermore, the subgroup $A_0\UUU$, where 
$A_0=[\UUU,A]$, is strongly closed in $\calf$. 

\item Set $A=M/C_M(O^{p'}(\Gamma))$ and $S=A\rtimes\UUU$. Thus 
$|A|=p^{p-1}$ and $|S|=p^p$. Let $\mu_A\:N_\Gamma(\UUU)\Right2{}\Delta$ be 
as in Notation \ref{n:not3}. Set $G=O^{p'}(\Gamma)\mu^{-1}(\Delta_{0})$. By 
Theorem \ref{t3:s/a} case (d.iv), there is a simple fusion system $\calf$ 
over $S$ such that $\autf(A)=\Aut_G(A)$, and such that 
$\EE\calf=\{A\}\cup\calb_0$ in the notation of \ref{n:not2}. Furthermore, 
the subgroup $A_0\UUU$, where $A_0=[\UUU,A]$, is strongly closed in 
$\calf$. 

\item Let $A$, $S$, $\mu_A$, and $G$ be as in (b). Fix 
$I\subseteq\{0,1,\dots,p-1\}$ with $|I|\ge2$. By Theorem \ref{t3:s/a} case 
(d.iv), there is a simple fusion system $\calf$ over $S$ such that 
$\autf(A)=\Aut_G(A)$ and $\EE\calf=\{A\}\cup\bigcup_{i\in I}\calb_i$, and 
no proper non-trivial subgroup of $S$ is strongly closed in $\calf$.

\end{enuma}
In all of these cases, $G$ has index $2$ in $\Gamma$, and $\calf$ is exotic by 
Table \ref{tbl:type3} and the theorem.
\end{Ex}

%%\newpage

\section{Representation-theoretic preliminaries}
\label{sec:repprelims}
\renewcommand{\k}{\mathbb{F}_p}
From now on, we restrict attention to the case where the abelian group $A$ 
in Section \ref{s:s/a} has exponent $p$. In other words, we are looking at 
$\F_pG$-modules, for certain finite groups $G$, for which the conditions in 
Corollary \ref{cor:s/a} are satisfied. We begin with some representation 
theory that we will need, in particular the representation theory of groups 
with a Sylow $p$-subgroup of order $p$, which is very well understood. Throughout this section, $p$ is an odd prime.

\begin{Defi} 
Let $\G$ be the class of finite groups whose Sylow $p$-subgroups are not 
normal and have order $p$. Let $\GG$ be the class of all $G\in\G$ such that 
$|\Aut_G(\UUU)|=p-1$ for $\UUU\in\sylp{G}$. 
\end{Defi}

The following notation will be used throughout this section, and in much of 
the rest of the paper.

\begin{Not} \label{n:not4}
When $p$ is odd and $G\in\G$, we 
\begin{itemize} 
\item fix an element $x$ of order $p$ in $G$ and set $\UUU=\gen{x}\in\sylp{G}$; and 
\item set $N=N_G(\UUU)$, $C=C_G(\UUU)$, and $C'=C_G'(\UUU)=O_{p'}(C_G(\UUU))$.
\end{itemize}
\end{Not}

For background to the following discussion about vertices, sources, and Green correspondents of $\k G$-modules, we refer to \cite[Chapter 3]{Benson1}, and especially to Sections 3.10--3.12.

Let $V$ be an indecomposable $\k G$-module. A \emph{vertex} for $V$ is a 
minimal subgroup $P\le G$ such that $V$ is relatively $P$-projective; i.e., 
such that each surjection $W\to V$ which splits $\k P$-linearly is also $\k 
G$-linearly split. This is always a $p$-subgroup of $G$, and is uniquely 
determined up to conjugacy. If $G\in\G$, then since $\UUU\in\sylp{G}$ has 
order $p$, either $V$ is $\k G$-projective and has trivial vertex, or $V$ 
is non-projective and $\UUU$ is a vertex of $V$. Note that $V$ is 
projective if and only if $V|_{\UUU}$ is projective, equivalently, if 
$\UUU$ (or $x$) acts on $V$ with Jordan blocks all of size $p$. 

In general, if $P$ is a vertex of $V$ then the restriction $V|_{N_G(P)}$ of 
$V$ to $N_G(P)$ is the direct sum of an indecomposable $\k N_G(P)$-module 
$W$ with vertex $P$, the \emph{Green correspondent} of $V$, and other 
indecomposable modules with vertices that are contained in intersections 
$P^g\cap N_G(P)$ for $g\notin N_G(P)$, so in particular not equal to $P$ 
and of order at most $|P|$. Thus when $\UUU\in\sylp{G}$ has order $p$ and 
$V$ is not projective, $V|_N\cong W\oplus X$ (recall $N=N_G(\UUU)$), where 
$W$ (the Green correspondent) is indecomposable and non-projective, and 
where $X$ is projective.

The restriction of any non-projective, indecomposable $\k G$-module $V$ to 
$\UUU$ is a sum of free modules and of copies of a fixed $\k \UUU$-module 
$T$, called the \emph{source} of $V$. (In general, $U$ is a sum of conjugates 
of the source, but for cyclic groups conjugate modules are isomorphic.) If 
the source of $V$ is the trivial module (i.e., $V|_\UUU$ is a sum of free 
modules and trivial modules) then $V$ is said to be \emph{trivial source}. 
The source of an indecomposable $\F_pG$-module and the source of its Green 
correspondent are the same. 

As $C$ is normal in $N$, the simple $\k N$-modules restrict to $C$ as a sum of simple modules. However, as $C=\UUU\times C'$, we see that the simple 
$\k C$-modules are just the simple $\k C'$-modules, which are irreducible ordinary characters as $C'$ is a $p'$-group. In particular, all simple $\k N$-modules are trivial source, and indeed these are the only trivial-source $\k N$-modules. Hence if $V$ is a simple, trivial-source $\k G$-module then its Green correspondent is also simple.

The following definition will be useful in our discussion.

%%We now give a few definitions related to our situation for fusion systems, 
%%which will help our discussion.

\begin{defn} \label{d:minact}
When $G\in\G$, an $\k G$-module $M$ is \emph{minimally active} if 
the action of $x$ on $M$ has at most one non-trivial Jordan block (i.e., 
at most one Jordan block with non-trivial action).
\end{defn}

\begin{Lem} \label{l:minact_decomp}
Fix a field $k$ of characteristic $p$,
assume that $G\in\G$, and let $M$ be a minimally active $kG$-module 
upon which $\UUU$ acts non-trivially. 
\begin{enuma} 
\item If $M=M_1\oplus M_2$, where the $M_i$ are $k G$-submodules, then 
$O^{p'}(G)$ acts trivially on exactly one of the $M_i$.

\item $M$ is indecomposable if and only if $C_M(O^{p'}(G))\le[O^{p'}(G),M]$.

\item If $M$ is indecomposable, then $M$ is absolutely indecomposable. If 
$M$ is simple then $M$ is absolutely simple. 

\item We can decompose $M|_N=M_0\oplus M_1$, where $\UUU$ acts 
trivially on $M_0$ and $M_1|_\UUU$ is indecomposable.

\end{enuma} 
\end{Lem} 

\begin{proof} \textbf{(a) } Since $M$ is minimally active, a generator 
$x\in\UUU$ acts with at most one non-trivial Jordan block. So $\UUU$ must 
act trivially on at least one of the $M_i$, and hence $O^{p'}(G)$ acts 
trivially on it. 

\smallskip

\noindent\textbf{(b) } Set $H=O^{p'}(G)$ for short. If $M$ is 
decomposable, then $C_M(H)\nleq[H,M]$ by (a). 

Assume, conversely, that $C_M(H)\nleq[H,M]$. Since $k[G/H]$ is 
semisimple, there is a $kG$-submodule $V_0\le C_M(H)$ such that $V_0\ne0$ and 
$C_M(H)=V_0\oplus(C_M(H)\cap[H,M])$. For the same reason, there is $V_1\le 
M$ such that $V_1\ge[H,M]$ and 
$M/[H,M]=(V_0+[H,M])/[H,M]\oplus(V_1/[H,M])$. Thus 
$M=V_0\oplus V_1$, $V_1\ne0$ since $[H,M]\ne0$ (since $\UUU$ acts 
non-trivially), and thus $M$ is decomposable.

\smallskip

\noindent\textbf{(c) } Set $\4M=\4k\otimes_{k}M$ for short. If $M$ is 
indecomposable, then $C_M(O^{p'}(G))\le[O^{p'}(G),M]$ by (b), so 
$C_{\4M}(O^{p'}(G))\le[O^{p'}(G),\4M]$, and $\4M$ is indecomposable by (b) 
again.

If $M$ is simple, then $\End_{kG}(M)\cong\F_{p^m}$ for some $m\ge1$. Then 
$m\bmid\dim(C_M(\UUU)\cap[\UUU,M])=1$, so $m=1$, and $M$ is absolutely 
simple.

\smallskip

\iffalse
We apply (a) and (b) with $N$ and $\UUU$ in the role 
of $G$ and $O^{p'}(G)$. 
\fi

\noindent\textbf{(d) } If $C_M(\UUU)\le[\UUU,M]$, then $M|_\UUU$ is 
indecomposable by (b), applied with $\UUU$ in the role of $G$, and 
we take $(M_0,M_1)=(0,M)$. Otherwise, $M|_N$ is decomposable by (b), 
this time applied with $N$ in the role of $G$, and $\UUU=O^{p'}(N)$ acts 
trivially on all but one of its indecomposable summands by (a). 
\end{proof}

Minimally active modules are what is needed in the situation of Theorem 
\ref{t3:s/a} and Corollary \ref{cor:s/a}. This is made more precise in the 
next proposition. 

\begin{Prop} \label{cor=>minimally active}
If $A$ is an $\F_pG$-module that satisfies the hypotheses of Corollary 
\ref{cor:s/a}, then $G\in\GG$, and $A$ is minimally 
active and indecomposable.
\end{Prop}

\begin{proof} By Corollary \ref{cor:s/a}(a), 
$\dim(C_A(\UUU)\cap[\UUU,A])=1$, so the action of $x$ has only one 
non-trivial Jordan block. Thus $A$ is minimally active. Also, 
$\UUU\in\sylp{G}$ has order $p$ and $|\Aut_G(\UUU)|=p-1$ by the conditions 
in Corollary \ref{cor:s/a}(d), so $G\in\GG$.

If $A=A_1\oplus A_2$, where the $A_i$ are $\F_pG$-submodules, then by Lemma \ref{l:minact_decomp}, at least one of its direct factors lies in $Z=C_A(\UUU)$ and intersects trivially with $[\UUU,A]$. In the terminology of Corollary \ref{cor:s/a}, this is a non-trivial $G$-invariant subgroup of $Z$ which is not equal to $Z_0$, contradicting point (b) in the corollary. Thus $A$ is indecomposable. 
\end{proof}

We note the following easy lemma.

\begin{lem} \label{inheritedprops}
The property of being minimally active is preserved under taking submodules, 
quotients, dual, tensoring by a $1$-dimensional module, and restricting to 
subgroups in $\G$.
\end{lem}

We will show that of the almost simple groups which lie in $\GG$, very few 
possess minimally active modules. Lemma \ref{inheritedprops} shows 
that if $M$ is a minimally active $\F_pG$-module and $H\le G$, then  $M|_H$ 
is also minimally active. This means that we can argue inductively 
and embed, say, $\SL_{n-1}(q)$ into $\SL_n(q)$. Here the automizers of 
cyclic subgroups are the same, and so if $\SL_{n-1}(q)$ is not in $\GG$ 
then neither is  $\SL_n(q)$.

\medskip

%%We now prove the main result on minimally active modules.

The next result describes some of the basic properties of minimally active modules. 

\begin{Prop}\label{prop:collatedresults}
Let $G\in\G$, and assume Notation \ref{n:not4}. Then the following hold for each 
indecomposable, minimally active $\k G$-module $M$ on which $\UUU$ acts 
faithfully. 
\begin{enuma}

\item Either \medskip
\begin{itemize} 
\item $\dim(M)<p$ and $M|_\UUU$ is indecomposable, or 

\item $\dim(M)=p$, $M|_\UUU$ is indecomposable and free, and $M$ is projective, or

\item $\dim(M)>p$, $M$ is a trivial-source module, and $M|_\UUU$ is the sum 
of one copy of $\F_p\UUU$ and a module with fixed action.
\end{itemize}

\item $M$ is a trivial-source module if and only if $\dim(M)\ge p$.

\item If $\dim(M)\ge p+2$ then $M$ is simple and absolutely simple.

\item  If $M$ is simple and has trivial source, then $M$ is the reduction 
modulo $p$ of a $\Z_p$-lattice in a simple $\Q_pG$-module. Moreover, $\chi_M$ 
extends to an absolutely irreducible ordinary character. 

\item If $\dim(M)>p$, then the Green correspondent $V$ of $M$ is simple.
\end{enuma}
\end{Prop}

\begin{proof} \noindent\textbf{(a) } By Lemma 
\ref{l:minact_decomp}(d), we can write $M|_N=V_0\oplus V_1$, where 
$V_1|_\UUU$ is a non-trivial Jordan block and $V_0$ has fixed action 
of $\UUU$.  (Recall that $N=N_G(\UUU)$.) If $V_0=0$, then $\dim(M)\le 
p$, as $p$ is the largest size of a Jordan block, and $M|_\UUU\cong 
V_1|_\UUU$ is indecomposable. If in addition, $\dim(V_1)=p$, then $V_1\cong 
M|_\UUU$ is projective, and hence $M$ is projective since $\UUU\in\sylp{G}$ 
(cf.\ \cite[Corollary 3.6.9]{Benson1}).

If $V_0\ne0$, then $M$ is not projective, so $M|_N$ is the sum of a 
projective module and the Green correspondent of $M$. Thus $V_1$ is 
projective, hence of rank $p$, and so $\dim(M)>p$. Also, $M$ is a trivial 
source module since $\UUU$ acts trivially on $V_0$.

\iffalse
a sum of free modules and conjugates of the source. Thus $M$ is trivial 
source, and $V_1\cong(\F_p\UUU)^r$ for some $r$, where $r\le1$ since $M$ is 
minimally active and $r>0$ since $G$ acts faithfully. In particular, 
$\dim(M)>p$ in this case.
\fi

\smallskip

\noindent\textbf{(b) } If $\dim(M)>p$, then we are done by (a). If 
$\dim(M)=p$, then $M$ is projective, and hence has trivial vertex and 
trivial source. If $\dim(M)<p$, then $M|_\UUU$ is indecomposable, hence is 
the source of $M$, and has non-trivial action since $M|_\UUU$ is 
assumed to be faithful.

\smallskip

\noindent\textbf{(c) } Assume that $\dim(M)\ge p+2$. By Lemma 
\ref{l:minact_decomp}(d), $M|_N=V_0\oplus V_1$, where $V_1$ is $\k 
N$-projective (and $V_1|_\UUU\cong \F_p\UUU$), and $V_0$ is the Green 
correspondent to $M$ and is an indecomposable (hence simple) $\k 
[N/\UUU]$-module. Note that by the Frattini argument, $G=O^{p'}(G)N$.

Assume first that there is a non-trivial submodule $0\ne M_0<M$ on which $\UUU$ acts trivially. Then $O^{p'}(G)$ acts trivially on $M_0$, and $M_0\le C_M(\UUU)=V_0\oplus C_{V_1}(\UUU)$. Since $V_0$ is $\k N$-irreducible, $\dim(V_0)=\dim(M)-p\ge2$, and $\dim(C_{V_1}(\UUU))=1$, either $M_0\ge V_0$ or $M_0=C_{V_1}(\UUU)$. If $M_0\ge V_0$, then $V_0$ is an $\k G$-submodule of $M$ (recall $G=O^{p'}(G)N$), hence a direct summand of $M$ since $M/V_0$ is $\k G$-projective, which contradicts the indecomposability of $M$.

Thus $M_0=C_{V_1}(\UUU)$. As $M/M_0$ does not satisfy any of the 
conditions in (a), (since $\UUU$ acts faithfully) it must be 
decomposable. By the Krull--Schmidt theorem, each proper direct sum 
decomposition of $(M/M_0)|_N$ has a summand isomorphic to $V_0$, so $M/M_0$ 
contains a direct summand whose restriction to $N$ is isomorphic to $V_0$, 
and which (by an argument similar to that in the last paragraph) must be 
equal to the image of $V_0$ in $M/M_0$. Hence $V_0\oplus 
C_{V_1}(\UUU)=C_M(\UUU)$ is an $\k G$-submodule, and we just showed that 
this is impossible.

Now assume that $1\ne M_0<M$ is an arbitrary non-trivial proper submodule. 
We just showed that $\UUU$ acts non-trivially on $M_0$, and by a similar 
argument applied to the dual $M^*$, $\UUU$ also acts non-trivially on 
$M/M_0$. If either of $M_0$ or $M/M_0$ is decomposable, then it has a 
direct factor on which $\UUU$ acts trivially (Lemma \ref{l:minact_decomp}), 
which contradicts the fact that $M$ has no submodules or quotients on which $\UUU$ acts trivially. So each of $M_0$ and $M/M_0$ has one of the 
forms listed in (a). Since $M$ is minimally active, $M|_\UUU$ and 
$(M/M_0)|_\UUU$ must both be indecomposable, so 
$\dim(C_M(\UUU))\le\dim(C_{M_0}(\UUU))+\dim(C_{M/M_0}(\UUU))=2$. But we 
already saw that $\dim(C_M(\UUU))=\dim(V_0)+1=\dim(M)-p+1\ge3$, so this is 
impossible. Absolute simplicity now comes from Lemma 
\ref{l:minact_decomp}(c).

\smallskip

\noindent\textbf{(d) } That $M$ is the mod $p$ reduction of a 
$\Z_p$-lattice $\5M$ is a general property of all trivial-source modules 
(see \cite[Corollary 3.11.4(i)]{Benson1}). If $\Q_p\5M$ is not simple, then 
it contains a non-trivial proper submodule $0\ne W<\Q_p\5M$, and the mod 
$p$ reduction of $W\cap\5M$ is a proper $\k G$-submodule of $M$, 
contradicting the assumption that $M$ is simple. Since this also holds 
for $\F_{p^n}$ for $n\ge1$, we get absolute irreducibility since $M$ is 
absolutely simple by Lemma \ref{l:minact_decomp}(c).

By Lemma \ref{l:minact_decomp}(c), $\4\F_p\otimes_{\F_p}M$ is a simple 
$\4\F_pG$-module. Hence by a similar argument, $K\otimes_{\Q_p}(\Q_p\5M)$ 
is a simple $KG$-module for each finite extension $K\supset \Q_p$ by roots 
of unity. So $\4\Q_p\otimes_{\Q_p}(\Q_p\5M)$ is simple, and the character 
of $\Q_p\5M$ is irreducible when regarded as a complex character of $G$.

\smallskip

\noindent\textbf{(e) } By (a), $M$ has trivial source, so its Green correspondent is an indecomposable $\k [N/\UUU]$-module, hence irreducible since $N/\UUU$ has order prime to $p$.
\end{proof}

The next lemma will be useful when showing that certain extensions of 
minimally active modules are again minimally active.

\begin{Lem} \label{l:ext-minact}
Fix $G\in\G$. Assume that $V$ is an indecomposable $\F_pG$-module of dimension 
at most $p+1$, and that $V$ has a nonzero submodule or quotient module 
which is minimally active. Then $V$ is also minimally active.
\end{Lem}

\begin{proof} Assume $V$ is not minimally active. Then $V|_{N}$ is  
the Green correspondent of $V$, and is thus indecomposable. By \cite[p. 
42]{Alperin} and since $\UUU$ is normal and cyclic in $N$, 
$V|_{N}$ is uniserial in the sense of \cite[p. 26]{Alperin}. 
(Alperin always assumes we are working over an algebraically closed field, 
but this proof does not use that.) In particular, the socle $C_V(\UUU)$ and 
the top $V/[\UUU,V]$ are both irreducible $N/\UUU$-modules, and they have 
rank at least $2$ since $V$ is not minimally active. But this is 
impossible: if $W<V$ is a submodule, then $C_W(\UUU)<C_V(\UUU)$ has rank 
$1$ if $W$ is minimally active, while the image of $W$ in $V/[\UUU,V]$ has 
corank $1$ if $V/W$ is minimally active. 
\end{proof}

%%by Fitting's lemma (see \cite[Lemma 1.1.4]{Benson1}), 

\iffalse
\begin{proof} Let $W<V$ be a submodule such that both $W$ and $V/W$ are 
simple and minimally active. Since neither $W$ nor $V/W$ is projective, 
each has dimension at most $p-1$. If $V$ is 
not minimally active, then $V|_{N_G(\UUU)}$ must be the Green correspondent of 
$V$, and is thus indecomposable. So $C_V(\UUU)$ is an irreducible 
$N_G(\UUU)/\UUU$-module of rank at least $2$, which is impossible since $C_W(\UUU)$ 
is a submodule of $C_V(\UUU)$ of rank $1$.
\end{proof}
\fi

As a consequence of Proposition \ref{prop:collatedresults}, if $M$ is a 
simple, minimally active module, then either $\dim(M)\leq p$, or $M$ has as 
Green correspondent the reduction modulo $p$ of an irreducible ordinary 
character of $G$, whose minimal degrees are known in the case where $G$ is 
quasisimple. The next result will help us to classify such modules.

\begin{prop}\label{prop:dimlt2p-1} 
The following hold for each faithful, indecomposable, minimally active 
$\k G$-module $M$.
\begin{enuma} 
%%\item If $\dim(M)>p$, then $\dim(M)-p$ divides $|N/\UUU|$.

\item Suppose that $\dim(M)>p$, and set $a=\dim(M)-p$. Then 
\smallskip

\begin{itemize} 
\item $a$ divides $|N/\UUU|$;

\item if $N/\UUU$ is abelian, then $a=1$;

\item if $C$ is abelian, then $a$ divides $|N/C|$; and 

\item if $C > \UUU$, then $a \le |C/\UUU|-1$.

\end{itemize}

%%\item If $C/\UUU$ is abelian, and $\dim(M)>p$, then $\dim(M)-p$ divides $p-1$. 

\item If $O^{p'}(G)=\gen{x,y}$ for some $x,y\in G$, where 
$|x|=|y|=p$ or $|x|=2$ and $|y|=p$, then for each central extension 
$\til{G}$ of $G$ of degree prime to $p$, the dimension of each minimally 
active faithful indecomposable $\k\til{G}$-module is at most $2p-2$.

\end{enuma}
\end{prop}

\begin{proof} \textbf{(a) } If $\dim(M)=p+1$, then all four statements 
hold. So we may assume that $\dim(M)>p+1$. In particular, by Proposition 
\ref{prop:collatedresults}(c), $M$ is absolutely simple, and the Green 
correspondent $W$ of $M$ is an absolutely simple module for $N$, and for 
the $p'$-group $N/\UUU$ since $\UUU$ acts trivially on $W$. Thus $\chi_W$ 
is an ordinary irreducible character for $N/\UUU$. Also, 
$\dim(W)=\dim(M)-p=a$. 

In particular, $\dim(W)$ divides $|N/\UUU|$ (see, e.g., \cite[Theorem 
3.11]{isaacs}), and $\dim(W)=1$ if $N/\UUU$ is abelian. If $C$ is abelian, 
then we apply a theorem of Ito (see \cite[Theorem 6.15]{isaacs}) to get 
that $\dim(W)$ divides $|N/C|$. This proves the first three statements. 

Set $\4W=\4\F_p\otimes_{\k}W$, where $\4\F_p\supseteq\k$ is the algebraic closure. 
By Clifford theory (see \cite[Theorem III.2.12]{feit}), $\4W|_C\cong 
e\cdot[\bigoplus_{i=1}^kW_i]$, where $e\ge1$, and where $W_1,\dots,W_k$ are 
pairwise distinct irreducible $\4\F_p[C/\UUU]$-modules which form one orbit under 
the $N/C$-action on the set of all irreducible representations. Also, $e=1$ 
since $N/C$ is cyclic (see \cite[Theorem III.2.14]{feit}), so $\4W|_C$ is a sum 
of distinct irreducible representations. Since $C/\UUU\ne1$, $N/C$ cannot 
act transitively on the set $\Irr(C/\UUU)$, and hence 
$\dim_{\k}(W)<\dim_{\4\F_p}(\4\F_p[C/\UUU])=|C/\UUU|$.

\iffalse
since $N/C$ is cyclic, we may apply Theorems III.2.12 and III.2.14 
from \cite{feit}, which state the following: suppose that $W'$ is a simple 
submodule of $W|_C$, and let $T$ be the set of all $t\in N$ such that ${}^t 
W'=W'$, the inertia group of $W$. The module $W$ has dimension $|N:T|\cdot 
\dim(W')$, and $|N:T|$ is the number of conjugates of $W'$ by elements of 
$N$, which are pairwise non-isomorphic modules for $C$. Since the trivial 
module has inertia group $N$, not every simple module for $C$ can lie in 
the $N$-orbit, and so $|N:T|\cdot \dim(W')$ is at most the sum of all 
irreducible character degrees of $C/\UUU$ minus $1$ for the trivial, thus 
$|C/\UUU|$, as claimed.
\fi

\smallskip

\textbf{(b) } By Lemma \ref{l:minact_decomp}(b), $M$ is $\k 
O^{p'}(G)$-indecomposable if it is $\k G$-indecomposable. So we can assume 
$G=O^{p'}(G)$.

If $G=\gen{x,y}$ where $x$ and $y$ have order $p$, then since 
a minimally active module $M$ has at most one non-trivial Jordan block, 
$C_M(x)$ and $C_M(y)$ have codimension at most $p-1$. This means that 
$C_M(x)\cap C_M(y)=C_M(O^{p'}(G))$ has codimension at most $2p-2$. Since 
$M$ is simple, $\dim(M)\leq 2p-2$, as needed. 

If $G=\gen{t,y}$ where $|t|=2$ and $|y|=p$, then $\gen{\9ty,y}$ is normal 
of index at most $2$ in $G$, hence equal to $G$ since $G=O^{p'}(G)$. If 
$\til{G}$ is a central extension of $G$ of degree prime to $p$, where 
$G=\gen{x,y}$ and $|x|=|y|=p$, then $x$ and $y$ lift to 
$\til{x},\til{y}\in\til{G}$ of order $p$, and $\gen{\til{x},\til{y}}\ge 
O^{p'}(\til{G})$. So in both cases, we are back in the first situation.
\end{proof}

Proposition \ref{prop:dimlt2p-1}(b) is useful because both sporadic and alternating groups are known to be generated by two elements of order $p$ whenever the Sylow $p$-subgroup is cyclic (see \cite{Craven2015un} for sporadic groups), and this can be checked for any individual groups we might encounter. In general, it appears that with few exceptions this is always true for finite simple groups, and such a statement is currently under investigation by the first author. 

As a generalization of part of Proposition \ref{prop:dimlt2p-1}, we get the following condition, which is useful for bounding the size of minimally active modules for a group in terms of minimally active modules for a subgroup.
 
\begin{prop}\label{prop:inactiverestriction} 
Let $H$ be a subgroup of $G$ such that $O^{p'}(H)=H$. Suppose that 
$O^{p'}(G)\le\gen{H_1,\dots,H_n}$ for some set $H_1,\dots,H_n$ of $n$ 
conjugates of $H$. Let $s$ be the maximal dimension of an indecomposable, 
faithful, minimally active $\F_pH$-module. If $M$ is an indecomposable, 
faithful, minimally active $\F_pG$-module such that
$C_M(O^{p'}(G))=0$, then $\dim(M)\le ns$. In particular, if 
$C_H(\UUU)$ is abelian, then $\dim(M)<2np$.
\end{prop}

\begin{proof} The proof is similar to that of Proposition 
\ref{prop:dimlt2p-1}(b). If $M$ is an indecomposable, faithful, 
minimally active $\F_pG$-module, then by Lemma 
\ref{l:minact_decomp}(a), for each $1\le i\le n$, the restriction of $M$ 
to $H_i$ must have a summand $N_i$ of codimension at most $s$ on which 
$H_i$ acts trivially. The intersection of the $N_i$ has codimension at most 
$ns$, and is contained in the proper submodule $C_M(O^{p'}(G))$. Since 
$C_M(O^{p'}(G))=0$ we have $\dim(M)\leq ns$, as claimed. 
\end{proof}

Having found minimally active, simple modules, we would like to know whether there are minimally active, indecomposable modules built from them. In almost every case we will see that if $V$ is minimally active then $\dim(V)>(p+1)/2$, so that Proposition \ref{prop:collatedresults}(c) eliminates any extensions between non-trivial modules. However, this leaves open the possibility that $V$ is minimally active for $G$ and $H^1(G,V)\neq 0$, so that $V$ has a minimally active extension with the trivial module; for example, the permutation module for the symmetric group $S_p$ is minimally active and has a trivial submodule and a trivial quotient. Of course, by Proposition \ref{prop:collatedresults}(c) again, $\dim(V)$ must be at most $p$ for this to work. The next lemma deals with this case, when $V$ is self dual.

\begin{lem}\label{l:cohomology}
Let $G$ be a group in $\GG$, and let $V$ be a self-dual, simple, minimally 
active module with $\dim(V)\leq p$. If $H^1(G,V)\neq 0$ then $\dim(V)=p-2$.
\end{lem}

\begin{proof} Set $m=\dim(V)$. If $m=p$, then $V$ is projective and 
$H^1(G,V)=0$, so $m<p$. Since $0\ne H^1(G;V)\cong\Ext^1_{\F_pG}(\F_p,V)$, 
there is an indecomposable extension $M$ of $\F_p$ by $V$ (with $V$ as 
submodule). Since restriction sends $H^1(G;V)$ injectively into 
$H^1(\UUU;V)$, $M|_\UUU$ is indecomposable and consists of one Jordan 
block. 
Set $M_0=[\UUU,M]=V$, and $M_i=[\UUU,M_{i-1}]$ for each $i\ge1$; then 
$\dim(M_i)=m-i$ for each $0\le i\le m$. 

Fix $g\in N$ such that $\9gu=u^r$ where $r$ generates $(\Z/p)^\times$. 
By Lemma \ref{l:A/Z-filter}(b), and since $g$ acts trivially on $M/M_0=M/V$, 
$g$ acts on each $M_{i-1}/M_{i}$ by multiplication by $r^i$. In particular, it 
acts on $M_0/M_1$ by multiplication by $r$ and on $M_{m-1}$ by 
multiplication by $r^m$, and since $V$ is self-dual, 
$r^m\equiv r^{-1}$ (mod $p$). Thus $m+1\equiv0$ (mod $p-1$), and since 
$0<m<p$, we have $m=p-2$. 
\end{proof}

\iffalse
The restriction of $M$ to $N=N_G(\UUU)$ is the sum of the Green 
correspondent of $M$ and a projective module, the latter of dimension a 
multiple of $p$. Since $\dim(M)=m+1\leq p$, $M|_N$ is indecomposable, and 
$M$ is minimally active.

with a trivial quotient, thus a quotient of the projective cover of the 
trivial module, $P(1)$. This module is the permutation module of $N$ on the 
cosets of a Hall $p'$-subgroup $H$ of $N$ (which exists since $N$ is 
$p$-solvable).

As $M$ is faithful, so is $M\downarrow_N$, but since $M\downarrow_N$ lies in the principal block of $N$, $O_{p'}(N)$ acts trivially on $N$ by \cite[Corollary 2.13]{feit}. Thus $O_{p'}(N)=1$, and so $N=\UUU\rtimes E$ for $E=\langle g\rangle$ a cyclic group of order $p-1$. Since the structure of $N$ is straightforward, it is easy to understand the action of $g$ on the $p$ radical layers of $P(1)$; these are in order
\[ 1,\lambda,\lambda^2,\lambda^3,\dots,\lambda^{p-2},1,\]
where $\lambda$ is a primitive $(p-1)$th root of unity. (Compare with the actions of $g$ and $h$ in Section \ref{sec:gl2p}.) As $M\downarrow_N$ is a quotient of $P(1)$, the eigenvalues of $g$ on $M$ are
\[ 1,\lambda,\lambda^2,\dots,\lambda^n,\]
where $\dim(V)=n$; finally since $V$ is self dual, if $\lambda$ is an eigenvalue of the action of $g$ on $V$, so is $\lambda^{-1}$, and hence $n=p-2$, as claimed.
\fi

In the next lemma, we give some tools for handling some of the other 
properties listed in Corollary \ref{cor:s/a}, especially those involving 
the homomorphism $\mu_V\:G\7\Right2{}\Delta$ of Notation \ref{n:not3}. For 
any finite abelian group $M$, let $\Aut\scal(M)$ be the group of scalar 
automorphisms: those of the form $(x\mapsto kx)$ for $k$ prime to $|M|$.

\begin{Lem} \label{l:moreprops}
Fix $G\in\G$, and let $V$ be a faithful, minimally active $\F_pG$-module. 
Let $\UUU\in\sylp{G}$, $G\7\le N_G(\UUU)$, and $\mu_V\:G\7\Right2{}\Delta$ 
be as in Notation \ref{n:not3}.
\begin{enuma} 

\item The homomorphism $\mu_V$ sends $G\7/\UUU$ injectively into $\Delta$. 

\item If $G\in\GG$, $\dim(V)\le p$, and $G\ge\Aut\scal(V)$, then 
$G\7=N_{G}(\UUU)$ and $\mu_V(G\7)=\Delta$.

\item Assume that $\dim(V)\ge p$, and set $Z=C_V(\UUU)$. Then for each $g\in 
N_G(\UUU){\sminus}C_G(\UUU)$, $\chi_V(g)=\chi_Z(g)$.

\end{enuma}
\end{Lem}

\begin{proof} Let $V_0=C_V(\UUU)[\UUU,V]$ and 
$[\UUU,V]=W_1>W_2>\cdots>W_m=0$ be as in Lemma \ref{l:A/Z-filter}.

\smallskip

\textbf{(a) } This is essentially Lemma \ref{l3:s/a}(a), but we give 
another proof here. Assume that $g\in G\7$ has order prime to $p$, and 
$\mu_V(g)=(1,1)$. By Lemma \ref{l:A/Z-filter}(b), $g$ acts via the identity 
on $V/V_0$, and on $W_i/W_{i+1}$ for each $1\le i\le m-1$ ($r=t=1$ in the 
notation of the lemma). By definition of $G\7$, $g$ acts via the identity 
on $C_V(\UUU)/W_{m-1}$, and hence also acts via the identity on 
$C_V(\UUU)[\UUU,V]/[\UUU,V]=V_0/W_1$. So by Lemma \ref{mod-Fr}, $g$ acts 
trivially on $V$, and hence $g=1$ since $G$ acts faithfully.

Thus $\Ker(\mu_V|_{G\7})\le\UUU$, and the opposite inclusion is clear.

%%Lemma \ref{l:A/Z-filter}

\smallskip

\noindent\textbf{(b) } Since $\dim(V)\le p$, $V|_{\UUU}$ contains only one 
Jordan block. So $\dim(C_V(\UUU))=1$, $Z=Z_0$ in the notation of Corollary 
\ref{cor:s/a}, and $G\7=N_{G}(\UUU)$. 

Let $r\in(\Z/p)^\times$ be a generator, and let $\psi_r\in\Aut\scal(V)\le 
N_G(\UUU)$ be the automorphism $(a\mapsto a^r)$. Then 
$\mu_V(\psi_r)=(1,r)$.

By assumption, there is $g\in N_G(\UUU)$ such that $\9gu=u^r$ for each 
$u\in\UUU$. Then $\mu_V(g)=(r,s)$ for some $s$, and thus 
$\mu_V(G\7)\ge\gen{\mu_V(g),\mu_V(\psi_r)}=\Gen{(1,r),(r,s)}=\Delta$.

\smallskip

\noindent\textbf{(c) } Let $W_0\le V$ and $V_0\le Z$ be such that $W_0$ is a 
non-trivial Jordan block for the action of $\UUU$ and $V=W_0\oplus V_0$. Since 
$\dim(V)\ge p$, $\dim(W_0)=p$ by Proposition \ref{prop:collatedresults}(a). 
Hence $\dim(V/Z)=\dim(W_0/C_{W_0}(\UUU))=p-1$.

Let $[\UUU,V]=[\UUU,W_0]=W_1>W_2>\cdots>W_p=0$ be as defined above. By Lemma 
\ref{l:A/Z-filter}(a), $|W_i/W_{i+1}|=p$ for each $1\le i\le p-1$, and also 
for $i=0$ since $W_0/W_1\cong W_0Z/W_1Z=V/V_0$. 

%%In particular, $W_{p-1}=Z_0$ and $W_p=0$. 

Fix $g\in N_G(\UUU)$. Let $r,t\in(\Z/p)^\times$ be such that $\9gu=u^r$ for 
each $u\in\UUU$, and $g$ acts on $W_0/W_1\cong V/V_0$ via multiplication by 
$t$. Then by Lemma \ref{l:A/Z-filter}(b), for each $1\le i\le p-1$, $g$ 
acts on $W_i/W_{i+1}$ via multiplication by $tr^i$. The action of $g$ on 
$W_0/W_{p-1}\cong V/Z$ thus has eigenvalues $t,tr,tr^2,\dots,tr^{p-2}$. 
Hence $\chi_V(g)-\chi_Z(g)=\sum_{i=0}^{p-2}\psi(tr^i)$ for some embedding 
$\psi\:(\Z/p)^\times\Right2{}\C^\times$. If $g\notin C_G(\UUU)$, then 
$\psi(r)\ne1$, the sum of this geometric series is zero since 
$\psi(r)^{p-1}=1$, and hence $\chi_V(g)=\chi_Z(g)$. 
\end{proof}

We now summarize the tools which will be used to compute $\mu_V(G\7)$ 
in later sections. 

\begin{Prop} \label{p:mu(Gvee)}
Fix $G\in\GG$, and let $V$ be a faithful, minimally active 
$\F_pG$-module such that $G\ge\Aut\scal(V)$. Let $\UUU\in\sylp{G}$, 
$G^\vee$, and $\mu_V$ be as in Notation \ref{n:not3}.
\begin{enuma} 
\item If $\dim(V)\le p$, then $\mu_V(G\7)=\Delta$ and 
$C_G(\UUU)=\UUU\times\Aut\scal(V)$.
\item If $\dim(V)>p$, then 
$|\mu_V(G\7)|=|G\7/\UUU|\le\frac{|N_G(\UUU)/\UUU|}{p-1}$, with equality if 
$\dim(V)=p+1$.

\item If $\dim(V)=p+1$, and $g\in N_G(\UUU)$ has order prime to $p$ and is 
such that $c_g$ generates $\Aut(\UUU)$, then the following hold.
\begin{enumi} 
%%\item If $\chi_V(g)=2\omega$ for some root of unity $\omega$, then 
%%$\mu_V(G\7)\ge\Delta_0$. 
\item If $|\chi_V(g)|=2$, then $\mu_V(G\7)\ge\Delta_0$. 
\item If $\chi_V(g)=0$, then $\mu_V(G\7)\ge\Delta_{(p-1)/2}$. 
\item If $|\chi_V(g)|=1$, then 
$\mu_V(G\7)\ge\Delta_{\gee(p-1)/3}$ for some $\gee=\pm1$.
%%\item If $\chi_V(g)$ is a root of unity, then 
%%$\mu_V(G\7)\ge\Delta_{\gee(p-1)/3}$ for $\gee=\pm1$.
\end{enumi}

\end{enuma}
\end{Prop}

\begin{proof} \textbf{(a) } The first statement was shown in Lemma  
\ref{l:moreprops}(b). Hence $|N_G(\UUU)/\UUU|=|G\7/\UUU|=|\Delta|=(p-1)^2$ 
by Lemma \ref{l:moreprops}(b,a), so $|C_G(\UUU)|=p(p-1)$, and the 
centralizer is as described.

\smallskip

\noindent\textbf{(b) } The first equality holds since $\Ker(\mu_V)=\UUU$ by 
Lemma \ref{l:moreprops}(a). Also, $\Aut\scal(V)\le C_G(\UUU)$ and 
intersects trivially with $G\7$, so $G\7$ has index at least $p-1$ in 
$N_G(\UUU)$. If $\dim(V)=p+1$, then $\dim(Z/Z_0)=1$ (in the notation of 
\ref{n:not3}), so for each $g\in N_G(\UUU)$, the coset $g\Aut\scal(V)$ 
contains a unique element in $G\7$. Hence $G\7$ has index exactly $p-1$ in 
$N_G(\UUU)$ in this case.

\smallskip

\noindent\textbf{(c) } Now assume that $\dim(V)=p+1$, and fix $g\in{}N_G(\UUU)$ 
of order $p-1$ such that $c_g$ generates $\Aut(\UUU)$. Let $g'\in 
g\Aut\scal(V)$ be the unique element in $G\7$. Let $r\in(\Z/p)^\times$ be 
such that $\9gu=\9{g'}u=u^r$ for each $u\in\UUU$. 

Set $Z=C_V(\UUU)$ and $Z_0=Z\cap[\UUU,V]$. Thus $\dim(Z)=2$, $\dim(Z_0)=1$, 
and $\chi_V(g)=\chi_Z(g)$ by Lemma \ref{l:moreprops}(c). Also, for 
some choice of monomorphism $\psi\:(\Z/p)^\times\Right2{}\C^\times$, 
$\chi_Z(g)=\psi(s)+\psi(t)$, where $g$ acts on $Z_0$ via multiplication 
by $s$ and on $Z/Z_0$ via multiplication by $t$ ($s,t\in(\Z/p)^\times$). 
Since $g'$ acts on $Z/Z_0$ via the identity by definition of $G\7$, it acts 
on $Z_0$ via multiplication by $st^{-1}$, and hence 
$\mu_V(g')=(r,st^{-1})$. 

Recall that $r$ generates $(\Z/p)^\times$ by the assumption on $g$. If 
$|\chi_V(g)|=|\psi(s)+\psi(t)|=2$, then $s=t$, so $\mu_V(g')=(r,1)$ 
generates $\Delta_0$. If $\chi_V(g)=0$, then $s=-t$, so $\mu_V(g')=(r,-1)$ 
generates $\Delta_{(p-1)/2}$. If $|\chi_V(g)|=1$, then 
$\psi(st^{-1})=\psi(s)/\psi(t)$ must be a primitive cube root of unity, and 
hence $\mu_V(g')=(r,st^{-1})$ generates $\Delta_{(p-1)/3}$ or 
$\Delta_{-(p-1)/3}$. 
\end{proof}

%%\newpage

\section{Representations which occur in simple fusion systems: a summary}
\label{s:reps-summary}

\newcommand{\Dta}{\Delta}
\newcommand{\pnote}[1]{{}^{\textbf{\textup{(#1)}}}}
\newcommand{\bnote}[1]{{}^{{\textup{[#1]}}}}

In this section, we present a summary of the rest of paper, by outlining 
our classification of all possible pairs $(G,A)$ satisfying parts (a) to 
(d) in Theorem \ref{t3:s/a} or Corollary \ref{cor:s/a}. From now on, these 
will be regarded as $\F_pG$-modules with additive structure (as opposed to 
the multiplicative group structure on $A$), and will be denoted by $V$ to 
emphasize this. Throughout this section and the next, we assume the 
classification of finite simple groups.

Table \ref{tbl:alm.simple.reps} is an attempt at tabulating this information, but its notation requires explanation. When the image of $G$ in $\PGL(V)$ is almost simple, the group $G_0=F^*(G)$ is listed in the second column, and a group $\4G\le N_{\GL(V)}(G_0)$ such that $\4G/G_0$ is a $p'$-subgroup of maximal order in $N_{\GL(V)}(G_0)/G_0$ is listed in the fourth column. In all cases, $N_{\GL(V)}(G_0)/G_0$ is solvable, so the choice of $\4G$ is unique up to conjugacy, and we can assume that $G_0\le G\le\4G$. (In almost all cases, $\4G=N_{\GL(V)}(G_0)$.)

When the image of $G$ in $\PGL(V)$ is not almost simple, the second column is left blank, $G$ is contained in the group $\4G$ in the fourth column, and we provide more information on the possibilities for $G$ later in this section. In all cases, the third column lists the possible dimensions of the minimally active module $V$ that becomes the subgroup $A$ in the saturated fusion system. The fifth and sixth columns list the images $\mu_V(G_0\7)\le\mu_V(\4G\7)\le\Delta$, and the final column gives the information as to whether this representation leads to a realizable fusion system ($R$) and/or an exotic fusion system ($E$) with superscripts indicating the number of such systems. In some cases we are necessarily vague when considering whole collections of possible groups $G$.

\iffalse
Recall that $\Aut\scal(V)$ denotes the group of scalar automorphisms of a 
vector space $V$. 
\fi

\begin{thm} \label{t:alm.simple.reps}
Assume that $G\in\GG$, and let $V$ be a minimally active, faithful, 
indecomposable $\F_pG$-module which satisfies the hypotheses of Corollary 
\ref{cor:s/a}. Then either 
\begin{enuma} 

\item the image of $G$ in $\PGL(V)$ is not almost simple, and $G\le\4G$ 
with the given action on $V$ for one of the pairs $(\4G,V)$ listed in Table 
\ref{tbl:alm.simple.reps} with no entry $G_0$; or 

\item the image of $G$ in $\PGL(V)$ is almost simple, and $G_0\le G\le\4G$ 
for one of the triples $(G_0,\4G,V)$ listed in Table 
\ref{tbl:alm.simple.reps}. 

\end{enuma}
When $G_0\cong\SL_2(p)$ or $A_p$ for $p\ge5$, more precise descriptions of 
the modules are given in Propositions \ref{SL2(p)-summary} and 
\ref{Ap-summary}, respectively. 
\end{thm}

\begin{table}[ht]
\begin{small} 
\[ \renewcommand{\arraystretch}{1.4}\addtolength{\arraycolsep}{-2pt}
\begin{array}{|c|c|c|c|c|c|c|} \hline\hline
p & G_0 & \dim(V) & \4G & \mu_V(\4G\7) & \mu_V(G_0\7) & E,R
\\\hline \hline
 &  \SL_2(p) \textup{ or } \PSL_2(p) & 
\Pile{1}{3\le n\le p ~\hfill\pnote{\ref{SL2(p)-summary}}}{\textup{socle of dim. $i$}} 
& \GL_2(p) \textup{ or}\hfill & \Dta & 
\{(u^2,u^{i-1})\}
& ER \\ \cline{3-3}\cline{5-7}
\halfup[3]{p} & \up[1]{(p\ge5)} & \textup{$2/(p{-}1)$}~\pnote{\ref{SL2(p)-summary}}
& \up[1]{\PGL_2(p)\times C_{p-1}} & \Delta_{-1} & \frac12\Delta_{-1} 
& E
\\\hline
p  & A_p \quad(p\ge5) & [1]/p{-}2/[1]~\pnote{\ref{Ap-summary}} 
& S_p\times(p{-}1) & \Dta & 
\Sm{$\frac12\Delta_0$ or $\frac12\Delta_{-1}$} & ER \\\hline
p & A_{p+1}\quad(p\ge5) & p & S_{p+1}\times(p{-}1) & \Dta & \frac12\Delta_0 
& ER \\\hline
p & A_n ~\Sm{$(p{+}2\le n\le 2p{-}1)$} & n-1 & S_n\times(p{-}1) & \Delta_0 
& \frac12\Delta_0 & R \\\hline
p & \textbf{---} & n~\pnote{\ref{p:reduce2simple-x}(b)} & C_{p-1}\wr S_n~\Sm{$(n\ge 
p)$} & \Dta & \textbf{---} & ER \\\hline\hline
3 & \textbf{---} & 2/2~\pnote{\ref{p:reduce2simple-x}(c)} & \GL_2(3) & \Delta_1 & \textbf{---} & E \\\hline\hline
5 & 2\cdot A_6 & 4 & 4\circ 2\cdot S_6  & \Dta  & \Delta_{1/2} & E \\\hline
5 & \textbf{---} & 4~\pnote{\ref{p:reduce2simple-x}(d)} & (C_4\circ2^{1+4}).S_6 & \Dta & \textbf{---} & ER \\\hline
5 & \PSp_4(3)=W(E_6)' & 6 & W(E_6)\times4 & \Delta_0.2 & \frac12\Delta_0 & R \\\hline
5 & \Sp_6(2)=W(E_7)' & 7 & G_0\times4 & \Delta_0 & \frac12\Delta_0 & R \\\hline
\hline
7 & 2\cdot A_7 & 4 & 2\cdot S_7\times3 & \Dta  & \Delta_{3/2} & E \\\hline
7 & 6\cdot\PSL_3(4) & 6 & G_0.2_1 & \Dta  & 
\F_p^{\times2}\times\F_p^\times & E \\\hline
7 & 6_1\cdot\PSU_4(3) & 6 & G_0.2_2 & \Dta  & 
\F_p^{\times2}\times\F_p^\times & E \\\hline
7 & \PSU_3(3) & 6 & G_0.2\times6 & \Dta & \frac12\Delta_1 & E \\\hline
7 & \PSU_3(3) & 7 & G_0.2\times6 & \Dta & \frac12\Delta_0 & E \\\hline
7 & \SL_2(8) & 7 & G_0:3\times6 & \Dta  & \frac13\Delta_1 & E \\\hline
7 & \Sp_6(2)=W(E_7)' & 7 & G_0\times6 & \Dta & \Delta_3 & ER \\\hline
7 & 2\cdot\Omega_8^+(2)=W(E_8)' & 8 & W(E_8)\times3 & \Delta_0.2 & 
\Delta_3 & R \\\hline
\hline
11 & J_1 & 7 & G_0\times10 & \Dta  & \Delta_3 & E \\\hline
11 & \PSU_5(2) & 10 & G_0.2\times10 & \Dta  & \frac12\Delta_2 & E \\\hline
11 & 2\cdot M_{12} & 10\bnote2 & G_0.2\times5 & \Dta &
\Sm{$\Delta_{1/2} \,,~ \Delta_{7/2}$} & E \\\hline
11 & 2\cdot M_{22} & 10\bnote2 & G_0.2\times5 & \Dta & 
\Sm{$\Delta_{1/2} \,,~ \Delta_{7/2}$} & E \\\hline
\hline
13 & \PSU_3(4) & 12 & G_0.4\times12 & \Dta  & \frac13\Delta_1 & E \\\hline
\hline
\multicolumn{7}{c}{
\parbox{\short}{\textup{\vphantom{$^|$} In the third column, a superscript 
in brackets shows the number of distinct representations of $G_0$ of the 
given dimension (not counting them as distinct if they differ by an 
automorphism of $G_0$). A superscript $\pnote{\ref{s:reps-summary}.x}$ 
means that these representations (and/or the groups) are described more 
precisely in Proposition \ref{s:reps-summary}.x.} When $k,\ell$ are 
relatively prime, we set 
$\Delta_{k/\ell}=\{(u^\ell,u^k)\,|\,u\in\F_p^\times\}$ (always cyclic of 
order $p-1$).} }
\end{array}
\]
\end{small}
\caption{Groups in $\GG$ with minimally active modules of dimension at least 
$3$ which appear in reduced fusion systems}
\label{tbl:alm.simple.reps}
\end{table}

\begin{proof} Let $V$ be a minimally active, faithful, indecomposable 
$\F_pG$-module with $\dim(V)\ge3$. When the image of $G$ in $\PGL(V)$ is 
not almost simple, the possibilities for $G$ and $V$ are listed in cases 
(b)--(e) of Proposition \ref{p:reduce2simple}. When the image of $G$ in 
$\PGL(V)$ is almost simple, and $G_0=F^*(G)$, the possible pairs $(G_0,V)$ 
are determined in Sections 6--11: $\SL_2(p)$ is handled in Proposition 
\ref{p:gl2p} (the only groups of Lie type in defining characteristic $p$ 
which appear), sporadic groups in Proposition \ref{p:sporadic}, alternating 
groups in Proposition \ref{p:alternating}, linear, unitary and symplectic 
groups in Propositions \ref{prop:finishedsln}, \ref{prop:finishedsun} and 
\ref{prop:finishedspn}, and orthogonal groups in Proposition \ref{lem:redorth}. 
Finally, the exceptional groups are dealt with in Proposition 
\ref{prop:determineexc}. Together, these show that in all cases, $(G,V)$ 
appears as an entry in one of the two Tables \ref{tbl:alm.simple.reps} or 
\ref{tbl:reps_without_r.f.s.}, as described in (a) or (b). 

If $V$ is a simple $\F_pG_0$-module, then it is absolutely simple by Lemma 
\ref{l:minact_decomp}(c), and hence $C_{\GL(V)}(G_0)=\Aut\scal(V)$ (the 
group of scalar automorphisms). So $\4G=O^p(N_{\GL(V)}(G_0))$ is an 
extension of $G_0\Aut\scal(V)$ by outer automorphisms of $G_0$ and hence 
can be determined from the tables of modular characters. When $V$ is 
indecomposable but not simple, $\4G$ has a similar form but is not in 
general the full normalizer of $G_0$ in $\GL(V)$ (see, e.g., the second-to-last 
sentence in Proposition \ref{p:gl2p}.

In both tables, $\mu_V(\4G\7)=\Delta$ by Proposition \ref{p:mu(Gvee)}(a) 
whenever $\dim(V)\le p$. When $G_0\cong A_n$ ($n\ge p+2$) or $G\le 
C_{p-1}\wr S_n$ ($n>p$), $\mu_V(\4G\7)$ is easily calculated using the 
definition, and when $G_0\cong\SL_2(p)$ and $\dim(V)=p+1$, it is calculated 
in Section \ref{sec:gl2p}. 
When $\4G\cong(C_3\times2^{1+6}_+).S_8$, $\mu_V(\4G\7)$ is determined in 
Lemma \ref{l:mu(2^1+6)}. In all other cases where $\dim(V)=p+1$, 
$\mu_V(\4G\7)$ can be determined using points (b) and (c) in Proposition 
\ref{p:mu(Gvee)}. This leaves only the representation where $p=5$, 
$G_0\cong\Sp_6(2)$, and $\dim(V)=7$: $|\mu_V(\4G\7)|=p-1$ by Proposition 
\ref{p:mu(Gvee)}(b), and the structure of fusion in $E_7(q)$ (when 
$v_5(q-1)=1$) together with Table \ref{tbl:(d)} show that it must be 
$\Delta_0$. 

The determination of $\mu_V(G_0\7)$ is slightly more delicate than that of 
$\mu_V(\4G\7)$. But in almost all cases, this can either be done either 
directly using the definitions, or with the help of character tables and 
Green correspondence, or by examining $V|_H$ for some $H<G_0$ isomorphic to 
$\SL_2(p)$ or $\PSL_2(p)$. 

Finally, with the help of Table \ref{tbl:(d)}, we determine which of these 
representations appear in simple fusion systems (i.e., satisfy the 
conditions in Corollary \ref{cor:s/a} for some choice of $G\le\4G$ 
containing $G_0$), and in those cases we determine $\mu_V(\4G\7)$ and give 
$\dim(V)$. Among those fusion systems, Table \ref{tbl:type3} tells us which 
are realizable.
\end{proof}

\begin{table}[ht]
\[ \renewcommand{\arraystretch}{1.25}
\begin{array}{|c|c|c|c|c|} \hline\hline
p & G_0 & \dim(V) & \4G & \mu_V(\4G\7) \\\hline \hline
p & \textit{(P)SL}_2(p)  & \Pile{0.9}{\textup{type 
$j/i$}\hfill\pnote{\ref{SL2(p)-summary}}}{\Sm{$(i{+}j=p{+}1,~j\ne2)$}} & 
\Pile{1}{\GL_2(p)\textup{ or}\hfill}{\PGL_2(p)\times C_{p-1}} & 
\Delta_{i-1} \\\hline
7 & \textbf{---} & 8~\pnote{\ref{p:reduce2simple-x}(e)} & (C_3\times2^{1+6}_+).S_8 & \Delta_3 \\\hline
7 & \SL_2(8) & 8 & G_0:3\times6 & \Delta_3 \\\hline
7 & 2\cdot\Sp_6(2) & 8 & G_0\times3 & \Delta_3 \\\hline
7 & 2\cdot A_7 & 4/4 & 2\cdot S_7\times3 & \Delta_3 \\\hline
7 & 2\cdot A_8 & 8 & 2\cdot S_8\times3 & \Delta_3 \\\hline
7 & 2\cdot A_9 & 8 & G_0\times3 & \Delta_3 \\\hline
11 & 2\cdot M_{12} & 12 & G_0.2\times5 & \Delta_5 \\\hline
13 & \lie2B2(8) & 14\bnote2 & G_0:3\times12 & \Delta_{\pm4} \\\hline
13 & G_2(3) & 14 & G_0:2\times12 & \Delta_6 \\\hline
17 & \Sp_4(4) & 18 & G_0.4\times16 & \Delta_8 \\\hline\hline
\end{array}
\]
\caption{Groups in $\GG$ with minimally active modules of dimension at least 
$3$ which do not satisfy the hypotheses of Corollary \ref{cor:s/a}}
\label{tbl:reps_without_r.f.s.}
\end{table}

The first few rows of Tables \ref{tbl:alm.simple.reps} and 
\ref{tbl:reps_without_r.f.s.} require the most explanation. The next 
proposition concerns the first row, and is basically a restatement of 
Proposition \ref{p:gl2p}.

\begin{Prop} \label{SL2(p)-summary}
Let $G_0$, $V$, and $\4G$ be as in Theorem \ref{t:alm.simple.reps}, and 
assume that $G_0=\SL_2(p)$. The indecomposable minimally active 
$\F_pG_0$-modules with faithful action of $G_0$ or of $G_0/Z(G_0)$ 
are described as follows (where unique always means up to isomorphism).
\begin{enuma} 

\item A unique simple $\F_pG_0$-module $V_i$ of dimension $i$ for each 
$2\le i\le p$. Each simple $\F_pG_0$-module is 
isomorphic to $V_i$ for some $1\le i\le p$, where $V_1=\F_p$ is the trivial 
module, $V_2$ is the natural module for $G_0=\SL_2(p)$, and 
$V_i=\Sym^{i-1}(V_2)$ for $3\le i\le p$ (the $(i{-}1)$-st symmetric power). 

%%, and $C_G(V_i)=\Gen{(-I)^{i-1}}$. 

\item A unique $(p-1)$-dimensional indecomposable module $V_{j,i}$ of type 
$V_j/V_i$ for each $1\le i,j\le p-2$ such that $i+j=p-1$. 

%%Here, $V_1=\F_p$ 
%%denotes the trivial module.

\item A unique $p$-dimensional indecomposable module $V_{1,p-2,1}$ of type 
$V_1/V_{p-2}/V_1$.

\item A unique $(p+1)$-dimensional indecomposable module $V_{j,i}$ of type 
$V_j/V_i$ for each $2\le i,j\le p-1$ such that $i+j=p+1$.
\end{enuma}
If the simple composition factors of $V$ are even dimensional, 
then $G_0=\SL_2(p)$ acts faithfully on $V$, and $\4G\cong\GL_2(p)$. If the 
simple composition factors of $V$ are odd dimensional, then the action of 
$G_0$ factors through $\PSL_2(p)$, and $\4G\cong\PGL_2(p)\times C_{p-1}$. 
Also, $\4G=N_{\Aut(V)}(G_0)$ except when $V\cong V_{i,i}$ for $i=(p\pm1)/2$ 
or $V\cong V_{1,p-2,1}$, in which cases $\4G$ has index $p$ in the 
normalizer. 
\end{Prop}

We now consider the case where $G_0$ is $A_p$.

\begin{Prop} \label{Ap-summary}
Let $G_0$, $V$, and $\4G$ be as in Theorem \ref{t:alm.simple.reps}, and 
assume that $G_0=A_p$. Then there is a unique faithful, simple, minimally active 
$\F_pG_0$-module $W$, of dimension $p-2$. There are unique non-simple 
indecomposable minimally active modules of each of the types $\F_p/W$, 
$W/\F_p$, and $\F_p/W/\F_p$, where the last is the permutation module. 
Also, $\4G\cong S_p\times C_{p-1}$ in all cases, and is equal to 
$N_{\Aut(V)}(G_0)$ except when $V$ is the permutation module, in which case 
$\4G$ has index $p$ in the normalizer. All of these representations give 
rise to simple fusion systems via Theorem \ref{t3:s/a} (for some choice of 
$G$), but only $W$ itself gives rise to simple, realizable fusion 
systems.
\end{Prop}
\begin{proof} See Proposition \ref{p:alternating} for a determination of the minimally active modules.
\end{proof}

It remains to describe the minimally active, indecomposable 
$\F_pG$-modules, when $G\in\GG$ and the image of 
$G$ in the projective group is not almost simple. 

\begin{Prop} \label{p:reduce2simple-x}
Assume that $G\in\GG$. Let $V$ be a minimally 
active, faithful, indecomposable $\F_pG$-module, and set $n=\dim(V)$. Then 
one of the following holds.
\begin{enuma}  
\item The image of $G$ in $\PGL(V)$ is almost simple, and $p$ divides the order of its socle.

\item $G\le{}C_{p-1}\wr S_n$ ($n\ge p$) acts as a group of monomial 
matrices on $V\cong(\F_p)^n$. More precisely, if we set $K=O_{p'}(G)$, then 
	\[ K = \bigl\{(a_1,\dots,a_n)\in (C_{p-1})^n \,\big|\, 
	a_1^t=a_2^t=\dots=a_n^t, ~ (a_1\cdots a_n,a_1^t)\in R \bigr\} \]
for some $1\ne t|(p-1)$ and some $R\le C_{p-1}\times C_{p-1}$;
and one of the following holds:  \smallskip

\begin{itemize}
\item $n=p$ and $G/K\cong C_p\rtimes C_{p-1}$;
\item $n=p+1$ and $G/K\cong\PGL_2(p)$;
\item $p\leq n\leq 2p-1$ and $G/K\cong S_n$;
\item $p+2\leq n\leq 2p-1$ and $G/K\cong A_n$.
\end{itemize}
\smallskip

\noindent Also, $V|_K$ splits as a direct sum of pairwise non-isomorphic 
$1$-dimensional $\F_pK$-modules which are permuted $2$-transitively by $G/K$.

\item $p=3$, $G\cong 2^{1+2}_-.S_3\cong Q_8.S_3\cong \GL_2(3)$, and either $n=2$ and $V$ 
is simple, or $n=4$ and $V$ is non-simple of type $2/2$.

\item $p=5$, $n=4$, $O_{5'}(G)\cong C_4\circ2^{1+4}$ or $2^{1+4}_-$, and 
$G/O_{5'}(G)\cong S_6\cong\Sp_4(2)$, $S_5\cong\SO_4^-(2)$, or $C_5\rtimes 
C_4$. (Note that $C_4\circ 2^{1+4}_+\cong C_4\circ 2^{1+4}_-$. If $O_{5'}(G)\cong 2^{1+4}_-$ then $G/O_{5'}(G)$ is not $S_6$.)

%%and $G\le(C_4\circ2^{1+4}).S_6$.

\item $p=7$, $n=8$, $O_{7'}(G)\cong C_3\times2^{1+6}_+$ or $2^{1+6}_+$, and 
$G/O_{7'}(G)\cong S_8\cong\SO_6^+(2)$, $S_7$, $\PGL_2(7)$, or $C_7\rtimes 
C_6$.

%%and $G\le(C_3\times2^{1+6}_+).S_8$.

\end{enuma}
%%If $G$ is equal to the group in (e), then $\mu_V(G\7)=\Delta_3$.
\end{Prop}

\begin{proof} This will be proved in Section \ref{sec:nonsimple}, as 
Propositions \ref{p:wreathcase} and \ref{p:reduce2simple}.
\end{proof}

In order to describe more explicitly how to get from Table 
\ref{tbl:alm.simple.reps} to actual fusion systems, we list a few cases in 
more detail in Table \ref{tbl:examples}. The first four columns in the 
table correspond to information given in Table \ref{tbl:alm.simple.reps}, 
while the last four columns consist of separate rows for the different 
cases (corresponding to cases (i)--(iv) in Table \ref{tbl:(d)}). In the 
last column, $E$ (or $E^+$) means that there is one (or more than one) exotic 
fusion system of this type; otherwise, a group is given which realizes it.

For example, when $p\ge5$, $G_0\cong\PSL_2(p)$, and $V$ is the simple 
$(p{-}2)$-dimensional $\F_pG_0$-module, we have $m=\dim(V)\equiv-1$ (mod 
$p-1$) and $\mu_V(\4G\7)=\Delta$. Hence there are three families of simple 
fusion systems which arise in this way, corresponding to the three cases 
(ii), (iii), (iv) in Table \ref{tbl:(d)}. The fusion systems of types (ii) 
and (iv) are unique by Theorem \ref{t3:s/a}, while in case (iii), 
$\EE\calf\sminus\{A\}$ can be any union of $\calh_i$'s. In all cases, 
$\autf(A)=G\ge G_0$ is determined (as a subgroup of $\4G$) by the third 
column in Table \ref{tbl:(d)}. For $P\in\EE\calf\cap(\calh\cup\calb)$, 
$\autf(P)$ is determined by Lemma \ref{l4:s/a}. By Table \ref{tbl:type3}, 
all of these fusion systems are exotic if $p\ge7$. If $p=5$, then the 
fusion systems of type (iii) are exotic, that of type (iv) is realized by 
$\Sp_4(5)$, and that of type (ii) is realized by $Co_1$.

In contrast, when $p=7$, $G_0\cong2\cdot\Omega_8^+(2)\cong W(E_8)'$, and $V$ is 
the simple $8$-dimensional $\F_pG_0$-representation, then 
$m=\min(\dim(V),p)=7$, $\mu_V(\4G\7)=\Delta_0.2=\Delta_0\Delta_3$, and 
$\mu_V(G_0\7)=\Delta_3$. By Table \ref{tbl:(d)}, and since 
$m\not\equiv0,-1$ (mod $p$), any simple fusion system which realizes 
$(G,V)$ must be in case (iii) or (iv), and it cannot be in case (iii) since 
$\mu_V(\4G\7)\ngeq\Delta_{-1}$. So there is exactly one simple fusion 
system $\calf$ of this type, of type (iv) and hence with 
$\EE\calf=\{A\}\cup\calb_0$. Also, $\autf(A)=G$ must contain $G_0$ with 
index $2$ by the condition in the third column of Table \ref{tbl:(d)}, so 
$G\cong W(E_8)$. By Table \ref{tbl:type3}, $\calf$ is realized by $E_8(q)$ 
for any $q$ such that $v_7(q-1)=1$.

\begin{table}[ht]
\begin{small} 
\[ \renewcommand{\arraystretch}{1.5} \addtolength{\arraycolsep}{-2pt}
\begin{array}{|c|c|c|c|c|c|c|c|} \hline
p & G_0 & \dim(V) & \mu_V(G_0\7) & G & \mu_V(G\7) & \EE0 &
\textup{Group}/E \\\hline \hline
 & & \up[-1]{\Sm{$4\le n\le p{-}3$}} & & 
\Pile{1}{\GL_2(p)/Z}{\Sm{$|Z|=(p{-}1,n{-}1)$}} & 
\Delta_0\Delta_{\frac{n-1}2} & \calb_0 & E \\\cline{5-8}
\up[3.0]{p} & \up[3.0]{\SL_2(p)} & \up[1]{\Sm{\textup{$n$ even}}} & 
\up[3]{\Delta_{\frac{n-1}2}} & \Pile{1}{\frac1\gee\GL_2(p)}{\Sm{$\gee=(p{-}1,n{+}1)$}} 
& \Delta_{-1}\Delta_{\frac{n-1}2} & \calh_0 & E \\\hline
%%%%%%%%
 &  &  &  &  G_0.2\times(p{-}1) & \Delta & 
\calb_0\cup\calh_* & \pile{}{Co_1~(p=5)}{E~(p\ge7)} \\\cline{5-8}
p & \halfup{\PSL_2(p)} & p-2 & \frac12\Delta_{-1} & \PGL_2(p) & \Delta_{-1} 
& \bigcup\calh_i & E \\\cline{5-8}
 & \halfup{(p\ge5)}  &  &  & G_0.2\times \frac{p-1}2 & \Delta_0.(\frac{p-1}2) & 
\calb_0 & \pile{}{\Sp_4(5)~(p=5)}{E~(p\ge7)} \\\hline
%%%%%%%%
%% & \halfup[-0.5]{\PSL_2(p)} &  &  &  G_0.2\times \frac{p-1}2 
%%& \Delta_0.(\frac{p-1}2) & \calb_0 & \Sp_4(p) \\\cline{5-8}
%%\halfup{p} & \halfup[0.5]{(p\ge7)} & \halfup3 & \halfup{\frac12\Delta_1} & 
%%\PGL_2(p)~(?) & \Delta_{-1}.(?) & \calh_0 & E \\\hline
%%%%%%%%
p & \Pile{1.1}{A_n}{\Sm{$(p{+}2\le n\le 2p{-}1)$}} & n-1 
& \frac12\Delta_0 & S_n & \Delta_0 
& \calb_0 & \PSL_n(q) \\\hline
 &  &  &  & 2\cdot S_7 & \Delta_0\Delta_3  & \calb_0 & E 
\\\cline{5-8}
\halfup7 & \halfup{2\cdot A_7} & \halfup4 & \halfup{\Delta_{3/2}} 
& 2\cdot S_7\times3 & \Dta  & \calh_0 & E \\\hline
 &  &  &  & G_0.2_1 & \Dta  & 
\calb_0\cup\calh_* & E \\\cline{5-8}
7 & 6\cdot\PSL_3(4) & 6 & \frac12\Delta & G_0.2_1 & \Dta  & \bigcup\calh_i
& E \\\cline{5-8}
 &  &  & & G_0.2_1 & \Dta  & \calb_0 & E \\\hline
 &  & & & G_0\times2 & \Delta_0\Delta_3 & \calb_0 & E_7(q) \\\cline{5-8}
\halfup7 & \halfup{\Sp_6(2)=W(E_7)'} & \halfup7 & \halfup{\Delta_3} 
& G_0\times3 & \Delta_{-1}\Delta_3 & \calh_0 & 
E \\\hline
7 & \Sm{$2\cdot\Omega_8^+(2)\cong W(E_8)'$} & 8 & \Delta_3 & \Sm{$G.2\cong 
W(E_8)$} & \Delta_0.2 & 
\calb_0 & E_8(q) \\\hline
 &  & 10 &  & G_0.2\times5 & \Delta & \calh_0\cup\calb_* & E \\\cline{5-8}
11 & 2\cdot M_{12} & (V|_{\SL_2(11)}\hfill & \Delta_{1/2} & G_0.2\times5 
& \Delta & \calh_0 & E \\\cline{5-8}
&  & \hfill\halfup[0.5]{\textup{type }8/2)} &  & G_0.2\times5 & \Delta & \bigcup\calb_i & E \\\hline
 &  &  &  & G_0.4\times12  
& \Delta & \calh_0\cup\calb_* & E \\\cline{5-8}
13 & \PSU_3(4) & 12 & \frac13\Delta_1 & G_0.4\times3 
& \Delta_{-1}\Delta_4 & \calh_0 & E \\\cline{5-8}
 &  &  & & G_0.4\times3 & \Delta_0\Delta_4  
& \bigcup\calb_i & E \\\hline
\multicolumn{8}{c}{
\parbox{\short}{\textup{\vphantom{$^|$} In all cases, $q$ is a prime power 
such that $v_p(q-1)=1$, and $\Delta_{k/\ell}$ is as in Table 
\ref{tbl:alm.simple.reps}.} }}
\end{array}
\]
\end{small}
\caption{Some examples}
\label{tbl:examples}
\end{table}

\section{Reduction to almost simple groups}
\label{sec:nonsimple}

In this section, we analyse the possibilities for $(G,V)$ as in 
Theorem \ref{t:alm.simple.reps} when the image of $G$ in $\PGL(V)$ is not 
almost simple, by using Aschbacher's classification of the maximal 
subgroups of $\GL_n(p)$. The almost simple cases will be handled in the 
later sections.

\iffalse
Throughout most of the rest of this paper, we analyse the possibilities for 
$G$ as in Theorem \ref{t:alm.simple.reps} when the image of $G$ in 
$\PGL(V)$ is almost simple. But first, we describe the other possibilities 
for $G$, by using Aschbacher's classification 
of the maximal subgroups of $\GL_n(p)$. 
\fi

Before proving a general result, we look at representations of $G\le 
C_{p-1}\wr S_n$ on $\F_p^n$, acting via monomial matrices. Two lemmas are 
first needed.

\begin{Lem} \label{l:prim.action}
Assume that $p$ is an odd prime and $H$ is a finite group with a Sylow $p$-subgroup $\UUU$ of order $p$ such that $|N_H(\UUU)/C_H(\UUU)|=p-1$. Assume also that 
$H$ acts faithfully and transitively on a set $\Omega$ in such a way that each $x\in H$ of 
order $p$ acts via a $p$-cycle. Then $H$ acts primitively and $2$-transitively on 
$\Omega$, and one of the following holds:
\begin{enuma}
\item $|\Omega|=p$ and $H\cong C_p\rtimes C_{p-1}$;
\item $|\Omega|=p+1$ and $H\cong\PGL_2(p)$;
\item $p\leq |\Omega|\leq 2p-1$ and $H=S_\Omega$;

\item $p+2\leq |\Omega|\leq 2p-1$ and $H=A_\Omega$; or

\item $p=7$, $|\Omega|=9$, and $H\cong\SSL_2(8)$.

\end{enuma}
\end{Lem}

\begin{proof} Fix $\UUU\in\sylp{H}$ and $1\ne x\in\UUU$. We first show that 
$H$ is primitive on $\Omega$. If $\Sigma$ is a block of a system of 
imprimitivity for the action of $H$ on $\Omega$ (thus $|\Sigma|>1$), 
then $x$ must stabilize $\Sigma$, as otherwise $x$ must move at least 
$p|\Sigma|>p$ points, contradicting our assumption. Choose $\Sigma$ so 
that $x$ acts non-trivially on $\Sigma$ and trivially on 
$\Omega\sminus\Sigma$. Since $H$ acts transitively on any system of 
imprimitivity, there exists $h\in H$ such that $x^h$ acts non-trivially on 
a different block $\Sigma^h$ but trivially on $\Sigma$, $\langle 
x,x^h\rangle$ is a subgroup of order $p^2$ in $H$, contradicting the 
assumption that $\UUU\in\sylp{G}$. Thus $H$ acts primitively on $\Omega$.

We now appeal to the classification of primitive permutation groups 
containing a $p$-cycle, as listed, for example, in \cite{Zieschang}. 
By that theorem, one of the following holds, where $n=|\Omega|$.
\begin{enuma}
\item $n=p$, and $H$ is a subgroup of $\AGL_1(p)=C_p\rtimes C_{p-1}$;

\item $n=p$, $p=\frac{q^d-1}{q-1}$ for some prime power $q$ and $d\geq 2$, and 
$\PSL_d(q)\le H\leq \PGGL_d(q)$;

\item $n=p+1$, $p=2^d-1$ is a Mersenne prime, and $\AGL_1(2^d)\leq H\leq 
\AGL_d(2)$;

\item $n=p+1$, $H\cong\PSL_2(p)$ or $H\cong\PGL_2(p)$; 

\item $n=p=11$, and $H=\PSL_2(11)$ or $H=M_{11}$, or $n=p=23$ and 
$H=M_{23}$;

\item $n=p+1=12$, and $H=M_{11}$ or $H=M_{12}$, or $n=p+1=24$ and 
$H=M_{24}$;

\item $n=p+2$, $p=2^d-1$ a Mersenne prime, and $H\cong\PSL_2(2^d)$ or 
$\PSSL_2(2^d)$; or

\item $H$ is $A_n$ or $S_n$, and $p\leq n\leq 2p-1$. 
\end{enuma}

Thus the structure of $H$ is very tightly controlled. Since 
$|\Aut_H(\UUU)|=p-1$, we can be even more restrictive: (a) can occur 
only with $H=C_p\rtimes C_{p-1}$, and (d) can occur only with 
$H\cong\PGL_2(p)$. In case (b), if $q=r^k$ where $r$ is prime, then 
$p-1=|\Aut_H(\UUU)|\le kd$, so $(q^d-1)/(q-1)\le kd+1$, which gives 
$(p,q,d)=(3,2,2)$ or $(5,4,2)$, and $H\cong S_3$ or $S_5$. In case (c), 
where $|\Aut_H(\UUU)|\le d$, we get $2^d-1=p=d+1$ and hence $d=2$, $p=3$, 
and $H\cong S_4$. For (e) and (f), $|\Aut_H(\UUU)|=(p-1)/2$, so these do 
not occur. For (g), $|\Aut_H(\UUU)|=2$ or $2d$, so this occurs only when 
$(p,d)=(3,2)$ or $(7,3)$, and $H\cong\PSL_2(4)\cong A_5$, $\PSSL_2(4)\cong 
S_5$, or $\PSSL_2(8)$. Of course (h) can occur, but not when 
$H=A_p$ or $H=A_{p+1}$. By inspection, all of these actions are 
2-transitive. 
\end{proof}

\iffalse
and $p=$ For (b) and (c), $|\Aut_H(\UUU)|=d=p-1$, and so we have 
$(q^d-1)/(q-1)\le d+1$ or $2^d-1\le d+1$, so in both cases $d=2$ and $p=3$, 
yielding $H\cong S_3$ or $S_4$.  
\fi

\begin{Lem} \label{l:aug.id}
Fix a prime $p$, and a subgroup $H\le S_p$ such that $H\in\GG$. 
Fix $m>1$ prime to $p$, and regard $(\Z/m)^p$ as a 
$\Z/m[H]$-module via the inclusion $H\le S_p$. Let $M\le(\Z/m)^p$ be a 
submodule, and assume that $\UUU$ acts non-trivially on $M/qM$ for each 
prime $q\mid m$. Then $M\ge I$, where 
	\[ I = \bigl\{ (x_1,\dots,x_p) \,\big|\, 
	x_i\in\Z/m, ~ x_1+\dots+x_p=0 \bigr\} \,. \]
\end{Lem}

%%there is $(x_1,\dots,x_p)\in M$ such that 
%%$x_1+\dots+x_p=0$ and $\gen{x_1,\dots,x_p}=\Z/m$.

\begin{proof} Fix $\UUU\in\sylp{H}$. Since $(m,p)=1$, 
$M|_\UUU=C_M(\UUU)\oplus(M\cap I)$, so $\UUU$ acts non-trivially on $(M\cap 
I)/q(M\cap I)$ for each prime $q\mid m$. So upon replacing $M$ by $M\cap 
I$, we can assume that $M\le I$.

Assume first that $m$ is a prime; thus $\Z/m\cong\F_m$. As an 
$\F_m\UUU$-module, $(\F_m)^p$ factors as a product of irreducible modules 
of which one is 1-dimensional with trivial action, and the others are 
permuted transitively by the action of $N_H(\UUU)/\UUU$ (since 
$|\Aut_H(\UUU)|=p-1$). Hence each $\F_mH$-submodule of $\F_m^p$ either has 
trivial action of $\UUU$, or contains all of the factors which have 
non-trivial action and thus contains their sum $I$. Since the action of 
$\UUU$ on $M$ is non-trivial by assumption, we have $M\ge I$.

Now assume that $m=q^a$ where $q$ is prime and $a>1$. By assumption, 
$M/qM$ has non-trivial action of $\UUU$, and we just showed that this 
implies that $I\le M+(q\Z/m\Z)^p$. Hence there is $\xx=(x_1,\dots,x_p)\in 
M$ such that $x_1\equiv1$ (mod $q)$, $x_2\equiv-1$ (mod $q$), and 
$x_i\equiv0$ (mod $q$) for each $i\ge2$. Then $q^{a-1}\xx$ and its 
$\UUU$-translates generate $q^{a-1}I$, so $q^{a-1}I\le M$. Since 
$q^{a-2}\xx\equiv(q^{a-2},-q^{a-2},0,\dots,0)$ (mod $q^{a-1}I$) 
(recall that $M\le I$), we have $(q^{a-2},-q^{a-2},0,\dots,0)\in M$, 
and hence $q^{a-2}I\le M$. Upon continuing in this way, we get $I\le M$.

If $m$ is not a prime power, the result now follows upon splitting $M$ and 
$I$ as products of their Sylow subgroups.
\end{proof}

We are now ready to describe the groups and modules which appear in case 
(b) of Proposition \ref{p:reduce2simple-x}. 

\begin{Prop} \label{p:wreathcase}
Set $\4G=C_{p-1}\wr S_n$ ($n\ge p$), and let $V\cong\F_p^n$ be the natural 
$\F_pG$-module where $\4G$ acts via monomial matrices. Let 
$G\le\4G$ be such that $G\in\GG$ and $V|_G$ is a simple, minimally 
active module. Assume also that the image of $G$ in $\PGL(V)$ is not almost 
simple. Let $\4K\nsg\4G$ be the subgroup of $\4G$ acting via diagonal 
matrices ($\4K\cong C_{p-1}^n$), and set $K=G\cap\4K$. Then for 
some $1< t\mid(p-1)$ and some $R\le C_{p-1}\times C_{p-1}$, 
	\[ K = \bigl\{(a_1,\dots,a_n)\in (C_{p-1})^n \,\big|\, 
	a_1^t=a_2^t=\dots=a_n^t, ~ (a_1\cdots a_n,a_1^t)\in R \bigr\} \,; \]
and one of the following holds:
\begin{enuma}
\item $n=p$ and $G/K\cong C_p\rtimes C_{p-1}$;
\item $n=p+1$ and $G/K\cong\PGL_2(p)$;
\item $p\leq n\leq 2p-1$ and $G/K\cong S_n$;
\item $p+2\leq n\leq 2p-1$ and $G/K\cong A_n$.
\item $p=7$, $n=9$, and $G/K\cong\SSL_2(8)$.
\end{enuma}
\end{Prop}

\begin{proof} Set $H=G/K$ for short, with its natural action on 
$\Omega=\{1,2,\dots,n\}$. If $H$ is intransitive on $\Omega$, 
then $\Omega=\Omega_1\cup\Omega_2$ with $\Omega_i$ being $H$-invariant. 
Then $V=\F_p\Omega_1\oplus\F_p\Omega_2$, a sum of $\F_pG$-submodules, which 
is impossible since $V$ is assumed to be indecomposable. Thus the action of 
$H$ on $\Omega$ is transitive.

Fix $\UUU\in\sylp{G}$. By assumption, $|\UUU|=p$, $\UUU\nnsg G$, and 
$|\Aut_G(\UUU)|=p-1$. Each $1\neq x\in\UUU$ acts as a single $p$-cycle on 
$\Omega$, since $x$ has a single non-trivial Jordan block on $M$. So by 
Lemma \ref{l:prim.action}, $H$ (as a subgroup of $S_n$) is one of the 
groups listed in (a)--(d). In particular, $H$ acts 2-transitively on 
$\Omega$.

Set $B=\{a_1a_2^{-1}\,|\,(a_1,\dots,a_n)\in K\}\le C_{p-1}$. Since $H$ acts 
2-transitively on $\Omega$, any other pair of distinct elements of $\Omega$ 
defines the same subgroup. Set 
	\[ \4K_B = \bigl\{ (a_1,\dots,a_n)\in(C_{p-1})^n \,\big|\, 
	a_1\equiv\cdots\equiv a_n \pmod{B} \bigr\} \,. \]
Thus $K\le \4K_B$. Define
	\[ \Phi\: \4K_B \Right4{} C_{p-1}\times(C_{p-1}/B) 
	\quad\textup{by setting}\quad \Phi(a_1,\dots,a_n)=(a_1\cdots a_n, 
	a_1B)\,. \]

If $B=1$, then each element of $K$ acts via multiplication by scalars, so 
either $\UUU\nsg G$ or the image of $G$ in $\PGL(V)$ is almost simple. 
Since we assume that neither of these holds, $B\ne1$. 

Fix a generator $u\in\UUU$. Without loss of generality, we can assume that 
elements in $\Omega$ are arranged so that $u=(1\,2\,\dots\,p)$. Let $H_0$ 
be the image of $N_H(\UUU)$ in $S_p$ via restriction of its action. 
Choose $\xa=(a_1,\dots,a_n)$ such that $a_1a_2^{-1}$ generates $B$, and set 
$\xb=\xa^{-1}u(\xa)$. Thus $\xb=(b_1,\dots,b_p,1,\dots,1)$, where $b_i\in 
B$ for $1\le i\le p$, $b_1b_2\dots b_p=1$, and $b_2$ generates $B$. Since 
$|B|$ is prime to $p$, the sequence $b_1,\dots,b_p$ is not constant 
modulo $q$ for any $q\bmid|B|$. We 
regard $\xb$ as an element of $B^p\cap K$, which is a subgroup invariant 
under the action of $H_0$. By Lemma \ref{l:aug.id}, $K$ contains the 
group of all $(b_1,\dots,b_p,1,\dots,1)$ such that $b_i\in B$ and 
$b_1\cdots b_p=1$. Since the action of $H$ on $\Omega$ is $2$-transitive, 
it now follows that 
	\[ K \ge \bigl\{ (a_1,\dots,a_n)\in B^n \,\big|\, a_1\cdots	a_n=1 \bigr\} = \Ker(\Phi) \,. \]

Set $R=\Im(\Phi)$. Thus $K=\{\xa\in\4K_B\,|\,\Phi(\xa)\in R\}$, and this 
translates to the description of $K$ given above (where $t=|B|$).
\end{proof}

The possibilities in Table \ref{tbl:alm.simple.reps} where there is no 
$G_0$, corresponding to point (a) in Theorem \ref{t:alm.simple.reps}, arise 
in an analysis of subgroups of $\GL_n(p)$ which lie in $\GG$  
and whose action on the associated module is minimally active. Using 
Aschbacher's classification of the maximal subgroups of $\GL_n(p)$, we 
restrict the options for $\4G$, leaving either almost simple groups or a 
few other possibilities. The next proposition performs that reduction.

\begin{Prop} \label{p:reduce2simple}
Assume that $G\in\GG$, let $V$ be a minimally active, 
faithful, indecomposable $\F_pG$-module, and set $n=\dim(V)$. Then one of 
the following holds.
\begin{enuma}  
\item The image of $G$ in $\PGL(V)$ is almost simple, and $p\bmid|F^*(G)|$.

\item $G\le{}C_{p-1}\wr S_n$ ($n\ge p$) and acts as a group of monomial matrices.

\item $p=3$, $G\cong 2^{1+2}_-.S_3\cong Q_8.S_3\cong \GL_2(3)$, and either $n=2$ and $V$ 
is simple, or $n=4$ and $V$ is non-simple of type $2/2$.

\item $p=5$, $n=4$, $O_{5'}(G)\cong C_4\circ2^{1+4}$ or $2^{1+4}_-$, and 
$G/O_{5'}(G)\cong S_6\cong\Sp_4(2)$, $S_5\cong\SO_4^-(2)$, or $C_5\rtimes 
C_4$. (Note that $C_4\circ 2^{1+4}_+\cong C_4\circ 2^{1+4}_-$. If $O_{5'}(G)\cong 2^{1+4}_-$ then $G/O_{5'}(G)$ is not $S_6$.)

%%and $G\le(C_4\circ2^{1+4}).S_6$.

\item $p=7$, $n=8$, $O_{7'}(G)\cong C_3\times2^{1+6}_+$ or $2^{1+6}_+$, and 
$G/O_{7'}(G)\cong S_8\cong\SO_6^+(2)$, $S_7$, $\PGL_2(7)$, or $C_7\rtimes 
C_6$.

%%and $G\le(C_3\times2^{1+6}_+).S_8$.

\end{enuma}
%%If $G$ is equal to the group in (e), then $\mu_V(G\7)=\Delta_3$.
\end{Prop}

\begin{proof} Fix $\UUU\in\sylp{G}$. By assumption, $\UUU\nnsg G$, 
$|\UUU|=p$, and $|\Aut_G(\UUU)|=p-1$.

\smallskip

\textbf{Case 1: } Assume first that the image of $G$ in 
$\PGL(V)$ is almost simple, and set $G_0=F^*(G)$ and $\Gamma=G_0/Z(G_0)$. 
Thus $G_0$ is quasisimple, $\Gamma$ is simple, and we must show that 
$p\bmid|\Gamma|$. 

Assume otherwise; then $p\bmid|\Out(G_0)|$, and hence $p\bmid|\Out(\Gamma)|$. Since $p$ is an odd prime (and clearly $\Gamma\not\cong A_6$), this is impossible when $\Gamma$ is an alternating or sporadic group \cite[\S\,I.5]{GL}. Hence $\Gamma$ is of Lie type. From the tables in \cite[\S\,I.7]{GL} it follows that $p\nmid|\Outdiag(\Gamma)|$, and hence that $\Out(\Gamma)$ possesses a normal $p$-complement. But then $\Aut_G(\UUU)=1$, which contradicts our assumptions on $G$. 

\smallskip

\textbf{Case 2: } Assume that $V$ is a simple $\F_pG$-module, and that the 
image of $G$ in $\PGL(V)$ is not almost simple. By the main theorem in 
\cite{Asch}, $G$ is contained in one of an explicit list of geometrically 
defined subgroups of $\GL_n(p)$, which fall into eight classes $\calc_i$ 
for $1\le{}i\le8$.  Of these, $\calc_8$ consists of the subgroups 
$\Sp_n(p)$ (if $n$ is even) and $GO_n(p)$ (for all choices of quadratic 
form).  If $G\le{}GL(\F_p^n,\qq)$ for some symplectic or quadratic form 
$\qq$, then $G$ is contained in a subgroup in one of the classes $\calc_i$ 
for this classical group.  

Assume that $G\le\4G\in\calc_i$ for $i\le7$.  Since $V$ is simple, $i\ne1$. 
Since $p$ is a prime, $\calc_5=\emptyset$. If $i=3$, then 
$G\le{}\GL_m(p^k)$ where $mk=n$ and $k>1$, which is impossible since each 
Jordan block over $\F_{p^k}$ splits as a sum of $k$ Jordan blocks over 
$\F_p$. If $i=4$ or $7$, then $G$ is contained in a tensor product or 
wreath tensor product of representations, which again implies that no 
element of order $p$ acts with exactly one non-trivial Jordan block. 

Now assume that $i=2$, so that $G\le\4G\cong\GL_m(p)\wr S_k$ for some 
$m,k$ such that $mk=n$ and $k>1$. If $m=1$, then we are in the situation of 
(b). So assume that $m\ge2$, set $H=G\cap(\GL_m(p))^k$, and let $H_i\le\GL_m(p)$ 
be the image of $H$ under projection to the $i$-th factor ($1\le i\le k$). 
The $H_i$ are all isomorphic, since otherwise $G/H$ would not permute them 
transitively, contradicting the assumption that $V$ is simple. If 
$p\big||H_i|$ for all $i$, then since $p^2\nmid|G|$, the elements of order 
$p$ in $G$ act with $m\ge2$ non-trivial Jordan blocks, a contradiction. Thus 
there is an element of order $p$ which non-trivially permutes some set of 
the $H_i$, this acts on $V$ with $k\ge2$ non-trivial Jordan blocks, which 
again is a contradiction.

We are left with the case where $G\le\4G\in\calc_6$.  Then for some 
prime $r|(p-1)$ and some $k\ge1$, $n=r^k$ and 
$\4G=N_{\GL_n(p)}(\4K)$ where $\4K\cong C_{p-1}\circ r^{1+2k}_\pm$ 
acts as a group of monomial matrices whose non-zero entries are $r$-th roots 
of unity.  Also, $\4G/\4K\cong C_{\Out(\4K)}(C_{p-1})$ 
is isomorphic to $\Sp_{2k}(r)$, or (if $r=2$ and $4\nmid(p-1)$) to 
$\GO_{2k}^\pm(r)$.  In particular, $p\bmid|\Sp_{2k}(r)|$. 

Set $k=\F_p$, or $k=\F_{p^2}$ if $r=2$. The action of the subgroup 
$(k^\times\circ r^{1+2k}).(\Sp_{2k-2}(r)\times\Sp_2(r))$ on 
$k\otimes_{\F_p}V$ factors as a tensor product $V\otimes_kW$, where 
$(k^\times\circ r^{1+2(k-1)}).\Sp_{2k-2}(r)$ acts on $V$ with 
$\dim_k(V)=r^{k-1}$, and $(k^\times\circ r^{1+2}).\Sp_{2}(r)$ acts on $W$ with 
$\dim_k(W)=r$. So if $p\big||\Sp_{2k-2}(r)|$, then the Jordan blocks for 
elements of order $p$ occur in multiples of $r$, which is impossible since 
there is a unique non-trivial block. We conclude that $p|(r^{2k}-1)$, so 
$p|(r^k\pm1)$, and $\dim(V)=r^k\equiv\pm1$ (mod $p$).  

For each non-central element $x\in\4K$, $C_{\Sp_{2k}(r)}(x)$ has index 
$r^{2k}-1$ in the symplectic group and hence has order prime to $p$. Thus 
$C_{\4K}(\UUU)=Z(\4K)=\Aut\scal(V)$, and 
$N_{\4K\UUU}(\UUU)=\UUU\Aut\scal(V)\cong C_{p(p-1)}$. Also, $V$ is 
irreducible as an $\F_p\4K$-module and hence as an 
$\F_p[\4K\UUU]$-module, and the $\4K\UUU$-Green correspondent of $V$ 
is indecomposable as an $\F_p[\UUU\Aut\scal(V)]$-module 
and hence of dimension at most $p$. Since $\dim(V)\equiv\pm1$ (mod $p$), 
the Green correspondent either has dimension $1$, in which case 
$\dim(V)=p+1$ since $V$ is minimally active, or it has dimension $p-1$, in 
which case $\dim(V)=p-1$. 

Thus $p=\dim(V)\pm1=r^k\pm1$, which is possible only if $r=2$ and $p$ is a 
Fermat or Mersenne prime. The action of $\UUU\cong C_p$ on the symplectic 
space $\4K/\Aut\scal(V)\cong(\F_2)^{2k}$ has at most $2k$ eigenvalues 
(in the algebraic closure $\4\F_2$), and since $|\Aut_G(\UUU)|=p-1$, they 
must include all $(p-1)$-th roots of unity other than $1$. Thus $p-1\le2k$, 
so $(p,k)=(3,1)$, $(5,2)$, or $(7,3)$, which correspond to cases (c)--(e) 
listed above. Note that $7\nmid|O_6^-(2)|$, so this case cannot occur. 

Set $G_0=G\cap\4K$, and regard $\4K/Z(\4K)\cong\F_2^{2k}$ as 
an $\F_2[G/G_0]$-module. Since $G$ contains $\UUU\cong C_p$ and 
$|\Aut_G(\UUU)|=p-1$ (and since $2k=p-1$ in each case), 
$\4K/Z(\4K)$ is a simple module. Hence either $G_0\le 
Z(\4K)$, or $G_0Z(\4K)=\4G$. By Lemma \ref{l:prim.action} (and 
since $|\Aut_G(\UUU)|=p-1$), we have $G/G_0\cong S_3$ in case (c); 
$G/G_0\cong S_6$, $S_5$, or $C_5\sd{}C_4$ in case (d); or $G/G_0\cong S_8$, 
$S_7$, $\PGL_2(7)$, or $C_7\rtimes C_6$ in case (e). So $G_0\nleq 
Z(\4K)$, since otherwise, either the image of $G$ in $\PGL(V)$ would 
be almost simple or $\UUU$ would be normal in $G$. We are left with the 
possibilities listed in the proposition. 

\smallskip

\noindent\textbf{Case 3: } Now assume that $V$ is not simple. Let $0=V_0<V_1<\cdots<V_k=V$ be $\F_pG$-submodules such that $W_i=V_i/V_{i-1}$ is simple for each $1\le i\le k$. Set $H_i=C_G(W_i)$ for each $i$. Thus $G/H_i$ acts faithfully on $W_i$. Since $O_p(G)=1$ by assumption, $G$ acts faithfully on $W_1\oplus\cdots\oplus W_k$ (cf. \cite[Theorem 5.3.2]{Gorenstein}), so $\bigcap_{i=1}^kH_i=1$.

If $H\nsg G$ and $p\bmid|H|$, then since $p^2\nmid|G|$, $H\ge O^{p'}(G)$. 
Hence there is some $1\le\ell\le k$ such that $p\nmid|H_\ell|$. Then 
$V=C_V(H_\ell)\oplus[H_\ell,V]$ as $\F_pG$-modules, 
$\dim(C_V(H_\ell))\ge\dim(W_\ell)>0$, and $V$ is indecomposable, so 
$[H_\ell,V]=0$ and hence $H_\ell=1$. Thus $G$ acts faithfully on $W_\ell$, 
and $(G,W_\ell)$ is one of the pairs listed in cases (a)--(e).

Set $K=O_{p'}(G)$. Let $\Irr_V(K)$ be the set of irreducible 
$\F_pK$-characters which appear as summands of $V|_K$, and similarly for 
$\Irr_{W_i}(K)$. For each $\chi\in\Irr_V(K)$, let 
$V_\chi\le V|_K$ be the submodule generated by all irreducible submodules 
with character $\chi$. Thus $V|_K$ is the direct sum (as $\F_pK$-modules) 
of the $V_\chi$ for $\chi\in\Irr_V(K)$. Since $V$ is 
$\F_pG$-indecomposable, the action of $G$ on $\Irr_V(K)$ 
induced by conjugation must be transitive. In particular, since the subsets 
$\Irr_{W_i}(K)$ are non-empty and $G$-invariant, we have 
$\Irr_{W_i}(K)=\Irr_V(K)$ for each $1\le i\le k$. 

If the image of $G$ in $\PGL(W_\ell)$ is almost simple, then $K$ is 
cyclic of order dividing $p-1$, and $\Irr_{W_\ell}(K)=\Irr_V(K)$ contains 
just one character. Hence $K$ acts on $V$ via multiplication by scalars, 
$G/K$ is almost simple, and $(G,V)$ is as in case (a). 

Now assume that $(G,W_\ell)$ is in one of the cases (b)--(e). Let $\4G$ be 
the maximal group listed in that case (thus $G\le\4G$), and set 
$\4K=O_{p'}(\4G)$. In case (b), $G/K$ acts 2-transitively on the 
set $\Irr_{W_\ell}(\4K)$ by Proposition \ref{p:wreathcase}, so after 
restriction to $K$, its elements are all distinct. (By the description of 
$K$ in Proposition \ref{p:wreathcase}, they cannot be pairwise isomorphic.) 
Hence $\dim(V)\ge2\dim(W_\ell)\ge2p$, which contradicts Proposition 
\ref{prop:collatedresults}(c) ($\dim(V)\le p+1$).

In each of cases (c)--(e), $W_\ell|_K$ is $\F_pK$-simple, so $\Irr_V(K)$ 
contains only its character. Hence $\dim(V)=m\dim(W_\ell)=m(p-1)$ for some 
$m\ge2$. Since $\dim(V)\le p+1$ by Proposition 
\ref{prop:collatedresults}(c), this is possible only when $p=3$ and $m=2$, 
as described in case (c).
\end{proof}

\iffalse
Having considered the first two possibilities from Proposition 
\ref{p:reduce2simple}, we examine cases (c) to (e). The first two arise 
from group fusion systems, and so their inclusion in Table 
\ref{tbl:alm.simple.reps} is clear. We exclude case (e) in the next lemma.
\fi

We finish the section with a computation which was used in the proof of 
Theorem \ref{t:alm.simple.reps} to show that the representations in case 
(e) of Proposition \ref{p:reduce2simple} cannot be used to 
construct simple fusion systems.

\begin{Lem} \label{l:mu(2^1+6)}
Set $G_0=2^{1+6}_+$, let $V$ be the unique faithful, irreducible 
$\F_7[G_0]$-module, and set $G=N_{\GL(V)}(G_0)$. Fix $\UUU\in\syl7{G}$. 
Then $G/G_0\cong C_3\times\SO_6^+(2)\cong C_3\times S_8$, 
$|N_G(\UUU)/\UUU|=6^2$, and $\mu_V(G\7)=\Delta_3$ (see Notation 
\ref{n:not3}). 
\end{Lem}

\begin{proof} Since $V$ is the unique faithful, irreducible 
$\F_7[G_0]$-module, $\Out_G(G_0)=\Out(G_0)\cong\SO_6^+(2)\cong S_8$, where 
$G_0/Z(G_0)$ is the natural module for $S_8$. Since 
$C_{\GL(V)}(G_0)=\Aut\scal(V)\cong C_6$, this proves that $G/G_0\cong 
C_3\times S_8$. Set $Z=Z(G_0)$, $\5Z=\Aut\scal(V)\cong C_3\times Z$, and 
$\5G_0=G_0\5Z\cong C_3\times G_0$. 

Now, $\Out_G(G_0)$ is contained in the symplectic group $\Sp_6(2)$, 
which acts transitively on the $63$ involutions in $\5G_0/\5Z\cong G_0/Z$. 
Since $7^2\nmid|\Sp_6(2)|$, this means that $C_G(x)$ has order prime to $7$ 
for each $x\in\5G_0{\sminus}\5Z$, and hence that 
$C_G(\UUU)=\UUU\times\5Z$ and $|N_G(\UUU)/\UUU|=6^2$. So by 
Proposition \ref{p:mu(Gvee)}(b), $|G\7/\UUU|=|\mu_V(G\7)|=6$, and hence 
$\mu_V(G\7)=\Delta_m$ for some $m$. 

Fix $g\in G\7$ such that $c_g$ has order $6$ in $\Aut_G(\UUU)$. Then 
$g^6=1$ since $|G\7/\UUU|=6$.

Identify $G_0=H\circ H\circ H$ (central product) and $V=W\otimes W\otimes 
W$, where $H\cong D_8$, $W$ is a faithful, irreducible $\F_7H$-module 
($\dim(W)=2$), and where $G_0$ acts on $V$ as the tensor power of the 
$H$-action on $W$. Choose $h_0\in N_{\Aut(W)}(H){\sminus}H$ of order $2$ 
(since $\GL_2(7)$ contains an extension of $\F_{49}^\times$ by a field 
automorphism, $N_{\Aut(W)}(H)$ contains a dihedral subgroup of order $16$). 
Define $h\in G=N_{\GL(V)}(G_0)$ by setting $h(w_1\otimes w_2\otimes w_3) = 
h_0(w_2)\otimes h_0(w_3)\otimes h_0(w_1)$. Since the action of $h$ on 
$G_0/Z\cong C_2^6$ permutes a basis transitively, the image of $h$ in 
$G/\5G_0\cong S_8$ is a $6$-cycle, thus conjugate to the image of $g$. 
Also, $|h|=6$, and $\chi_V(h)=0$ since $V$ has a basis permuted by $h$ with 
cycles of length $6$ and $2$.

Since $C_{G_0/Z}(h^3)=[h^3,G_0/Z]$ (both of rank $3$), $\5G_0/\5Z\cong 
G_0/Z\cong C_2^6$ is projective and hence cohomologically trivial as an 
$\F_2[\gen{h}]$-module. Hence all subgroups of $\5G_0\gen{h}/\5Z$ 
complementary to $\5G_0/\5Z$ are conjugate to each other. We already saw 
that $g$ is $G$-conjugate to some element of order $6$ in the coset 
$h\5G_0$, and hence $g$ is conjugate to $hz$ for some 
$z\in\5Z=\Aut\scal(V)$. Thus for some $6$th root of unity $\zeta$, 
$\chi_V(g)=\chi_V(hz)=\zeta\cdot\chi_V(h)=0$. Proposition 
\ref{p:mu(Gvee)}(c.ii) now implies that $\mu_V(G\7)=\Delta_3$. 
\end{proof}

\iffalse
In light of Proposition \ref{p:reduce2simple}(a), we still need 
to consider subgroups of $\GL(V)$ whose image in $\PGL(V)$ is almost 
simple. Such groups are often called ``nearly simple'' in the literature. 
We say that a group $G$ has \emph{type $H$} if $G/Z(G)$ is an almost simple 
group with simple normal subgroup $H$. Thus $\GL_n(q)$ is of type 
$\PSL_n(q)$. This will be useful because the precise group of type 
(say) $A_n$ will change depending on the representation considered, and 
could be $A_n$, a central extension $2\cdot A_n$, an almost simple group 
$S_n$, or a central extension $2\cdot S_n$ of these; each of these groups 
is of type $A_n$. 
\fi

%%\newpage

\section{$\PSL_2(p)$}
\label{sec:gl2p}

Recall Definition \ref{d:type}: for a finite simple group $L$, we say 
that a finite group $G$ is \emph{of type $L$} if $Z(G)$ is cyclic and 
$F^*(G)/Z(G)\cong L$. This concept provides a convenient way to organize 
the search for all pairs $(G,V)$, where $G\in\GG$, $V$ is minimally active, 
and the image of $G$ in $\PGL(V)$ is almost simple. So in the remaining 
sections, we go systematically through the list of non-abelian simple groups 
$L$, and list for each prime $p$ the groups $G\in\GG$ of type $L$ and their 
indecomposable, minimally active modules.

In this section, we handle the case where $L$ is a simple group of Lie type 
in defining characteristic $p$. If $L\in\G$, then $L$ is isomorphic to 
$\PSL_2(p)$, as all other groups of Lie type have Sylow $p$-subgroups of 
order greater than $p$. We are thus reduced to the cases where $G_0$ is 
isomorphic to $\PSL_2(p)$ or $\SL_2(p)$, and consider the simple and 
indecomposable modules for $\SL_2(p)$.

There are $p$ different simple modules $V_1,\dots,V_p$ for $\SL_2(p)$, 
where $\dim(V_i)=i$. (One way to construct them is to let $V_2$ be the 
natural module, and set $V_i=S^{i-1}(V_2)$ for $3\le i\le p$.) Hence by 
Proposition \ref{prop:collatedresults}(c), the dimension of each 
indecomposable minimally active $\F_pG$-module is at most $p+1$. To 
determine all such modules, we describe the projective covers of the $V_i$, 
referring to \cite[pp. 75--79]{Alperin} for more detail. (Although 
Alperin's descriptions are only for an algebraically closed field, since 
all simple modules are defined over $\F_p$ the projectives must be so as 
well.)

The module $V_p$ is projective of dimension $p$. For $2\le i\le p-2$, 
the projective cover of $V_i$ is of shape
	\[ V_i/(V_{p-1-i}\oplus V_{p+1-i})/V_i,\]
while the projective covers of $V_1$ and $V_{p-1}$ are of shape 
$V_1/V_{p-2}/V_1$ and $V_{p-1}/V_2/V_{p-1}$, respectively. Here `$/$' 
delineates the radical layers, with the socle appearing on the right, so 
that $B$ is the submodule and $A$ the quotient in $A/B$.

This yields indecomposable modules $V_i/V_{p-1-i}$ (for each $1\le i\le 
p-2$) and $V_i/V_{p+1-i}$ (each $2\le i\le p-1$), of dimension $p-1$ and 
$p+1$ respectively (and they are all minimally active by Lemma 
\ref{l:ext-minact}). Any minimally active indecomposable modules not yet 
found must have dimension $p$ or $p+1$. If $\dim(V)=p$, then $V$ is 
projective by Proposition \ref{prop:collatedresults}(a), so $V\cong V_p$ or 
$V$ is the projective cover of the trivial module $V_1/V_{p-2}/V_1$ as 
described above. 

If $V$ has dimension $p+1$ and has at least three composition factors, then 
there are either three factors including a copy of $V_2$, or four factors 
including two copies of $V_1$. In the former case, either $V$ or $V^*$ has 
a simple socle, and so is a quotient module of one of the projectives above, 
which by inspection cannot occur. If there are four composition factors, 
then two are trivial, so at least one of the other two must have 
non-trivial $1$-cohomology. By Lemma \ref{l:cohomology}, 
$V_{p-2}$ is the only such module, so $V$ would have this and three trivial 
modules as composition factors, which is impossible by the above discussion. 

We have now shown that each minimally active indecomposable module $V$ is 
of one of the following types: 
\begin{itemize}
\item $V\cong V_i$ for $i>1$;
\item $V\cong V_i/V_{p\pm 1-i}$; or 
\item $V\cong V_1/V_{p-2}/V_1$ the projective cover of the trivial module $V_1$.
\end{itemize}
Also, the action of $\SL_2(p)$ on each $V_i$ extends to one of $\GL_2(p)$ 
(being a symmetric power of the natural module). 

We claim that the action of $G_0=\SL_2(p)$ on each of the non-simple 
indecomposable modules listed above also extends to an action of $G=\GL_2(p)$. To see this, 
fix $i$, and let $U_1,\dots,U_{p-1}$ be the distinct simple $G$-modules 
whose restriction to $G_0$ is isomorphic to $V_i$. (These are obtained by 
taking one such module, and tensoring it by each of the 1-dimensional 
$G/G_0$-modules.) By Frobenius reciprocity, the induced module $V_i|^G$ is 
isomorphic to the direct sum of the $U_j$. Hence the natural projection of 
$P(V_i)|^G$ onto $V_i|^G$ (where $P(-)$ denotes projective cover) lifts to 
a homomorphism $\Phi=\bigoplus_{j=1}^{p-1}\Phi_j$ from $P(V_i)|^G$ to 
$\bigoplus_{j=1}^{p-1}P(U_j)$. Since by definition, the kernel of the 
natural projection $P(U_j)\Right2{}U_j$ is contained in the radical of 
$P(U_j)$ and hence in all maximal submodules, $\Im(\Phi_{1})=P(U_{1})$, 
$\Im(\Phi_{2}|_{\Ker(\Phi_{1})})=P(U_{2})$, etc. Thus $\Phi$ is onto. 
Also, $\dim(P(U_j))\ge\dim(P(V_i))$ for each $j$ since $P(U_j)|_{G_0}$ 
contains $P(V_i)$ as a direct summand. So by comparing dimensions, we see 
that $\Phi$ is an 
isomorphism and $P(U_j)|_{G_0}\cong P(V_i)$. Thus for each $i$, the action 
of $G_0$ on $P(V_i)$ extends to $G$, and hence the same holds for the 
quotient modules of these projective covers listed above.

We next determine the normalizer $N_{\GL(V)}(G_0)$, when $V$ is one of the 
$G_0$-modules just listed. We first consider the centralizer 
$C_{\GL(V)}(G_0)$, which obviously contains the scalar matrices 
$\Aut\scal(V)$. If $V$ is a simple $\F_pG_0$-module, then it is absolutely 
simple by Lemma \ref{l:minact_decomp}(c), and hence 
$C_{\GL(V)}(G_0)=\Aut\scal(V)$ by Schur's lemma. If $V$ acts indecomposably 
with socle $V_i$ and quotient $V_{p\pm 1-i}$ (both simple), and $g\in 
C_{\GL(V)}(G_0)$, then $g$ stabilizes $V_i$, and $g|_{V_i}=u\cdot\Id$ and 
$g\equiv u'\cdot\Id$ (mod $V_i$) for some $u,u'\in\F_p^\times$. Also, $u=u'$ 
since $V$ is indecomposable, and $g$ has the form $g(x)=ux+\psi(x+V_i)$ for 
some $\psi\in\Hom_{G_0}(V/V_i,V_i)$. Thus either 
$C_{\GL(V)}(G_0)=\Aut\scal(V)$, or $V/V_i\cong V_i$ (so $i=(p\pm1)/2$) and 
$C_{\GL(V)}(G_0)\cong\Aut\scal(V)\times C_p$. A similar argument 
proves that $C_{\GL(V)}(G_0)=\Aut\scal(V)\times C_p$ when $V$ is the 
projective cover of the trivial module. 

Thus $C_{\GL(V)}(G_0)\cdot G_0=G_0\circ\Aut\scal(V)$ or 
$(G_0\circ\Aut\scal(V))\times C_p$ in all cases. Also, since 
$N_{\GL(V)}(G_0)/G_0\cdot C_{\GL(V)}(G_0)$ is a subgroup of $\Out(G_0)\cong 
C_2$, the normalizer $N_{\GL(V)}(G_0)$ contains $C_{\GL(V)}(G_0)\cdot G_0$ 
with index at most $2$. Thus $G\le\4G$, where $\4G$ (as defined in Section 
\ref{s:reps-summary}) has index $1$ or $p$ in $N_{\GL(V)}(G_0)$, and in all 
cases has the form $(G_0\circ\Aut\scal(V)).2$. (As noted above, the action 
of $\SL_2(p)$ or $\PSL_2(p)$ on $V$ always extends to the outer 
automorphism.)

\iffalse
Since $V$ is uniquely determined by $\dim(V)$ and 
$\dim(\soc(V))$, any automorphism of $G_0$ must act on this module, and 
hence $\4G$ is either equal to $N_{\GL(V)}(G_0)$ or a subgroup of index $p$ 
in it, but in both cases it has the structure
\[\left(G_0\circ Z(\GL(V))\right)\rtimes Z_2,\]
with the $Z_2$ acting on $G_0$ as the diagonal automorphism. In particular, since $\4 G$ has $\PGL_2(p)$ as a subquotient, $\Aut_G(\UUU)$ has order $p-1$ and so $\4 G\in \GG$.
\fi

By Proposition \ref{p:mu(Gvee)}(a), 
it is only the $(p+1)$-dimensional modules that might not produce reduced 
fusion systems. To know whether they do or not, we need to understand the two 
modules in the socle of $V|_N$ (recall $N=N_G(\UUU)$).

Assume that $V$ is an extension of $V_i$ (the submodule) by $V_j$, where $i+j=p+1$. 
Let $\zeta\in\F_p^\times$ be a generator, and fix the following elements in 
$\GL_2(p)$: 
	\[ g = \Mxtwo\zeta001, \qquad h = \Mxtwo100\zeta, \qquad 
	u=\Mxtwo1101. \]
Thus $\UUU=\gen{u}$, $\9gu=u^\zeta$, and $\9hu=u^{\zeta^{-1}}$. On the natural 
module $V_2$, $g$ acts with eigenvalues $\{\zeta,1\}$ and $h$ with 
eigenvalues $\{1,\zeta\}$, beginning with those of the socle.

Identify $V_i=\Sym^{i-1}(V_2)$ and $V_j=\Sym^{j-1}(V_2)$. Then $g,h$ have 
eigenvalues on $V_i$ and $V_j$ as follows: 
	\[ \renewcommand{\arraystretch}{1.5}
	\begin{array}{|c||c|c|} \hline
	 & V_i & V_j \\\hline\hline
	g & \zeta^{i-1},\zeta^{i-2},\dots,\zeta,1 &
	\zeta^{j-1},\zeta^{j-2},\dots,\zeta,1 \\\hline
	h & 1,\zeta,\zeta^2,\dots,\zeta^{i-1} &
	1,\zeta,\zeta^2,\dots,\zeta^{j-1} \\\hline
	\end{array} \]
(in each case from socle to top). 
In general, these actions of $g$ and $h$ don't extend to an action on $V$ 
(while the action of $gh^{-1}\in\SL_2(p)$ does extend). So let 
$z=\zeta\cdot\Id_{V_j}$, let $\hat{g},\hat{h}$ have the actions of $g,h$ on $V_i$, but 
the actions of $gz^{i-1}$ and $hz^{i-1}$ on $V_j$. Since 
$(i-1)+(j-1)\equiv0$ (mod $p-1$), we get the following eigenvalues: 
	\[ \renewcommand{\arraystretch}{1.5}
	\begin{array}{|c||c|c|} \hline
	 & V_i & V_j \\\hline\hline
	\hat{g} & \zeta^{i-1},\zeta^{i-2},\dots,\zeta,1 &
	1,\zeta^{-1},\zeta^{-2},\dots,\zeta^{-j+1} \\\hline
	\hat{h} & 1,\zeta,\zeta^2,\dots,\zeta^{i-1} &
	\zeta^{i-1},\zeta^i,\zeta^{i+1},\dots,1 \\\hline
	\end{array} \]
Then $\hat{g}$ and $\hat{h}$ have the same action on the top of $V_i$ and the 
socle of $V_j$, so they do extend to actions on $V$. In particular, 
$\hat{g}\in\4G\7$ and $\mu_V(\hat{g})=(\zeta,\zeta^{i-1})$, while 
$h^*\defeq\hat{h}\circ\zeta^{1-i}\cdot\Id_V \in\4G\7$ and 
$\mu_V(h^*)=(\zeta^{-1},\zeta^{1-i})$. 

We conclude that $\mu_V(\4G\7)=\Delta_{i-1}$. Thus each of the $\Delta_k$ 
except $\Delta_0$ can appear as $\mu_V(\4G\7)$ for some choice of $V$. When 
$i=p-1$ and $j=2$, we get $\mu_V(\4G\7)=\Delta_{-1}$. 

This yields the following proposition.

\begin{prop}\label{p:gl2p} 
Let $G\in \GG$ be a group of type $\PSL_2(p)$, and let $V$ be a 
non-trivial, minimally active module for $G$. Then one of the following 
holds:
\begin{enumi}
\item $V|_{G_0}\cong V_i$ is simple for $i\geq 2$;
\item $V|_{G_0}\cong V_{p-1-i}/V_i$ is indecomposable of dimension $p-1$;
\item $V|_{G_0}$ is projective of dimension $p$, of the form $V_1/V_{p-2}/V_1$;
\item $V|_{G_0}\cong V_{p+1-i}/V_i$ is indecomposable of dimension $p+1$.
\end{enumi}
Furthermore, $G$ is contained in $\left(G_0\circ Z(\GL(V))\right).2$.
In case (iv), $\mu_V(G)\leq \Delta_i$, and hence $\mu_V(G)=\Delta_{-1}$ only when $i=p-1$.
\end{prop}

%%\newpage

\section{Sporadic groups}

In this short section, we determine all minimally active modules for those 
sporadic groups (and their extensions) which lie in $\GG$. For the reader's 
convenience, we include a table listing, for each simple sporadic group 
$G_0$, all primes $p$ for which the Sylow $p$-subgroup has order $p$. Those 
primes for which $G_0\in\GG$ are in bold, and the other primes for which 
$\Aut(G_0)\in\GG$ are in italics.

\newcommand{\comp}[1]{\text{\mathversion{bold}$#1$\mathversion{normal}}}
\newcommand{\compaut}[1]{\textit{{#1}}}
\newcommand{\others}[1]{\Small{#1}}

\begin{center}
\renewcommand{\arraystretch}{1.3}
\addtolength{\tabcolsep}{-0.5mm}
\begin{tabular}{|cc|cc|cc|}
\hline Group & Primes & Group & Primes & Group & Primes \\\hline 
$M_{11}$ & \comp{5}; \others{11} & 
$HS$ & \comp7; \compaut{11} & 
$Ru$ & \comp{7,13}; \others{29} \\ 

$M_{12}$ & \comp{5}; \compaut{11} & 
$McL$ & \compaut{7}; \others{11} & 
$ON$ & \comp{5,11}; \compaut{31}; \others{19} \\ 

$M_{22}$ & \comp5; \compaut{11}; \others{7} & 
$\Suz$ & \comp{7,11}; \compaut{13} & 
$Fi_{22}$ & \comp7; \compaut{11,13} \\

$M_{23}$ & \comp5; \others{7,11,23} & 
$Co_3$ & \comp7; \others{11,23} & 
$Fi_{23}$ & \comp{7,11,17}; \others{13,23} \\

$M_{24}$ & \comp{5,11}; \others{7,23} & 
$Co_2$ & \comp{7,11}; \others{23} & 
$Fi_{24}'$ & \comp{11,13,17}; \compaut{29}; \others{23} \\

$J_1$ & \comp{3,7,11}; \others{5,19} & 
$Co_1$ & \comp{11,13}; \others{23} & 
$Ly$ & \comp7; \others{11,31,37,67} \\

$J_2$ & \comp{7} & 
$He$ & \compaut{17} & 
$B$ & \comp{11,13,17,19}; \others{23,31,47} \\ 

$J_3$ & \compaut{5,19}; \others{17} & 
$HN$ & \comp{7,11}; \compaut{19} & 
$M$ & \comp{17,19,29,41};\others{23,31,47,59,71} \\ 

$J_4$ & \comp{5,23};\others{7,29,31,37,43} &
$\Th$ & \comp{13,19}; \others{31} & 
& \\
\hline \end{tabular}\end{center}

By \cite{Craven2015un}, if $G$ is a sporadic simple group and $p\bmid|G|$, 
then with a few exceptions when $p=3$ (none of which are in $\G[3]$), 
$G=\gen{t,y}$ where $|t|=2$ and $|y|=p$. So by Proposition 
\ref{prop:dimlt2p-1}(b), for each central extension $\til{G}$ of $G$ of 
degree prime to $p$, each minimally active $\F_p\til{G}$-module has 
dimension at most $2p-2$.

The table \cite[Table 1]{jansen2005} provides a helpful list of the minimal 
degrees for sporadic groups for each prime. This allows us to eliminate 
almost all cases from the above table, just by applying the bound 
$\dim(V)\leq 2p-2$. We are left with the following possibilities for $G_0$ 
(or for $G_0.2$ when $G_0\notin\GG$):
\begin{center}\renewcommand{\arraystretch}{1.3}
\begin{tabular}{|c|c|}
\hline Prime & Possibilities for $G_0$ or $G_0.2$
\\ \hline 7& $2\cdot J_2$, $6\cdot\Suz$
\\11& $M_{12}.2$, $2\cdot M_{12}.2$, $J_1$, $M_{22}.2$, $2\cdot M_{22}.2$, 
[$6\cdot\Suz$]
\\13& [$6\cdot\Suz.2$], $2\cdot Co_1$
\\19& [$3\cdot J_3.2$]
\\\hline
\end{tabular}
\end{center}

Since our modules are defined over $\F_p$ and are minimally active, 
$Z(G_0)$ must act via multiplication by scalars, and hence 
$|Z(G_0)|\bmid(p-1)$ and $Z(G_0)$ is central in $G$ in all cases. This 
criterion allows us to eliminate the three entries in brackets in the above 
table. Note that the outer automorphisms of $6\cdot\Suz$ and 
$3\cdot J_3$ invert the centres: as described, e.g., in the tables in 
\cite[\S\,I.1.5]{GL}.

When $p=7$ and $G_0=2\cdot J_2$, the two $6$-dimensional modules over the 
algebraically closed field amalgamate into a single $12$-dimensional module 
over $\F_7$: this can be seen either by computer or from the Brauer character table on \cite[p.105]{modatlas}, where we see that there are irrationalities in the Brauer character of this representation, which from \cite[p.289]{modatlas} we see require $\F_{49}$.

By Proposition \ref{prop:dimlt2p-1}, if an 
$\F_pG$-module $V$ is minimally active (for $G\in\G$), and $\dim(V)>p$, 
then $\dim(V)-p$ divides $|N_G(\UUU)/\UUU|$, and $\dim(V)=p+1$ if 
$N_G(\UUU)/\UUU$ is abelian. We can thus eliminate the $\F_pG$-modules in 
the following dimensions:
\[\renewcommand{\arraystretch}{1.3}
\begin{array}{|c|c|c|c|c|}
\hline
p & \textup{Group} & \dim(V) & \dim(V)-p & N_G(\UUU)/\UUU \\\hline
7 & 6\cdot\Suz & 12 & 5 & 6\cdot((3\times A_4).2) \\\hline
11 & J_1 & 14 & 3 & 10 \\\hline
11 & M_{12}.2 & 16 & 5 & 10 \\\hline
11 & M_{22}.2 & 20 & 9 & 10 \\\hline
13 & 2\cdot\Co_1 & 24 & 11 & 2\cdot((6\times A_4).2) \\\hline
\end{array}
\]

We are left with the following cases, all for $p=11$. For $J_1$, there is 
one  $7$-dimensional module. For $2\cdot M_{12}.2$, there are two 
pairs of modules of dimension $10$ and one pair of dimension $12$ (where 
the modules in each pair are isomorphic after restriction to $2\cdot 
M_{12}$). For $2\cdot M_{22}.2$, there are two pairs of modules of 
dimension $10$. All of these modules are minimally active. 

We have nearly proved the following proposition.

\begin{prop}\label{p:sporadic}
Let $G\in\GG$ be a group of sporadic type, and let $V$ be a non-trivial, 
minimally active indecomposable module for $G$. Let $G_0=E(G)$ be 
the unique quasisimple normal subgroup of $G$. Then $p=11$, and one of the 
following holds:
\begin{enumi} 

\item $G_0\cong2\cdot M_{12}$, $G/Z(G)\cong M_{12}:2$, and $\dim(V)=10$ 
(two possible $F_{11}G_0$-modules) or $12$ (one such module); or

\item $G_0\cong2\cdot M_{22}$, $G/Z(G)\cong M_{22}:2$, and 
$\dim(V)=10$ (two possible $\F_{11}G_0$-modules); or

\item $G_0\cong G/Z(G)\cong J_1$, and $\dim(V)=7$ (one module).
\end{enumi}
\end{prop}

\begin{proof}
If $V$ is simple, then we are done. So assume that $V$ is 
indecomposable and not simple. In particular, $\dim(V)\le p+1$ by Proposition 
\ref{prop:collatedresults}(c). Hence $V$ 
has one non-trivial composition factor $W$ of dimension at most $p$ and the 
others are trivial. Then $Z(G_0)=1$, so $G_0\cong J_1$, $\dim(W)=7$, 
$W$ is self-dual since it is the only module in 
that dimension, and hence $H^1(G_0;W)=0$ by Lemma \ref{l:cohomology}. So 
this case is impossible.
\end{proof}

\newcommand{\dbl}[2]{\renewcommand{\arraystretch}{1.0}%
\begin{tabular}{c}\rule{0pt}{12pt}$#1$\\$#2$\end{tabular}}
\newcommand{\tpl}[3]{\renewcommand{\arraystretch}{1.0}%
\begin{tabular}{c}\rule{0pt}{12pt}$#1$\\$#2$\\$#3$\end{tabular}}
\newcommand{\qpl}[4]{\renewcommand{\arraystretch}{1.0}%
\begin{tabular}{c}\rule{0pt}{12pt}$#1$\\$#2$\\$#3$\\$#4$\end{tabular}}
\newcommand{\lowint}[1]{\lfloor#1\rfloor}

\section{Alternating groups}

In this section we determine all minimally active modules for almost quasisimple groups associated to the alternating groups.

\begin{prop}\label{p:alternating} 
Let $G\in\GG$ be a group of type $A_n$ for $n\geq 5$, and set $G_0=E(G)$. Let $V$ be a non-trivial,
minimally active indecomposable module for $G$.  Then one of the following holds:
\begin{enumi}
\item $G_0\cong A_p$, $G/Z(G)\cong S_p$, and $V|_{G_0}$ is a subquotient of the 
permutation module, which has structure $1/W/1$ for $W$ of dimension $p-2$;

\item $G_0\cong A_{p+1}$, $G/Z(G)\cong S_{p+1}$, and $V|_{G_0}$ is the $p$-dimensional non-trivial summand of the permutation module;

\item $G_0\cong A_n$, $G/Z(G)\cong A_n$ or $S_n$ for $p+2\leq n\leq 2p-1$, 
and $V|_{G_0}$ is the $(n-1)$-dimensional summand of the permutation module; 

\item $p=5$, $G_0\cong A_5\cong\PSL_2(5)$, $G/Z(G)\cong 
S_5\cong\PGL_2(5)$, and $V$ is one of the modules described in Proposition 
\ref{p:gl2p} other than those in (i);

\item $p=5$, $G_0\cong 2\cdot A_5\cong \SL_2(5)$, $G/Z(G)\cong S_5$, and $V$ is as in Proposition \ref{p:gl2p};

\item $p=5$, $G_0\cong2\cdot A_6$, $G/Z(G)\cong S_6$, and $\dim(V)=4$ 
(two modules);

\item $p=7$, $G_0\cong2\cdot A_7$, $G/Z(G)\cong S_7$, and $\dim(V)=4$; 

\item $p=7$, $G_0\cong2\cdot A_8$ or $2\cdot A_9$, $G/Z(G)\cong S_8$ or 
$A_9$, and $\dim(V)=8$ (one or two $\F_7G_0$-modules, 
respectively); or

\item $p=7$, $G_0\cong2\cdot A_7$, $G/Z(G)\cong S_7$, $\dim(V)=8$, and $V$ has the form $W/W$ where $W$ is as in (vii).
\end{enumi}
\end{prop}

\begin{proof} By a result of Miller \cite[pp.29--30]{miller1928}, for each 
$3<p\le n$, $A_n$ is generated by an element $t$ of order $2$ and an 
element $x$ of order $p$ (and this is easily seen to hold when 
$(p,n)=(3,5)$). Thus by Proposition \ref{prop:dimlt2p-1}(b), each minimally 
active $\F_pA_n$- or $\F_p[2\cdot A_n]$-module has dimension at most 
$2p-2$. Our knowledge of small-dimensional representations of these groups is rather extensive, which makes these cases relatively easy to handle.

The smallest (faithful) simple module for $S_n$ is the module arising from 
the permutation module, having dimension $n-1-\kappa_n$, where $\kappa_n=0$ 
if $p\nmid n$ and $\kappa_n=1$ if $p\mid n$. By \cite[Theorem 7 and Table 
1]{james1983} or \cite[Lemma 1.18]{BK}, if $p\ge7$ and $n\ge9$, or $p=5$ 
and $n\ge11$, the dimension of each larger $\F_pS_n$-module is strictly 
greater than $n(n-5)/2$, and hence that of each larger $\F_pA_n$-module is 
greater than $n(n-5)/4$. For each pair $(p,n)$ such that $n\ge10$ and $p\le 
n<2p$, $n(n-5)/4>2p-2$ except when $p=n=11$, and the smallest faithful 
$\F_{11}A_{11}$-module other than those in point (i) has dimension $36$ (see 
\cite{modatlas}). We are thus reduced to checking the cases $5\le n\le 9$.

\iffalse
The smallest (faithful) simple module for $S_n$ is the module arising from 
the permutation module, having dimension $n-1-\kappa_n$, where $\kappa_n=0$ 
if $p\nmid n$ and $1$ if $p\mid n$. From \cite[Theorem 7 and Table 
1]{james1983} (see also \cite[Lemma 1.18]{guralnicktiep2}, noting that 
$n\leq 9$ for $p=5$ since the Sylow $p$-subgroup is cyclic) we see that for 
all $n\geq 9$ the next degree up is greater than $n(n-5)/2$, which is 
greater than $2n-2\geq 2p-2$. We are thus reduced to checking the cases 
$5\le n\le 8$ (and $n=9$ for $p=5$, where the next smallest degree is 
$21\geq 2p$).
\fi

%%, so we now examine the faithful representations of $2\cdot A_n$ 
%%for $n\geq 9$.

For $n\geq 12$, the smallest faithful representation of $2\cdot A_n$ is of 
dimension $2^{\lfloor (n-2-\kappa_n)/2\rfloor}$, where $\kappa_n$ is as 
above (see, e.g., \cite{kleschtiep}).
If $V$ is minimally active, we have $\dim(V)\leq 2p-2\leq 2n-2$. Hence 
	\[ 2n-2\geq 2^{\lfloor (n-2-\kappa_n)/2\rfloor},\] 
which yields $n\leq 11$. (Note that since $n<2p$, $\kappa_n=0$ whenever 
$n\ne p$, in particular, when $n$ is not prime.) 

We can get yet more restrictions based on Green correspondence. Assume that $G$ 
is a central extension of $A_n$ for $p\le n\le2p-1$, and let $V$ be an indecomposable 
minimally active $\F_pG$-module. Recall (Proposition 
\ref{prop:collatedresults}(c,e)) that the Green correspondent of $V$ is an 
absolutely simple $N_G(\UUU)/\UUU$-module.
\begin{enumerate}[(1) ] 
\item If $p\le n\le p+2$, then $N_{A_n}(\UUU)/\UUU$ is cyclic of order $p-1$ or 
$(p-1)/2$, so $N_G(\UUU)/\UUU$ is abelian, and $\dim(V)\le p+1$ by 
Proposition \ref{prop:dimlt2p-1}(a).
\item If $n=p+3$, then $N_{A_n}(\UUU)/\UUU\cong C_3\rtimes C_{p-1}$. Hence 
$N_G(\UUU)/\UUU$ has an abelian subgroup of index $2$, so its irreducible 
representations have dimension at most $2$, and $\dim(V)\le p+2$.

\item If $n=p+4$ and $G=2\cdot A_n$ (and acts faithfully on $V$), then 
$N_G(\UUU)/\UUU$ contains a subgroup $H\cong B\circ2\cdot A_4$ with index 
$2$, where $B$ is abelian. Since the absolutely irreducible representations 
of $H$ on which the central involution acts non-trivially are 
all 2-dimensional, either $\dim(V)\le p$, or $\dim(V)=p+2$ or $p+4$. 

\end{enumerate}

\iffalse
We will not mention 
the module for $A_n$ (and two for $S_n$) of dimension $n-1-\kappa_n$, since 
we already know that exists.
\fi

\begin{table}[ht]
\[ \renewcommand{\arraystretch}{1.3} 
\newcommand{\nma}[1]{^{(#1)}} 
\newcommand{\xxx}{\cline{1-5}\cline{7-11}}
\begin{array}{|c||c|c||c|c|c|c||c|c||c|c|}
\xxx
G_0 & p & \textup{dimen.} & p & \textup{dimen.} & \qquad & G_0 & p & 
\textup{dimen.} & p & \textup{dimen.} 
\\\xxx
A_5 & 3 & 3^*,3^* & 5 & \textup{Prop.6.1} && 2\cdot A_5 & 3 & 
2^*,2^* & 5 & \textup{Prop.6.1} \\\xxx
A_6 & 5 & 8\nma1 & \multicolumn{2}{c|}{} && 2\cdot A_6 & 5 & \textbf{4,4} & 
\multicolumn{2}{c|}{} \\\cline{1-3}\cline{7-9}
3\cdot A_6 & 5 & 3\nmid(p-1) & \multicolumn{2}{c|}{} && 6\cdot A_6 & 5 & 6\nmid(p-1) & 
\multicolumn{2}{c|}{} \\\xxx
A_7 & 5 & 8\nma1 & 7 & 10\nma1 && 2\cdot A_7 & 5 & 4^*,4^* & 7 & 
\textbf4 \\\xxx
3\cdot A_7 & 5 & 3\nmid(p-1) & 7 & 6^\dag,9^\dag && 6\cdot A_7 & 5 & 6\nmid(p-1) & 7 & 
6^\dag,6^\dag \\\xxx
A_8 & 5 & \textup{none} & 7 & \textup{none} && 2\cdot A_8 & 5 & 8\nma2 & 7 & \textbf8 \\\xxx
A_9 & 5 & \textup{none} & 7 & \textup{none} && 2\cdot A_9 & 5 & 8\nma3,8\nma3 & 7 & \textbf{8,8} \\\xxx
2\cdot A_{10} & 7 & \textup{none} & \multicolumn{2}{c|}{} && 2\cdot A_{11} & 7 & 8 & 11 & 16\nma1 
\\\xxx
\end{array}
\]
\caption{Modules of dimension $\le2p-2$ for quasisimple alternating groups}
\label{t:alternating}
\end{table}

Thus we can restrict attention to faithful representations of $2\cdot A_n$ 
for $n\leq 11$ and of $A_n$ for $n\leq9$, as well as 
those of $3\cdot A_n$ and $6\cdot A_n$ for $n=6,7$. All simple modules for 
all primes are known for these alternating groups, and we can simply check 
them one by one and prime by prime. This is done in Table 
\ref{t:alternating}, where (based on \cite{modatlas}) we list dimensions of 
all irreducible $\4\F_pG_0$-modules of dimension at most $2p-2$ when $G_0$ 
is a central extension of $A_n$ for $p\le n\le2p-1$, except for the natural 
modules for $A_n$ described in points (i)--(iii). 

For the modules listed in the table, an asterisk $(-)^*$ means that it is 
not realized over $\F_p$ (hence does not give rise to any minimally active 
$\F_pG_0$-module); a dagger $(-)^\dag$ means that the module does not 
extend to $S_n$ (when $n=p$ or $p+1$), and a superscript $(-)^{(i)}$ for 
$i=1,2,3$ means that it is not minimally active by point ($i$) above. Note 
that the groups $3\cdot S_7$ and $6\cdot S_7$ are ``twisted'' in the sense 
that their outer automorphisms invert their centres, so there are no 
indecomposable modules for these groups over $\F_7$. The remaining modules 
(aside from the case $(n,p)=(5,5)$ of points (iv) and (v)) are shown in 
boldface, and are, in fact, minimally active, as can be seen by using the 
Magma command \texttt{IndecomposableSummands(Restriction(V,U))} to check 
the block sizes. These are precisely the modules listed in points 
(vi)--(viii). 

Now assume that $V$ is indecomposable but not simple. By Proposition 
\ref{prop:collatedresults}(c), $\dim(V)\le p+1$. If one or more of the 
simple components of $V$ is 1-dimensional, then $G_0\cong A_n$, and $V$ 
contains a simple composition factor $V_0$ of dimension $p-2$ by Lemma 
\ref{l:cohomology}. Hence we are in the situation of (i).

The only remaining possibility is the case where $p=7$ and $G_0\cong2\cdot 
A_7$, and $V$ is an indecomposable extension of the 4-dimensional module in 
Table \ref{t:alternating} by itself. By \cite[Proposition 
21.7]{Alperin}, there is at most one such extension. From the tables in 
\cite{modatlas}, we see that the restriction to $2\cdot A_7$ of the simple 
8-dimensional $2\cdot A_8$-module of case (viii) has this form. Hence there 
is a module of this type, and it is the restriction of a $2\cdot 
S_7$-module.
\end{proof}

\iffalse
That there is exactly one such 
$\F_7[2\cdot A_7]$-module can be shown via the Magma command 
\texttt{Ext(M,M)}. Alternatively, this follows from the Brauer tree for the 
appropriate block for $2\cdot A_7$ as described in \cite[p. 435, fig. 
3]{Mueller} (see \cite[\S\,17]{Alperin} for details on how to interpret 
it). The module $V$ is minimally active by Lemma \ref{l:ext-minact}.} 
\fi

%%\newpage

\section{Groups of Lie type: notation and preliminaries}

We continue our notation that $G$ is a finite group, $\UUU$ is a Sylow 
$p$-subgroup of $G$, and $x\in \UUU$ has order $p$, writing $N=N_G(\UUU)$ 
and $C=C_G(\UUU)$, with $C=\UUU\times C'$. 

In this section we consider groups of Lie type, and use induction to reduce 
the problem of classifying minimally active modules to a small set of 
situations, essentially where the centralizer is abelian. We start with a 
brief overview of the orders of finite simple groups of Lie type, using 
\cite{GL} or \cite{malletest} as a reference. We assume a 
passing familiarity with basic concepts from algebraic groups, such as 
simple connectivity (see \cite{GLS3} for a brief outline of the 
background assumed, and \cite{malletest} for more details). 
It will be useful, in many 
cases, to write $\SL_n^\pm(q)$ or $E_6^\pm(q)$, where 
	\[ \SL_n^+(q)=\SL_n(q), \quad \SL_n^-(q)=\SU_n(q), \quad 
	E_6^+(q)=E_6(q), \quad E_6^-(q)=\lie2E6(q). \]

If $G={}^r\gg(q)$ is a finite group of Lie type of universal type, then its 
order is
\[ |G|=q^N\prod_{d} \Phi_d(q)^{a_d},\]
a power of $q$ times a product of cyclotomic polynomials. However, to get 
the order of the associated simple group (if there is one), we must 
divide this order by $z=|Z(G)|$, an integer. The polynomials in $q$ are 
given in Tables \ref{tab:classical} and \ref{tab:exceptional}. We have 
dealt with the case where $p\mid q$ in Section \ref{sec:gl2p}, where we saw that the only 
possibility is $\SL_2(p)$. So we assume for the rest of the paper that 
$\gcd(p,q)=1$.

\begin{table}[ht]
\begin{small} 
\renewcommand{\arraystretch}{1.2}
\[ \begin{array}{|cccc|}
\hline G & |G|_{q'} & |Z(G)| & \textup{$d$ with $a_d=1$} 
\\\hline \SL_2(q) & (q-1)(q+1) & (2,q-1) & \comp{1,2}
%%\\ \SL_n(q) & \D \prod_{i=2}^n (q^i-1) & (n,q-1) & 
%%{\scriptstyle\lfloor n/2\rfloor+1,\dots,n-2},\bolddd{$n{-}1,n$}
%%\\ \SU_n(q) & \D \prod_{i=2}^n (q^i-(-1)^i) & (n,q+1) & 
%%{\scriptstyle\overline{\lfloor 
%%n/2\rfloor+1},\dots,\overline{n-2}},\bolddd{$\overline{n{-}1},\overline{n}$}
\\ \dbl{\SL_n^\gee(q)\I20}{(n\ge3)\I20} & \D \prod_{i=2}^n ((\gee q)^i-1) & 
(n,q-\gee) & {\scriptstyle\lfloor n/2\rfloor+1,\dots,n-2},\comp{n-1,n}
\\ \dbl{\Sp_{2n}(q),\hfill}{\Spin_{2n+1}(q)} & \D \prod_{i=1}^n (q^{2i}-1) & 
(2,q-1) & 
\dbl{{\scriptstyle\lfloor n/2\rfloor+1,\dots,n-1},\comp{n} ~\textup{(odd only)}}
{{\scriptstyle{n+1,\dots,2n-2}},\comp{2n}~\textup{(even only)}}
\\ \Spin_{2n}^+(q) & \D  (q^n-1)\prod_{i=1}^{n-1} (q^{2i}-1) & 
\dbl{(4,q^n-1)\textup{ or}}{(2,q-1)^2} &
\dbl{{\scriptstyle\lfloor n/2\rfloor+1,\dots,}\comp{n-1,n} ~\textup{(odd only)}}
{{\scriptstyle{n+1,\dots,2n-4}},\comp{2n-2}~\textup{(even only)}}
\\ \Spin_{2n}^-(q) & \D  (q^n+1)\prod_{i=1}^{n-1} (q^{2i}-1) & 
(4,q^n+1) &
\dbl{{\scriptstyle\lfloor n/2\rfloor+1,\dots,}\comp{n-1} ~\textup{(odd only)}}
{{\scriptstyle{n,\dots,2n-4}},\comp{2n-2,2n}~\textup{(even only)}}
\\ \hline
\end{array} \]
\end{small}\caption{Orders of classical groups of Lie type (regular 
$d$ in bold)}\label{tab:classical}
\end{table}

\begin{table}[ht]\begin{small}\renewcommand{\arraystretch}{1.2}
\[ \begin{array}{|cccc|}
\hline G & |G|_{q'} & |Z(G)| & \textup{$d$ with $a_d=1$}
\\\hline \lie2B2(q) & \Phi_1\Phi_4 & 1 & \comp{1,4',4''} 
\\ \lie2G2(q) & \Phi_1\Phi_2\Phi_6 & 1 & \comp{1,2,6',6''}
\\ \lie2F4(q) & \Phi_1^2\Phi_2^2\Phi_4^2\Phi_6\Phi_{12} & 1 & \comp{6,12',12''}
\\ G_2(q) & \Phi_1^2\Phi_2^2\Phi_3\Phi_6 & 1 & \comp{3,6}
\\ \lie3D4(q) & \Phi_1^2\Phi_2^2\Phi_3^2\Phi_6^2\Phi_{12} & 1 & \comp{12}
\\ F_4(q) & \Phi_1^4\Phi_2^4\Phi_3^2\Phi_4^2\Phi_6^2\Phi_8\Phi_{12} & 1 & \comp{8,12}
\\ E_6(q) & \Phi_1^6\Phi_2^4\Phi_3^3\Phi_4^2\Phi_5\Phi_6^2\Phi_8\Phi_9\Phi_{12}& 
(3,q-1) & {\scriptstyle5},\comp{8,9,12}
\\ \lie2E6(q) & \Phi_1^4\Phi_2^6\Phi_3^2\Phi_4^2\Phi_6^3\Phi_8\Phi_{10}\Phi_{12}\Phi_{18} & 
(3,q+1) & \comp{8},{\scriptstyle10},\comp{12,18}
\\ E_7(q) & 
\dbl{\Phi_1^7\Phi_2^7\Phi_3^3\Phi_4^2\Phi_5\Phi_6^3\Phi_7\qquad\qquad}
{\qquad\qquad\Phi_8\Phi_9\Phi_{10}\Phi_{12}\Phi_{14}\Phi_{18}} & (2,q-1) & 
{\scriptstyle5},\comp{7},{\scriptstyle8},\comp{9},{\scriptstyle10,12},\comp{14,18}
\\ E_8(q) & 
\dbl{\Phi_1^8\Phi_2^8\Phi_3^4\Phi_4^4\Phi_5^2\Phi_6^4\Phi_7\Phi_8^2\Phi_9\qquad\qquad}
{\qquad\qquad\Phi_{10}^2\Phi_{12}^2\Phi_{14}\Phi_{15}\Phi_{18}\Phi_{20}\Phi_{24}\Phi_{30}}
& 1 & {\scriptstyle7,9,14},\comp{15},{\scriptstyle18},\comp{20,24,30}
\\ \hline
\end{array} \]
\end{small}
\caption{Orders of exceptional groups of Lie type (regular 
$d$ in bold)}\label{tab:exceptional}
\end{table} 

In Table \ref{tab:exceptional}, the values $d=4',4''$, etc. represent the 
factorizations of the cyclotomic polynomials $\Phi_d(q)$:
	\begin{align*} 
	G&=\lie2B2(q) & q&=2^{2m+1} & d&=4 & \Phi_4(q) &= 
	(q+\sqrt{2q}+1)(q-\sqrt{2q}+1) \\
	G&=\lie2G2(q) & q&=3^{2m+1} & d&=6 & \Phi_6(q) &= 
	(q+\sqrt{3q}+1)(q-\sqrt{3q}+1) \\
	G&=\lie2F4(q) & q&=2^{2m+1} & d&=12 & \Phi_{12}(q) &= 
	(q^2+q\sqrt{2q}+q+\sqrt{2q}+1) 
	\\ &&&&&&& \hskip15mm
	\cdot(q^2-q\sqrt{2q}+q-\sqrt{2q}+1) .
	\end{align*}
These factors are the orders of the largest cyclic subgroups in $G$ of 
order dividing $\Phi_d(q)$.

We claim that 
	\beqq G/Z(G)\in\G \quad\implies\quad p\nmid z \textup{ and hence } 
	G\in\G. \label{e:p|z} \eeqq
Assume otherwise: let $G={}^r\gg(q)$ be a counterexample. Since $p$ is odd 
and $p\mid z$, $z$ is not a power of $2$, and hence $\gg=A_n$ or $E_6$. If 
$\gg=E_6$, then $z=(3,q\pm1)$, so $p=3$, which is impossible since 
$3^4\bmid(q^2-1)^4\bmid|G|$ for all $q$ prime to $3$ (Table 
\ref{tab:exceptional}). Thus $G=\SL_n^\gee(q)$ and 
$p\mid z=(n,q-\gee)$. So $n\ge3$ and $(q-\gee)^3\nmid|G|$, which by Table 
\ref{tab:classical} implies $n=p=3$. But then $3\mid(q-\gee)$ implies 
$3^2\mid(q^3-\gee)$, so $3^3\bmid|G|$ and $3^2\bmid|G/Z(G)|$. 
This proves \eqref{e:p|z}. In particular, since $G\in\G$, 
	\beqq \textup{there exists a unique $d$ such that $p\mid\Phi_d(q)$, 
	and for this $d$, $p^2\nmid \Phi_d(q)$, and $a_d=1$.} 
	\label{e:unique-d} \eeqq

We now want to understand the subgroups $N=N_G(\UUU)$ and $C=C_G(\UUU)$, 
which are closely related to the integer $d$, and not really dependent on the prime $p$.

We start by assuming that $\gg$ is simply connected, e.g., $\SL_n$. In this 
case, a theorem of Steinberg (see \cite[Theorem 14.16]{malletest}) states 
that, for each semisimple element $s\in \gg$, the centralizer $C_\gg(s)$ is 
connected. A semisimple element $s\in G$ is \emph{regular} if, inside the 
corresponding algebraic group $\gg$, $\dim(C_\gg(s))$ is minimal among all 
semisimple elements, so equal to the rank of $\gg$. By \cite[Corollary 
14.10]{malletest}, if $s$ is regular then $C_\gg(s)^0=\mathbb{T}$ for any 
maximal torus $\mathbb{T}$ containing $s$, and hence $s$ is 
contained inside a unique maximal torus $\mathbb{T}$ and 
$C_\gg(s)=\mathbb{T}$.

Let $F$ be a Frobenius endomorphism such that $G=\gg^F$, and assume for the 
moment that $G$ is not a Ree or Suzuki group. To continue we need one fact 
from the theory of $d$-tori. Rather than giving a formal definition of 
$d$-tori here, we instead refer to \cite[Section 25]{malletest}, and in 
particular to \cite[Definition 25.6]{malletest}. The property we need 
is that by \cite[Theorems 25.14 and 25.19]{malletest}, if $a_d=1$, then 
for the finite group $G=\gg^F$, there exists a cyclic torus 
$T_d=\mathbb{T}_d^F$ (called a \emph{Sylow $d$-torus}), where 
$\mathbb{T}_d$ is an $F$-stable torus in $\gg$, such that $\UUU\leq T_d$ 
and $C_G(\UUU)=C_G(T_d)$.

Let $T_d=\langle s\rangle$. If $s$ is regular then we know that $C_\gg(s)$ is a maximal torus of $\gg$, and so $C_G(\UUU)=C_G(s)\leq C_\gg(s)$ is abelian. Note that in general, if $s$ is regular then, although $C_G(s)$ is a torus, we have $C_G(s)\geq T_d$, with equality if and only if $T_d$ is a maximal torus. This last statement is true if and only if $\phi(d)=\mathrm{rank}(G)$.

If $G$ is a Suzuki or Ree group then the maximal subgroups of $G$ are known (see for example \cite{wilsonbook}), so we can deduce the same result for those groups and for our particular primes $p$, and get $C_G(\UUU)$ abelian in all cases for Table \ref{tab:exceptional}. (The $d$-torus theory can be extended to these groups with some complications coming from the fact that cyclotomic polynomials split, but it is easier for us to use the lists of maximal subgroups directly to prove that $C_G(\UUU)$ is abelian.)

We have now shown the following:

\begin{prop} \label{reg=>CGU-abel}
Let $G={}^r\gg(q)$ be the universal form of a group of Lie type (so 
a quasisimple group as given in Tables \ref{tab:classical} and 
\ref{tab:exceptional}), and suppose that $G\in \G$ and $p\nmid q$. Let $d$ 
be the multiplicative order of $q$ modulo $p$. If $d$ is a regular number 
for $G$ then $C_G(\UUU)$ is abelian.
\end{prop}

The numbers $d$ that yield regular elements $s$ for a given group $\gg$ and Frobenius endomorphism $F$ were computed by Springer in \cite[Section 5, 6.9--6.11]{springer}. They are given in Tables \ref{tab:classical} and \ref{tab:exceptional}, where we list all $d$ such that $a_d=1$ (i.e., such that the Sylow $\Phi_d$-subgroup has order $p$), and among them list in bold the $d$ that are regular.

In view of Proposition \ref{reg=>CGU-abel}, we would like to always be in the situation that $d$ is a regular number, but from Tables \ref{tab:classical} and \ref{tab:exceptional} we see that this is not true. However, as we will see, there is always a subgroup $H$ of $G$, also a group of Lie type, such that $\UUU\leq H$ and $d$ is regular for $H$. 

For classical groups at least, we also need to understand the structure of $N_G(\UUU)$, not just of $C_G(\UUU)$, although this is easy and again depends only on $d$. For $d$ a positive integer, let $\4 d$ be defined by $\4 d=2d$ if $d$ is odd, $\4 d=d/2$ if $d$ is even but $4\nmid d$, and $\4 d=d$ if $4\mid d$. Note that if $d$ is the multiplicative order of $q$ modulo $p$, then $\4 d$ is the order of $-q$ modulo $p$. 

For $\SL_2(q)$, the automizer is always of order $2$. 
For $\SL_n(q)$ and $\SU_n(q)$ for $n\geq 3$, by \cite[Proposition 
(3D)]{fongsrin1}, $|N/C|=d$ or $\4 d$, respectively. For the 
other classical groups, by \cite[p.128]{fongsrin2}, $|N/C|$ equals
$2d$ if $d$ is odd or $d$ if $d$ is even, i.e., $\lcm(d,2)$ in all cases.

The order of $N/C$, when $p\mid\Phi_d(q)$, is summarized for the classical groups in Table \ref{tab:automizers}.

%%In Table \ref{tab:automizers}, we assume that $p\mid \Phi_d(q)$.

\begin{table}[ht]
\begin{center} \renewcommand{\arraystretch}{1.2}
\begin{tabular}{|c|c|}
\hline Group & $|N/C|$ when $p\mid\Phi_d(q)$ 
\\ \hline $\PSL_2(q)$ & $2$
\\ $\PSL_n(q)$ ($n\geq 3$) & $d$
\\ $\PSU_n(q)$ & $\4 d$
\\ $\PSp_{2n}(q)$ or $P\varOmega_m^\pm(q)$ & $\lcm(2,d)$
\\ \hline
\end{tabular}
\end{center}
\caption{}
\label{tab:automizers}
\end{table}

\iffalse
For linear and unitary groups, there are many examples of minimally active 
modules for groups not in $\GG$: for example, the Weil representations of 
$\SL_n(q)$ (for all $n$). 
\fi

\iffalse
In order for our induction to work well, we prove 
now that for $\SL_n(q)$ and $\SU_n(q)$, generically there is a unique 
$\Phi_d(q)$ that $p$ may divide, this being $d=n$ and $d=\overline n$ 
respectively.
\fi

When we work with cyclotomic polynomials, the following relations are useful:
\begin{enumi}
\item for all $n>1$, $\Phi_n(q)\mid (q^n-1)/(q-1)$;
\item if $n$ is even but $4\nmid n$, then $\Phi_n(q)\mid (q^{n/2}+1)/(q+1)$; and
\item if $4\mid n$ then $\Phi_n(q)\mid q^{n/2}+1$.
\end{enumi}

We are now ready to work with the individual groups: first the classical 
groups and then the exceptional groups.

\iffalse
Having done reductions for classical groups, we are now in a position to completely determine all minimally active modules for classical groups in $\GG$, which will be accomplished in the next section.
\fi

%%\newpage

\section{Determination for classical groups}
\label{s:classical}

In this section, we classify minimally active modules for classical groups and 
their extensions in $\GG$ when $p$ is not the defining characteristic. Throughout this section, $G=G(q)$ is a group of Lie type, $p\nmid q$ is a prime dividing $|G|$, $\UUU$ is a Sylow $p$-subgroup of $G$ with generator $x$, and
\begin{center}
$p$ has order $d$ modulo $q$, so that $p\mid\Phi_d(q)$.
\end{center}

Let $V$ be a minimally active module for $G$. In Table 
\ref{tab:classical2}, we list the minimal possible dimension $l(G)$ 
for a faithful $\k G$-module, as determined in \cite{landazuriseitz1974}, 
\cite{seitzzalesskii1993}, or \cite{guralnicktiep}. 

\begin{table}[ht]
\renewcommand{\arraystretch}{1.3}
\[ \begin{array}{|c|c|c|c|}
\hline G & \textup{Lower bound for dimension} & \textup{Ref.} & \textup{Exceptions}
\\\hline \SL_2(q) &  (q-1)/z \quad (z=(2,q-1)) & \textup{[LS]}  & 
\dbl{l(\SL_2(4))=2}{l(\SL_2(9))=3}
\\\hline \halfup[-1]{\SL_n(q)} & q^{n-1}-1 & \textup{[LS]} & 
\dbl{l(\SL_3(2))=2}{l(\SL_3(4))=4}
\\ \halfup[1]{(n\ge3)}  & q(q^{n-1}-1)/(q-1)-1 & \textup{[GT]} & \dbl{l(\SL_4(2))=7}{l(\SL_4(3))=26}
\\\hline \dbl{\SU_n(q)}{(n\ge3)} & 
\begin{array}{cl}(q^n-q)/(q+1)&\textup{($n$ odd)}\\
(q^n-1)/(q+1)&\textup{($n$ even)}\end{array} & \textup{[LS]}   & 
\dbl{l(\SU_4(2))=4}{l(\SU_4(3))=6}
%%\\\hline \Sp_{2n}(q)~(2\nmid q) & (q^n-1)/2 & \textup{[LS]} & \textup{---}
%%\\\hline \Sp_{2n}(q)~(2\mid q) & 
%%\dfrac{q(q^n-1)(q^{n-1}-1)}{2(q+1)} & \textup{[SZ]} & 
%%\dbl{l(\Sp_4(2)')=2}{l(\Sp_6(2))=7}
\\\hline \Sp_{2n}(q) & 
\begin{array}{cl}(q^n-1)/2&(2\nmid q)\\
q(q^n-1)(q^{n-1}-1)\big/2(q+1)&(2\mid q) \end{array}
& \textup{[SZ]} & l(\Sp_4(2)')=2
\\\hline \dbl{\Spin_{2n+1}(q)}{(n\ge3,~2\nmid q)} & q^{n-1}(q^{n-1}-1) 
& \textup{[LS]} & l(\Spin_7(3))=27
\\\hline \dbl{\Spin_{2n}^+(q)}{(n\ge4)} & q^{n-2}(q^{n-1}-1) & \textup{[LS]} & 
l(\Spin_8^+(2))=8 
\\\hline \dbl{\Spin_{2n}^-(q)}{(n\ge4)} & (q^{n-1}+1)(q^{n-2}-1) & \textup{[LS]} & 
\textup{---}
\\\hline
\end{array} \]
\caption{Minimal representation dimensions of classical groups}
\label{tab:classical2}
\end{table}

We begin with the linear groups.

\begin{prop}\label{prop:finishedsln} 
Let $G\in\GG$ be a group of type $\PSL_n(q)$, where $n\ge2$ and $p\nmid q$, 
and set $G_0=E(G)$. Let $V$ be a non-trivial, minimally active 
$\F_pG$-module. Then one of the following holds:
\begin{enumi}
\item $G_0$ is a central extension of $\PSL_2(4)\cong A_5$, $\PSL_2(5)\cong 
A_5$, $\PSL_2(9)\cong A_6$ or $\PSL_4(2)\cong A_8$, and the modules are as 
in Proposition \ref{p:alternating};

\item $G_0$ is a central extension of $\PSL_2(7)\cong \PSL_3(2)$, $p=7$, 
and $V$ is one of the modules in Proposition \ref{p:gl2p};

\item $G_0\cong\SL_2(8)$, $G/Z(G)\cong\SL_2(8):3$, $p=7$, and 
$\dim(V)=7,8$;

\item $G_0\cong6\cdot\PSL_3(4)$, $G/Z(G)\cong\PSL_3(4).2$, $p=7$, and 
$\dim(V)=6$.
\end{enumi}
\end{prop}

\begin{proof} In all cases, we fix $n$ and $q=r^t$, where $r$ is prime 
and $t\ge1$, such that $G$ is of type $\PSL_n(q)$. If $(n,q)=(2,4)$, 
$(2,5)$, $(2,9)$, or $(4,2)$, then $\PSL_n(q)$ is an alternating group, and 
we are in the situation of (i). So we assume from now on that $(n,q)$ is 
not one of these pairs. Since $\PSL_3(2)\cong\PSL_2(7)$, we can also assume 
that $(n,q)\ne(3,2)$. 

Among the remaining cases, $\SL_n(q)$ is the universal central extension of 
$\PSL_n(q)$ with only one exception: $\PSL_3(4)$ has an ``exceptional 
cover'': a central extension of the form $4^2\cdot\SL_3(4)$ 
\cite[\S\,3.12]{wilsonbook}. In all other cases, if $p\mid\Phi_d(q)$ for 
regular $d$, then $C_{G_0}(\UUU)$ is abelian by Proposition 
\ref{reg=>CGU-abel}, and hence by Proposition \ref{prop:dimlt2p-1}(a), 
$\dim(V)\le2p-1$ for each indecomposable minimally active $\F_pG_0$-module 
$V$.

\smallskip

\textbf{Case 1: } We start with the groups of type 
$\PSL_2(q)$, where $q=r^t$. Thus $p\mid\Phi_d(q)$ where 
$d=1,2$, and the normalizer $N_{G_0}(\UUU)$ is dihedral of order 
$\Phi_d(q)$ or quaternion of order $2\cdot\Phi_d(q)$, containing 
$C_{G_0}(\UUU)$ as a cyclic subgroup of index $2$. Hence if $V$ is minimally active, then $\dim(V)\leq p+2$ by Proposition \ref{prop:dimlt2p-1}(a). 

In the group $\Aut(\SL_2(r^t))$, the automizer of $\UUU$ has order at most $2t$, so that if $G\in \GG$ then $p\leq 2t+1$. On the other hand, the smallest dimension for $V$ is $(r^t-1)/2$ for $r$ odd and $r^t-1$ for $r$ even, so that $\dim(V)\leq p+2$ becomes
\begin{align*} 
r^t &\le \dim(V)+1 \le p+3 \le 2t+4 & &\textup{if $r=2$} \\
r^t &\le 2\dim(V)+1 \le 2(p+2)+1 \le 4t+7 & &\textup{if $r>2$.}
\end{align*}
If $r=2$, then $t\le3$; while if $r$ is odd, then either $t=2$ and $r=3$, 
or $t=1$ and $r\le11$. 

When $G_0\cong\PSL_2(11)$ and $p=3$, there is a $5$-dimensional $\F_3G_0$-module 
$V$, but $N_{G_0}(\UUU)/\UUU\cong C_2^2$ is abelian in this case, so 
$V$ is not minimally active by Proposition \ref{prop:dimlt2p-1}(a). Since 
we have already dealt with the alternating groups 
$\PSL_2(4)\cong\PSL_2(5)\cong A_5$ and $\PSL_2(9)\cong A_6$, and do not 
need to deal with solvable groups, we are left with the groups $\PSL_2(7)$ 
with $p=3$, and $\PSL_2(8)$ with $p=7$.

For $G\cong\SL_2(7)$ and $p=3$, the modules of dimension $3$ are not defined over $\F_3$, 
but only over $\F_9$. For $G\cong\SL_2(8)$ and $p=7$, there are four 
$7$-dimensional modules of which only one extends to $\SL_2(8):3\in\GGp7$, 
and one $8$-dimensional module which also extends to $\SL_2(8):3$. Both of 
these are minimally active.

\smallskip

\noindent\textbf{Case 2: } We next consider the groups of type $\PSL_n(q)$ for $n\geq 3$, where $p\mid\Phi_d(q)$ for some 
$\lowint{\frac{n}2}+1\le d\le n-1$. By Table \ref{tab:classical2}, 
$\dim(V)\geq q(q^{n-1}-1)/(q-1)-1$ with the exceptions in Table 
\ref{tab:classical2}, which (aside from cases that we already eliminated) are 
\[(n,q,p)\in \{(3,4,5),(4,3,13)\}.\]
(Here we consider those $\PSL_n(q)$ in the table, together with $p$ such that $p\mid\Phi_d(q)$ with $\lowint{\frac{n}2}+1\le d\le 
n-1$ and $p^2\nmid |\PSL_n(q)|$.) Furthermore, if $V$ is the reduction modulo $p$ of a complex character (this is true if $\dim(V)\geq p+1$ by Proposition \ref{prop:collatedresults}(d)) then $\dim(V)\geq q(q^{n-1}-1)/(q-1)$ with the exceptions above.

%%\[(n,q,p)\in \{(3,2,3),(3,4,5),(4,2,7),(4,3,13)\}.\]

Suppose firstly that $d=n-1$. If $(n,q,p)=(3,4,5)$, then $\UUU\cong 
C_5$ is self-centralizing in $\PSL_3(4)$, so that $C_{G_0}(\UUU)$ is 
abelian when $G_0$ is any central extension. Hence $C_G(\UUU)$ is abelian 
and $\dim(V)\leq 2p-1$ in all cases. With the exceptions above, we get
\[ q(q^{n-1}-1)/(q-1)-1\leq \dim(V)\leq 2p-1\leq 2\Phi_{n-1}(q)-1\leq 2(q^{n-1}-1)/(q-1)-1.\]
Thus $q=2$ and $\dim(V)=2p-1$, whence $\dim(V)>p$ and so $V$ is the reduction of a complex character. So $\dim(V)\geq q(q^{n-1}-1)/(q-1)$ by Table \ref{tab:classical2}, which is a contradiction.

If $G_0/Z(G_0)\cong\PSL_3(4)$ and $p=5$, then the smallest projective 
representation has dimension $6$ (see \cite{modatlas}), but this is for 
$6.\PSL_3(4)$, and as the centre has size $6$ the two $6$-dimensional 
modules are not realizable over $\F_5$, so there are no minimally active 
modules. For $G_0$ a central extension of $\PSL_4(3)$ and 
$p=13=\Phi_3(3)$, $\dim(V)\geq 2p$ and so there are no minimally active 
modules (since $d=3$ is regular).

If $\lowint{n/2}+1\le d<n-1$ and $V$ is minimally active for $\SL_n(q)$, 
then $n\ge5$, and there is $H\cong\SL_{n-1}(q)$ such that $\UUU\le H$, 
where $V|_H$ is also minimally active. So by what was just shown, 
$H\cong\SL_4(2)$ and $p=7$, and $G_0\cong\SL_5(2)$. In this case, 
$N_{G_0}(\UUU)/\UUU\cong3\times S_3$, so the Green correspondent of $V$ has 
dimension at most $2$, and $\dim(V)\le p+2=9$, which contradicts both 
bounds in Table \ref{tab:classical2}.

\smallskip

\noindent\textbf{Case 3: } We now assume $G$ is of type $\PSL_n(q)$, where $n\geq 3$ and $p\mid\Phi_n(q)$. If $G\in \GG$, then by allowing for a possible graph automorphism and a field automorphism of order $t$, we have that $|\Aut_G(\UUU)|\le2nt$, so $p\leq 2nt+1$. On the other hand, by Table \ref{tab:classical2}, with the exceptions of $\SL_3(2)$ and 
$\SL_3(4)$, the dimension of any minimally active module $V$ is at least 
$r^{t(n-1)}-1$ (as the group is $\SL_n(r^t).2.t$), whence $\dim(V)<2p$ 
(which also holds when $G_0$ is a central extension of $\PSL_3(4)$ and 
$p=7$) becomes
	\[ r^{2nt/3} \le r^{t(n-1)} \le \dim(V)+1 < 2p+1\le 4nt+3. \]
Set $\ell=nt\ge3$. Then $r^{2\ell/3}\le4\ell+2$, and from this we see that 
$r=2$ implies $\ell<9$ (hence $t\le2$), $r=3$ implies $\ell<6$ (hence 
$t=1$), and $r\ge5$ is impossible. Upon returning to the inequality 
$r^{t(n-1)}\le4nt+3$, we have $q=r^t=4$ implies $n=3$, $q=3$ implies $n=3$, 
and $q=2$ implies $n\le5$. Since we are assuming that $(n,q)\ne(3,2)$ 
or $(4,2)$, we are left with the following possibilities:
	\[ \renewcommand{\arraystretch}{1.3}
	\begin{array}{c|ccc}
	(n,q) & (5,2) & (3,3) & (3,4) \\\hline
	\Phi_n(q) & 31 & 13 & 21 
	\end{array} \]
Note that this list includes the remaining exception to the Landazuri--Seitz bounds.

%%is of type $\SL_n(q)$, where $q=r^t$ and $r$ is prime, 

\iffalse
The group $\SL_4(2)\cong A_8$ is as in Proposition \ref{p:alternating}, and 
$\SL_3(2)\cong\PSL_2(7)$ is handled in case 1. 
\fi

We can eliminate groups of type $\PSL_5(2)$ 
for $p=31$ since $\Aut(\PSL_2(5))\notin\GGp{31}$, and those of type 
$\PSL_3(3)$ for $p=13$ since $\Aut(\PSL_3(3))\notin\GG[13]$. 
It remains to consider $\PSL_3(4)$ and its covers when $p=7$. In this case, 
since $N_{\PSL_3(4)}(\UUU)/\UUU$ has order $3$, $N_G(\UUU)/\UUU$ is abelian 
in all cases, and hence $\dim(V)\le p+1=8$ by Proposition \ref{prop:dimlt2p-1}(a). We have modules of dimension $8$ for 
$4_1\cdot\PSL_3(4)$ (but these cannot be defined over $\F_7$ as the centre 
has order $4$), and of dimension $6$ for $6\cdot \PSL_3(4)$, which are 
indecomposable on restriction to $\UUU$. 
\end{proof}

We now turn to unitary groups. 

\begin{prop}\label{prop:finishedsun} 
Let $G\in\GG$ be a group of type $\PSU_n(q)$, where $n\ge3$ and $p\nmid q$, 
and set $G_0=E(G)$. Let $V$ be a non-trivial, minimally active 
$\F_pG$-module. Then one of the following holds: either
\begin{enumi}
\item $G_0\cong\PSU_3(3)$, $G/Z(G)\cong\PSU_3(3).2$, $p=7$, and $\dim(V)=6,7$; or 
\item $G_0\cong\PSU_3(4)$, $G/Z(G)\cong\PSU_3(4):4$, $p=13$, and $\dim(V)=12$; or
\item $G_0\cong G/Z(G)\cong\PSU_4(2)$, $p=5$, and $\dim(V)=6$; or
\item $G_0\cong6_1\cdot\PSU_4(3)$, $G/Z(G)\cong\PSU_4(3).2_2$ 
($G$ contains the complex reflection group $G_{34}$), $p=7$, and $\dim(V)=6$; 
or
\item $G_0\cong\PSU_5(2)$, $G/Z(G)\cong\PSU_5(2).2$, $p=11$, and $\dim(V)=10$.
\end{enumi}
\end{prop}

\begin{proof} Let $G$ be of type $\PSU_n(q)$, where $q=r^t$ and $r$ is 
prime. Since $\SU_3(2)$ is solvable, we assume $(n,q)\ne(3,2)$. Among 
the other cases, $\SU_n(q)$ is the universal central extension of 
$\PSU_n(q)$ with exactly three exceptions: $\PSU_4(2)$, $\PSU_6(2)$, and 
$\PSU_4(3)$ \cite[\S\,3.12]{wilsonbook}. In all other cases, if 
$p\mid\Phi_d(q)$ for regular $d$, then $C_{G_0}(\UUU)$ is abelian, and 
$\dim(V)\le2p-1$ when $V$ is indecomposable and minimally active, by 
Propositions \ref{reg=>CGU-abel} and \ref{prop:dimlt2p-1}(a).

By Table \ref{tab:classical2}, $\dim(V)\ge(q^n-q)/(q+1)$ (whether $n$ 
is even or odd), with just two exceptions. 

\smallskip

\noindent\textbf{Case 1: } Assume that $p\mid\Phi_d(-q)=\Phi_{\4d}(q)$ for 
some $\lowint{n/2}+1\le d\le n-1$. If $d=n-1$, then by the above 
remarks, either $G_0/Z(G_0)\cong\PSU_4(2)$, $\PSU_4(3)$, or $\PSU_6(2)$, or 
	\[ q(q^{n-1}-1)/(q+1) \le \dim(V) < 2p \le 2(q^{n-1}+1)/(q+1) . \]
From this, it follows immediately that $q=2$ and $\dim(V)=2p-1$. Since $V$ 
is minimally active and $C_{G_0}(\UUU)$ is abelian, this last condition 
implies that $|\Aut_{G_0}(\UUU)|=p-1$ by Proposition 
\ref{prop:dimlt2p-1}(a). Hence $n-1\ge p-1$, so $n\ge p>(2^{n-1}-1)/3$, and 
$n\le4$. But the group $\SU_3(2)$ is solvable, and $3^2\bmid|\PSU_4(2)|$ 
($3=\Phi_3(-2)$), so both of these are eliminated.

This leaves the two other exceptional cases. If 
$G_0/Z(G_0)\cong\PSU_6(2)$ and $p\mid\Phi_5(-2)=11$, then $\UUU$ is 
self-centralizing in $\PSU_6(2)$, and so $C_{G_0}(\UUU)$ is abelian for 
each cover $\UUU$. Thus $\dim(V)\le2p-1=21$, which contradicts Table \ref{tab:classical2}. If $G_0/Z(G_0)\cong\PSU_4(3)$ and 
$p\mid\Phi_3(-3)=7$, then since $|N_{G_0}(\UUU)/\UUU|=3$, $\dim(V)\le 
p+1=8$ when $V$ is minimally active (Proposition \ref{prop:dimlt2p-1}(a)). 
By \cite[p.137]{modatlas}, there is a single $6$-dimensional module for 
$G_0\cong6_1\cdot\PSU_4(3)$ over $\F_7$ (using ATLAS notation for the 
central extension), it extends to $6_1\cdot\PSU_4(3).2_2\in\GG[7]$, and all 
other modules are of dimension larger than $8$. 

If $\lowint{n/2}+1\le d<n-1$, then $n\ge5$, and there is $H<G_0$ such 
that $H/Z(H)\cong\PSU_{n-1}(q)$ such that $\UUU\le H$ and $V|_H$ is still 
indecomposable and minimally active. Hence $(n,q)=(5,3)$, $p=7$, and 
$H\cong6_1.\SU_4(3)$. Since this central extension is not a subgroup of 
$\SU_5(3)$ (and this group has no central extensions), there are no 
minimally active modules for $p=7$ and $G=\SU_5(3)$.

\smallskip

\noindent\textbf{Case 2: } 
Now assume that $p\mid\Phi_n(-q)=\Phi_{\4n}(q)$, so that 
$|\Aut_G(\UUU)|\le2nt$ by Table \ref{tab:automizers}. Thus $p\le2nt+1$ and 
$\dim(V)<2p$ imply (with the three exceptions noted above)
	\[ r^t(r^{t(n-1)}-1)\leq (4nt+2)(r^t+1)\]
and hence (since $n\ge3$)
	\[ r^{2tn/3} \le r^{t(n-1)} \le 1+(4nt+2)(1+\tfrac1{r^t}) \le 
	(4nt+3)(1+\tfrac1r). \]

Set $\ell=nt$; we thus have $r^{2\ell/3}\le\frac{r+1}r(4\ell+3)$. When 
$r=2$, this implies $\ell<9$ and hence $t\le2$; when $r=3$ it implies 
$\ell<6$ and hence $t=1$, and there are no solutions for $r\ge5$ and 
$\ell\ge3$. If we now go back to the original inequality, we see that the 
only solutions (including the exceptional cases) are the following ones:
	\[ \renewcommand{\arraystretch}{1.3}
	\begin{array}{r|cccccc}
	(n,q) & (4,2) & (5,2) & (6,2) & (3,3) & (3,4) & (4,3) \\\hline
	\Phi_n(-q) & 5 & 11 & 7 & 7 & 13 & 10
	\end{array} \]
Here, we omit the pair $(n,q)=(3,2)$ since $\SU_3(2)$ is solvable.
\begin{itemize} 

\item For $\PSU_4(2)$ and $p=5$, there are two $5$-dimensional modules 
whose irrationalities in their Brauer characters \cite[p.62]{modatlas} 
imply that they need $\F_{25}$ from \cite[p.288]{modatlas}, and there is a 
single $6$-dimensional module, which is minimally active for $\PSU_4(2)\in 
\GGp{5}$. 

\item For $\PSU_5(2)$ and $p=11$, we see from \cite[p.184]{modatlas} that 
there is a module of dimension $10=p-1$ and two of dimension $11$, 
amalgamating over $\PSU_5(2).2\in \GGp{11}$. So the proposition holds in 
these cases.

\item For $\PSU_6(2)$, the Landazuri--Seitz bound gives $\dim(V)\ge21$, so 
there are no minimally active modules when $p=7$.

\item For $\PSU_3(3)$ and $p=7$, we need $\PSU_3(3).2$ to be in 
$\GGp7$. By \cite[p.24]{modatlas} there is a single module for 
$G_0=\PSU_3(3)$ of dimension $6$ and one of dimension $7$ (the two other 
dual modules of dimension $7$ amalgamate for $G$). 

\item For $\PSU_3(4)$ and $p=13$, there is by \cite[p.73]{modatlas} a 
single module of dimension $12$, extendible to $\PSU_3(4).4\in\GGp{13}$, 
and four modules of dimension $13$, amalgamating into a single 
$52$-dimensional for $\PSU_3(4).4$. 

\item For $\PSU_4(3)$ and $p=5$, there are many covers, as the Schur 
multiplier is of order $36$. However, for our module to be definable over 
$\F_5$, the centre of $G$ must have order dividing $4$, and so $G_0$ is a 
quotient group of $\SU_4(3)$, and its smallest 
simple module has dimension $20$ \cite[p.128]{modatlas}. Thus $\dim(V)\le2p-1=9$ by Propositions \ref{reg=>CGU-abel} and \ref{prop:dimlt2p-1}(a), a contradiction.

\end{itemize}

\iffalse
It remains to consider groups of type $\PSU_4(3)$: for $p=7\mid\Phi_3(-3)$ (by 
Step 1), and for $p=5\mid\Phi_4(-3)$ since this is an exception 
to the Landazuri--Seitz bounds (Table \ref{tab:classical2}). When $p=7$, by 
\cite[p.137]{modatlas}, there is a single 6-dimensional module for 
$G_0=6_1\cdot\PSU_4(3)$, and it extends to $6_1\cdot\PSU_4(3).2_2\in\GGp7$. 
When $p=5$, there are many covers, as the Schur multiplier is of order 
$36$. However, for our module to be definable over $\F_5$, the centre of 
$G$ must have order dividing $4$, and so the smallest simple module has 
dimension $20$ \cite[p.128]{modatlas}. 
\fi

This completes the proof.
\end{proof}

\iffalse
Having proved that, with the exception of $2\cdot \Omega_8^+(2)$, there are 
no minimally active modules for orthogonal groups, to complete the 
classification for classical groups we need to do symplectic groups, which 
we do now.
\fi

We next consider the symplectic groups.

\begin{prop}\label{prop:finishedspn} 
Let $G\in\GG$ be a group of type $\PSp_{2n}(q)$, where $n\ge2$ and $p\nmid q$, 
and set $G_0=E(G)$. Let $V$ be a non-trivial, minimally active 
$\F_pG$-module. Then one of the following holds: either
\begin{enumi}
\item $G$ is of type $\Sp_4(2)'\cong A_6$ and $p=5$, and $V$ is as in Proposition \ref{p:alternating}; or
\item $G_0\cong G/Z(G)\cong\PSp_4(3)\cong \PSU_4(2)$, $p=5$, and $\dim(V)=6$; or
\item $G_0\cong\Sp_4(4)$, $G/Z(G)\cong\Sp_4(4).4$, $p=17$, and $\dim(V)=18$; or
\item $G_0\cong G/Z(G)\cong\Sp_6(2)$, $p=5$ or $p=7$, and $\dim(V)=7$; or 
\item $G_0\cong2\cdot\Sp_6(2)$, $G/Z(G)\cong\Sp_6(2)$, $p=7$ and $\dim(V)=8$.
\end{enumi}
\end{prop}

\begin{proof} In the first three cases, we assume that $p\mid\Phi_{2n}(q)$, 
or (if $n$ is odd) that $p\mid\Phi_n(q)$. In particular, in these cases, 
$p\leq (q^n\pm1)/(q\pm1)$ if $n$ is odd, and $p\leq (q^n+1)$ if $n$ is 
even. The remaining possibilities (where $d$ is not regular) are 
handled inductively in Case 4. 

Since the group $\Sp_4(2)'\cong A_6$ has already been handled in 
Proposition \ref{p:alternating}, we assume from now on that 
$(2n,q)\ne(4,2)$. The only (other) case in which $\Sp_{2n}(q)$ has a proper 
central extension is the group $\Sp_6(2)$.

\noindent\textbf{Case 1: } Suppose that $q$ is even. By Table 
\ref{tab:classical2} (and since $(2n,q)\ne(4,2)$), we have 
$\dim(V)\geq q(q^n-1)(q^{n-1}-1)/2(q+1)$. When $(2n,q)=(6,2)$, we have 
$p=7=\Phi_3(2)$ (since $\Phi_6(2)=3$ and $\Sp_6(2)\notin\G[3]$), $\UUU$ is 
self-centralizing in $\Sp_6(2)$, so $C_{G_0}(\UUU)$ is abelian when $G_0$ 
is any central extension of $\Sp_6(2)$, and the bound $\dim(V)<2p$ still applies.

Suppose firstly that $p\mid \Phi_n(q)$, where $n\geq 3$ is odd. The statement $\dim(V)<2p$ yields
	\[ q(q^n-1)(q^{n-1}-1)/2(q+1)<2(q^n-1)/(q-1).\]
If $q\geq 4$ then this can never happen, but if $q=2$ then this reduces to 
$(2^{n-1}-1)<6$, and hence $n=3$. Thus $(2n,q)=(6,2)$, and 
$p=\Phi_3(2)=7$. In this case, $\Sp_6(2)$ has a faithful 7-dimensional 
module and $2\cdot\Sp_6(2)$ a faithful 8-dimensional module, and both are 
minimally active.

If $p\mid \Phi_{2n}(q)$, then since $\dim(V)<2p$, we get 
	\[ q(q^n-1)(q^{n-1}-1)/2(q+1)<2(q^n+1),\]
whence $q(q^{n-1}-1)(q^n-1)<4(q+1)(q^n+1)$, and $q\le4$. If $q=4$, 
then $(4^n-1)(4^{n-1}-1)<5(4^n+1)$, which is satisfied for $n=2$ only. If 
$q=2$, we have $(2^n-1)(2^{n-1}-1)<6(2^n+1)$, which is satisfied for $n\leq 
3$ only, but $\Phi_6(2)=3$, so the case $(2n,q)=(6,2)$ can be 
eliminated. Since $(2n,q)\ne(4,2)$, we are left with the case 
$G_0\cong\Sp_4(4)$ and $p\mid\Phi_4(4)=17$. Then 
$\Aut(G_0)\cong G_0:4\in\GG[17]$, and there is a (unique) 
$18$-dimensional simple minimally active $\F_{17}G_0$-module for which 
the action extends to $\Aut(G_0)$. 

\iffalse
It remains to check $2\cdot \Sp_6(2)$ and $p=7$. Since $\dim(V)<2p$, we 
examine \cite[p.113]{modatlas} there are modules of dimension $7$ and $8$, 
with both minimally active.
\fi

\smallskip

\noindent\textbf{Case 2: }
If $q$ is odd and $n$ is odd, the smallest cross-characteristic 
representations are the Weil representations, of dimension $(q^n-1)/2$ (see 
Table \ref{tab:classical2} and \cite{gmst2002}). If $p\mid \Phi_n(q)$ or 
$\Phi_{2n}(q)$ then $p\leq (q^n\pm 1)/(q\pm 1)$, whence the inequality 
$\dim(V)<2p$ yields
	\[ (q^n-1)/2<2(q^n\pm 1)/(q\pm 1),\]
and hence $q=3$.

If $V$ is minimally active and $G=\Sp_{2n}(3)$ ($n$ odd) lies in $\GG$, and 
$p\mid\Phi_n(3)$, then since there can be no graph or field automorphisms, 
$\Aut_G(\UUU)$ has order $2n$, so that $p=2n+1$. On the other hand, 
$\dim(V)\geq (q^n-1)/2$, and since $q$ and $n$ are both odd, we have that 
$(n,q)=(3,3)$, yielding $p=2n+1=7=\Phi_6(3)$ and $V$ has dimension $13$, 
larger than $p+1$.

\smallskip

\noindent\textbf{Case 3: } Assume that $p\mid \Phi_{2n}(q)$, $q$ is odd, 
and $n$ is even. Assume that $q=r^t$ where $r$ is prime. If $G=\Sp_{2n}(q^t).t$, 
where $t$ is the order of a graph automorphism, then $\Aut_G(\UUU)$ has 
order at most $2nt$, so that $p\leq 2nt+1$. Thus $\dim(V)<2p\leq 4nt+2$, 
and using the fact that $\dim(V)\geq (q^n-1)/2$, we have that
	\[ r^{nt}-1<8nt+4.\]
As $r\geq 3$ and $n\geq 2$ is even, we have $nt=2$, so $t=1$ and 
$(n,q)=(2,3)$. But $\PSp_4(3)\cong \PSU_4(2)$, so this case has already 
been done (Proposition \ref{prop:finishedsun}).

\smallskip

\noindent\textbf{Case 4: } Now assume that $G\in\GG$ is of type 
$\PSp_{2n}(q)$, where $p\mid\Phi_d(q)$ for odd $d<n$ or even $d<2n$ (see 
Table \ref{tab:classical}), and that there is an indecomposable minimally 
active $\F_pG$-module $V$. Then there is $H<G$ of type $\PSp_{2n-2}(q)$ 
with $H\in\GG$, and $V|_H$ is again minimally active. So we can assume 
inductively that some indecomposable summand of $V|_H$ is already on our 
list.

Thus we next consider groups of type $\Sp_6(2)$ ($p=5$), $\PSp_6(3)$ 
($p=5$), $\Sp_6(4)$ ($p=17$), and $\Sp_8(2)$ ($p=7$). Since 
$\Aut(\Sp_6(4))\notin\GG[17]$, we can eliminate this case.

\iffalse
Having established all base cases of minimally 
active modules when $G$ is in $\GG$, we must now check inductively whether 
a minimally active module for $\Sp_{2n}(q)$ yields a minimally active 
module for $\Sp_{2n+2}(q)$. For this we need to consider $\Sp_6(2)$, $\Sp_6(3)$ and 
$\Sp_8(2)$, because for $\Sp_6(4)$, the automizer of the Sylow 
$17$-subgroup cannot have order $16$, since there is no graph automorphism 
this time.
\fi

Assume $G$ is of type $\Sp_6(2)$ and $p=5$. Then 
$N_{G_0/Z(G_0)}(\UUU)\cong5:4\times S_3$, contained in the subgroup 
$\Omega_6^+(2)\cong S_8$ in $\Sp_6(2)$. Hence $N_{G_0/Z(G_0)}(\UUU)/\UUU$ 
contains a cyclic subgroup of index $2$, and $N_{G_0}(\UUU)/\UUU$ contains 
an abelian subgroup of index $2$. So by Green correspondence, if $V$ is 
indecomposable and minimally active, then $\dim(V)\le p+2=7$. By 
\cite{modatlas}, there is a 7-dimensional module for $\Sp_6(2)$, and it is 
minimally active.

If $G$ is of type $\PSp_6(3)$ and $p=5$, the smallest faithful module has 
dimension at least $13$ by Table \ref{tab:classical2}, more than twice that 
of the minimally active module for $\Sp_4(3)$, so by Proposition 
\ref{prop:inactiverestriction}, there are no minimally active 
$\F_pG$-modules. Similarly, if $G$ is of type $\Sp_8(2)$ and $p=5$ or $7$, 
then $\dim(V)\ge28$ by Table \ref{tab:classical2}, this is more 
than twice the dimension of the modules we found for $\Sp_6(2)$, so this is 
again impossible by Proposition \ref{prop:inactiverestriction}. 
\end{proof}

It remains to handle the orthogonal groups.

\begin{prop}\label{lem:redorth}
Let $G\in\GG$ be a group of type $\Omega_{2n+1}(q)$ for $q$ odd and 
$n\ge3$, or of type $P\varOmega_{2n}^\pm(q)$ for $n\ge4$, where $p\nmid q$ in 
all cases. Set $G_0=E(G)$. Let $V$ be a non-trivial, minimally active 
$\F_pG$-module. Then $p=7$, $G_0\cong2\cdot\Omega_8^+(2)$, 
$G/Z(G)\cong\Omega_8^+(2)$ or $\Omega_8^+(2).2$, and $\dim(V)=8$. 
\end{prop}

%%, and $G\in \GG[7]$.

\begin{proof} \textbf{Case 1: } Assume that 
$G_0/Z(G_0)\cong\Omega_{2n+1}(q)$ (where $q$ is odd). By Table 
\ref{tab:classical2}, if $V$ is a minimally active module for $G_0$ or one 
of its covers, then $\dim(V)\ge q^{n-1}(q^{n-1}-1)$, except possibly for 
$\Omega_7(3)$, which will be handled separately. As with symplectic groups, 
we use an inductive argument to reduce to the case where $p\mid \Phi_d(q)$ 
for $d=n,2n$, starting with the base case of the induction, 
$\Omega_5(q)=\Sp_4(q)$.

If $p\mid \Phi_n(q)$, then $p\leq (q^n-1)/(q-1)$, and so $\dim(V)<2p$ becomes
	\[ q^{n-1}(q^{n-1}-1)<\frac{2(q^n-1)}{q-1} \quad\implies\quad 
	2q^2\le q^{n-1}(q-1)<\frac{2(q^n-1)}{q^{n-1}-1}\le 2(1+q),\]
which has no solutions.

If $p\mid \Phi_{2n}(q)$, then $p\leq (q^n+1)$, and so $\dim(V)<2p$ becomes
	\[ q^{n-1}(q^{n-1}-1)\leq 2(q^n+1),\]
yielding $q^{2n-2}\leq 2q^n+q^{n-1}+2$ for $n,q\geq 3$, again clearly 
having no solutions. Thus there are no minimally active modules when $p\mid 
\Phi_n(q)$ or $p\mid \Phi_{2n}(q)$. So by induction, and since the only groups of type $\Omega_5(q)\cong\PSp_4(q)$ with minimally active modules are those for $q=3$ (Proposition \ref{prop:finishedspn}(ii)), the only cases left to consider are those where $G_0/Z(G_0)\cong\Omega_7(3)$ and $p=5\mid\Phi_4(3)$, $7=\Phi_6(3)$, or $13=\Phi_3(3)$. 

When $p=5$, Proposition \ref{prop:inactiverestriction}, applied with $H<G$ of type $\Omega_5(q)=\PSp_4(q)$, says that the dimension of a minimally active $\F_5G_0$-module $V$ is at most twice that of a minimally active $\F_pH$-module (since $\Omega_7(3)$ is generated by two conjugates of $\Omega_5(3)$). Thus $\dim(V)\le12$ by Proposition \ref{prop:finishedspn}(ii), which is impossible by Table \ref{tab:classical2}.

When $p=7$ or $13$, $C_{\Omega_7(3)}(\UUU)$ is cyclic (of order $14$ or $13$, respectively), so the centralizer of $\UUU$ is abelian in each cover of $\Omega_7(3)$. Hence $\dim(V)<2p$, which again contradicts Table \ref{tab:classical2}.

\smallskip

\noindent\textbf{Case 2: } We move on to the case where 
$G=\Omega_{2n}^+(q)$ (for $n\ge4$). Here we are concerned with $p\mid 
\Phi_d(q)$, where by Table \ref{tab:classical}, $d$ is regular when $d=2n-2$, or $d\in\{n-1,n\}$ is odd. The Landazuri--Seitz 
bound from Table \ref{tab:classical2} is $q^{n-2}(q^{n-1}-1)$, with 
$\Omega_8^+(2)$ the only exception. This is also the only case which has an 
exceptional Schur multiplier. 

If $d=n-1$ or $d=n$, then $p\leq (q^n-1)/(q-1)$, and $\dim(V)<2p$ implies
	\[ q^2\le q^{n-2}(q-1)<2(q^n-1)/(q^{n-1}-1) \le 2(1+q). \]
This is satisfied only for $\Omega_8^+(2)$, which we need to consider 
separately in any case. If $d=2n-2$, then $p\leq (q^n+1)$, and we have
	\[ q^{n-2}(q^{n-1}-1)<2(q^n+1) \quad\implies\quad
	q\le q^{n-3}<2+q^{-2}+2q^{-n},\]
which again is only satisfied for $\Omega_8^+(2)$. 

Now, $\Omega_8^+(2)\in\G$ only for $p=7$, 
$C_{\Omega_8^+(2)}(\UUU)=\UUU$ in this case, and thus $C_{G_0}(\UUU)$ is 
abelian when $G_0$ is any central extension of $\Omega_8^+(2)$. Thus we can 
always assume that $\dim(V)<2p=14$. The smallest non-trivial representation 
of $\Omega_8^+(2)$ itself is of dimension $28$, which is too large. There 
is an $8$-dimensional representation for the exceptional cover $2\cdot 
\Omega_8^+(2)$, it is minimally active, and it extends to a module for 
$2\cdot\Omega_8^+(2).2$ (the Weyl group of $E_8$). 

Since the only minimally active module is for the exceptional 
cover $2\cdot\Omega_8^+(2)$, it does not extend to a minimally active 
module over a group of type $\Omega_{10}^+(2)$.

\smallskip

\noindent\textbf{Case 3: } Finally, consider $G$ of type $\Omega_{2n}^-(q)$ 
(again for $n\ge4$). By Table \ref{tab:classical}, $d$ is 
regular when $d=2n$, $2n-2$, or $n$ is even and $d=n-1$, and we first 
consider $p\mid\Phi_d(q)$ for such $d$. The Landazuri--Seitz bound gives 
$\dim(V)\ge(q^{n-1}+1)(q^{n-2}-1)$ (Table \ref{tab:classical2}), and we do 
the same analysis as for the plus-type case.

\iffalse
This time the important $d$, for induction 
purposes, where $p\mid\Phi_d(q)$, are the regular numbers $d=2n,2n-2$, and 
(if $n$ is even) $d=n-1$. 
\fi

If $d=n-1$, then $p\leq (q^{n-1}-1)/(q-1)$, and since $\dim(V)<2p$, 
	\[(q^{n-1}+1)(q^{n-2}-1)(q-1)< 2(q^{n-1}-1),\]
so that $(q^{n-2}-1)(q-1)<2$, which has no solutions when $n\ge4$. If $d=2n-2$ 
or $2n$, then $p\leq q^n+1$, and since $\dim(V)<2p$, we have
	\[ (q^{n-1}+1)(q^{n-2}-1) < 2(q^n+1) \quad\implies\quad
	q\le q^{n-3} < 2+3q^{-n}+q^{-1}-q^{-2}. \]
Thus $(2n,q)=(8,2)$. Also, $p=\Phi_6(2)=3$ or $p=\Phi_8(2)=17$, and since 
$3^2\bmid|\Omega_8^-(2)|$, we have $p=17$. However, upon checking the Brauer 
character table from \cite[p.248]{modatlas}, we see that the smallest 
non-trivial module has dimension $34=2p$, so this case does not occur.

We now must check the base case of our induction. For 
$G_0\cong\Omega_8^-(q)$, $G_0\in\G$ only when $p\mid\Phi_d(q)$ for 
$d=3,4,6,8$. We have already handled $d=3,6,8$, but we are left with $d=4$, 
for which the centralizer of $\UUU$ might not be abelian. As with the 
$\Omega_7(q)$ case, we take two copies of $H=\SU_4(q)\cong\Omega_6^-(q)$ 
inside $G_0\cong\Omega_8^-(q)$ that generate $G$, and apply 
Proposition \ref{prop:inactiverestriction} to get $\dim(V)<4p\leq 
4(q^2+1)$. Thus the lower bound in Table \ref{tab:classical2} becomes
	\[ (q^3+1)(q^2-1)\le\dim(V) < 4(q^2+1),\]
which clearly has no solutions, not even for $q=2$. This completes the proof.
\end{proof}

\section{Exceptional groups}
\label{s:exceptional}

In this section we treat the exceptional groups of Lie type. We maintain the notation of the previous section, so that $G$ is of type $\gg(q)$, an exceptional group of Lie type, $p\nmid q$ is a prime dividing $G$, a Sylow $p$-subgroup $\UUU$ has order $p$ and is generated by $x$, and $p$ has order $d$ modulo $q$, so that $p\mid\Phi_d(q)$. In Table \ref{tab:exceptional2}, we list the minimal possible dimensions for $V$, as determined in \cite{landazuriseitz1974} or \cite{seitzzalesskii1993}. 
\begin{table}[ht]
%%\begin{small}
\renewcommand{\arraystretch}{1.3}
\[ \begin{array}{|c|c|c|c|}
\hline G & \textup{Lower bound for $\dim(V)$} & \textup{Ref.} & \textup{Exceptions}
\\\hline \lie2B2(q) & (q-1)\sqrt{q/2} & \text{[LS]} & l(\lie2B2(8))=8
\\\hline \lie2G2(q) & q(q-1) & \text{[LS]} & \textup{---}
\\\hline \lie2F4(q) & q^4(q-1)\sqrt{q/2} & \text{[LS]} & \textup{---}
\\\hline G_2(q) & q(q^2-1) & \text{[SZ]} & \dbl{l(G_2(3))=14}{l(G_2(4))=12}
\\\hline \lie3D4(q) & q^3(q^2-1) & \text{[LS]} & \textup{---}
\\\hline F_4(q) & \begin{array}{cl}q^6(q^2-1)&(2\nmid q) \\
q^7(q^3-1)(q-1)&(2\mid q)\end{array} & \text{[LS]} & 
l(F_4(2))\ge44
\\\hline E_6^\pm(q) & q^9(q^2-1)  & \text{[SZ]} & \textup{---}
%%\\ \lie2E6(q) & 
\\\hline E_7(q) & q^{15}(q^2-1)  & \text{[LS]} & \textup{---}
\\\hline E_8(q) & q^{27}(q^2-1)  & \text{[LS]} & \textup{---}
\\ \hline
\end{array} \]
%%\end{small}
\caption{Minimal representation dimensions of exceptional groups}
\label{tab:exceptional2}
\end{table}

\begin{prop}\label{prop:determineexc} 
Let $G\in\GG$ be such that $E(G/Z(G))$ is an exceptional simple group of 
Lie type in defining characteristic different from $p$. 
Set $G_0=E(G)$. Let $V$ be a non-trivial, minimally active 
$\F_pG$-module. Then one of the following holds: either
\begin{enumi}
\item $G_0\cong\lie2B2(8)$, $G/Z(G)\cong\lie2B2(8):3$, $p=13$, and $\dim(V)=14$; or
\item $G_0\cong\SL_2(8)$, $G/Z(G)\cong\lie2G2(3)\cong \SL_2(8):3$, $p=7$, and $\dim(V)=7,8$; or
\item $G_0\cong\PSU_3(3)$, $G/Z(G)\cong G_2(2)\cong \PSU_3(3):2$, $p=7$, and $\dim(V)=6,7$; or
\item $G_0\cong G_2(3)$, $G/Z(G)\cong G_2(3).2$, $p=13$, and $\dim(V)=14$.
\end{enumi}
\end{prop}

%%a group of type an exceptional group of Lie type in 
%%defining characteristic different from $p$, 

\begin{proof} In all cases, $G_0/Z(G_0)$ is one of the exceptional groups 
listed in Tables \ref{tab:exceptional} and \ref{tab:exceptional2}, and 
unless stated otherwise, $p\mid\Phi_d(q)$ for one of the $d$ listed in 
the first table. Thus $\UUU\in\sylp{G_0}$ has order $p$. As usual, $V$ is 
assumed to be a minimally active $\F_pG_0$-module.

\smallskip

\noindent\textbf{Case 1: } Assume that $G_0$ is a Suzuki or Ree group or 
$G_0\cong\lie3D4(q)$. In all of these cases, $d$ is regular by Table 
\ref{tab:exceptional}. Aside from $\lie2B2(8)$, these groups have no 
exceptional covers \cite[\S\,4.2.4]{wilsonbook}, so $C_{G_0}(\UUU)$ is 
abelian in all other cases, and $\dim(V)<2p$. 

If $G_0/Z(G_0)$ is a Suzuki group $\lie2B2(q)$, then $p$ divides one of $q-1$, 
$q+\sqrt{2q}+1$ and $q-\sqrt{2q}+1$. In particular, $p\le q+\sqrt{2q}+1$, 
and since $\dim(V)\ge(q-1)\sqrt{q/2}$ by Table \ref{tab:exceptional2},
	\[ (q-1)\sqrt{q/2} \le \dim(V) < 2p \le 2q + 2\sqrt{2q} + 2
	\quad\implies\quad \sqrt{q}<\sqrt8(1+\tfrac1q) + 5/\sqrt{q}. \]
Hence there are no minimally active modules for $q\geq 32$. For $q=8$ we 
can just go through the known character tables \cite{modatlas}. If 
$G_0\cong\lie2B2(8)$ (i.e., not an exceptional cover), then of the three primes 
$p=5$, $7$, and $13$ for which $G_0\in\G$, there is a non-trivial 
$\F_pG_0$-module of dimension less than $2p$ only for $p=13$: the two 
$14$-dimensional simple modules in this case are minimally active and each 
of them extends to $G_0:3\in\GG[13]$. If $G_0\cong2\cdot\lie2B2(8)$, then 
$\Out(G_0)=1$ (since $\Out(\lie2B2(8))\cong C_3$ acts faithfully on the 
Schur multiplier $C_2^2$), and $G_0\in\GG$ only for $p=5$. In this case, 
there is an $8$-dimensional $\F_5G_0$-module, but by Proposition 
\ref{prop:dimlt2p-1}(a) and since $C_{G_0}(\UUU)$ is cyclic of order $10$, 
it cannot be minimally active since $(8-5)\nmid(p-1)=4$.

If $G_0$ is a small Ree group $\lie2G2(q)$, then $p$ divides one of $q-1$, 
$q+1$, $q-\sqrt{3q}+1$ or $q+\sqrt{3q}+1$. Together with the lower bound 
for $\dim(V)$ in Table \ref{tab:exceptional2}, this gives 
	\[ q(q-1) \le \dim(V) < 2p \leq 2q+2\sqrt{3q}+2. \]
For $q\geq 27$ we therefore can have no minimally active modules. For 
$q=3$, we have that $\lie2G2(3)\cong \SL_2(8):3$, and this case was 
covered in Proposition \ref{prop:finishedsln}.

If $G_0$ is a large Ree group $\lie2F4(q)$, then $p$ divides $\Phi_6(q)$ or 
one of the two polynomials into which $\Phi_{12}(q)$ splits over 
$\Z[\sqrt2]$. The Landazuri--Seitz bound for $\dim(V)$ gives 
	\[ q^4 \le q^4(q-1)\sqrt{q/2} \le \dim(V)<2p\leq 4q^2+4q+2. \]
So there are no minimally active modules if $q\ge8$. If $q=2$, then 
$G_0\in\G$ only for $p=13$, $N_{G_0}(\UUU)/\UUU$ is abelian in this case, 
so $\dim(V)\le p+1=14$, while $\dim(V)\ge16$ by the above bound.

If $G$ is a triality group $G=\lie3D4(q)$, then by Table 
\ref{tab:exceptional}, $p$ divides $\Phi_{12}(q)=q^4-q^2+1$, and together 
with the bound in Table \ref{tab:exceptional2}, this implies that
	\[ q^3(q^2-1) \le \dim(V) < 2p\leq 2q^4-2q^2+2, \] 
and hence $q=2$. But in this case, $p=\Phi_{12}(2)=13$, and the smallest 
non-trivial module has dimension $26$ by \cite{modatlas}. So there are no 
minimally active modules in any of these cases.

\medskip

\noindent\textbf{Case 2: } We next consider the small exceptional groups 
$G_2(q)$ and $F_4(q)$. Again for these groups, by Table 
\ref{tab:exceptional} and Proposition \ref{reg=>CGU-abel}, $C_{G_0}(\UUU)$ 
is abelian for all primes $p$ such that $G_0\in\G$ (unless possibly 
$G_0$ is one of the exceptional covers $3\cdot G_2(3)$, $2\cdot G_2(4)$, or 
$2\cdot F_4(2)$), and hence $\dim(V)<2p$.

If $G_0=G_2(q)$, then $p\mid\Phi_3(q)$ or $p\mid \Phi_6(q)$, so 
in particular $p\leq q^2+q+1$. Together with the 
Seitz--Zalesskii bound in Table \ref{tab:exceptional2}, this gives 
	\[ q(q^2-1) \le \dim(V) < 2p \le 2q^2+2q+2, \]
for $q\geq 5$, yielding no solutions. We have already dealt with the case 
$G_2(2)\cong \PSU_3(3):2$. 

If $G_0/Z(G_0)\cong G_2(3)$, then $G_0\in\G$ implies $p=7$ or $13$, and 
$C_{G_0}(\UUU)$ is abelian since the Sylow $7$- and $13$-subgroups of 
$G_2(3)$ are self-centralizing. There are no non-trivial $\F_7G_0$-modules 
of degree less than $14$, while there are $14$-dimensional simple modules 
for $G_2(3):2$ and $p=13$ (see \cite{modatlas}), and they are minimally 
active.

If $G_0/Z(G_0)\cong G_2(4)$, then $G_0\in\G$ implies $p=7$ or $13$, and 
$C_{G_0}(\UUU)$ is abelian since the centralizers of the Sylow $7$- and 
$13$-subgroups of $G_2(4)$ are cyclic of order $21$ or $13$, respectively 
\cite{atlas}. Thus $\dim(V)<2p$, and by \cite{modatlas}, there are 
$12$-dimensional modules for $2\cdot G_2(4)$ over $\F_7$ and over $\F_{13}$ 
to be considered. When $p=7$, this module cannot be minimally active 
by Proposition \ref{prop:dimlt2p-1}(a) and since $(12-7)\nmid(p-1)=6$. When $p=13$, we have $|\Aut_{G_0}(\UUU)|=6$, so we need to extend to 
$G=2\cdot G_2(4).2$. However, from the table on \cite[p.277]{modatlas}, we 
see that there are irrationalities in the Brauer character of this 
representation, and by \cite[p.291]{modatlas}, it is defined only over 
$\F_{13^2}$.

For $G_0=F_4(q)$, the Landazuri--Seitz bounds in Table \ref{tab:exceptional2} 
gives the inequalities 
	\[ q^6 \le \dim(V) < 2p \le 2\cdot\max\{\Phi_8(q),\Phi_{12}(q)\} = 2q^4+2 \]
when $q>2$, and $44<2\cdot2^4+2=34$ when $q=2$. Since these have no 
solutions, $G_0$ has no minimally active modules.

\smallskip

\noindent\textbf{Case 3: } Assume that $G_0=E_6^\gee(q)$ (the 
universal or adjoint group). If $p\bmid\Phi_d(\gee q)$ for $d=8,9,12$, 
then $C_{G_0}(\UUU)$ is abelian (see Table \ref{tab:exceptional}), so 
$\dim(V)<2p$, while $\dim(V)\ge q^9(q^2-1)$ by Table 
\ref{tab:exceptional2}. Thus 
	\[ q^9\le q^9(q^2-1)<2\Phi_d(\gee q)\le2\Phi_9(q)=q^6+q^3+1, \]
which is impossible. This leaves $p\bmid\Phi_5(\gee q)$, in which case 
$\UUU\le H<G$ for $H$ of type $\PSL_6^\gee(q)$. We already saw in the proofs of 
Propositions \ref{prop:finishedsln} (Case 2) and \ref{prop:finishedsun} 
(Case 1) that $\SL_6^\gee(q)$ has no minimally active projective 
representations over $\F_p$ when $p\mid\Phi_5(\gee q)$, and so $G_0$ 
has no minimally active modules.

If $G_0$ is an exceptional cover $2\cdot\lie2E6(2)$ 
\cite[\S\,4.11]{wilsonbook}, then $G_0\in\G$ 
only for $p=11,13,17,19$. In the last three cases, the Sylow $p$-subgroup 
of $\lie2E6(2)$ is self-centralizing, so $C_{G_0}(\UUU)$ is abelian, 
$\dim(V)<2p$, and the above argument applies. If $p=11=\Phi_5(-2)$, then 
the above comparison with $U_6(2)$ again applies. 

Assume that $G_0=E_7(q)$ and $p\bmid\Phi_d(q)$. If $d=7,9,14,18$, then 
$C_{G_0}(\UUU)$ is abelian (see Table \ref{tab:exceptional}), so 
$\dim(V)<2p$, while $\dim(V)\ge q^{15}(q^2-1)>q^{15}$ by Table 
\ref{tab:exceptional2}. Thus 
$q^{15}<2\Phi_d(q)\le2\Phi_7(q)\le q^7-1$, which is impossible. 
If $d=5,8,10,12$, then $\UUU\le H<G_0$ for $H\cong E_6^\pm(q)$, and we just 
showed that these groups have no minimally active modules over $\F_p$. So 
$G_0$ has no minimally active modules.

Finally, assume that $G_0=E_8(q)$ and $p\bmid\Phi_d(q)$. If $d=15,20,24,30$, 
then $C_{G_0}(\UUU)$ is abelian (see Table \ref{tab:exceptional}), so 
$\dim(V)<2p$, while $\dim(V)\ge q^{27}(q^2-1)$ by Table 
\ref{tab:exceptional2}. Since $\Phi_d$ has degree $8$ in each of these 
cases, one easily sees that we cannot have $\dim(V)<2p$. If $d=7,9,14,18$, 
then $\UUU\le H<G_0$ for $H\cong E_7(q)$, and we just showed that these 
groups have no minimally active modules over $\F_p$. So again in this case, 
$G_0$ has no minimally active modules.
\end{proof}

\iffalse
\begin{prop}\label{prop:determineexc} If $G\in \GG$ is of type an exceptional group of Lie type and $V$ is a minimally active module for $G$ then 
\begin{enumi}
\item $G=\lie2B2(8):3$, $p=13$, $\dim(V)=14$,
\item $G=\lie2G2(3)\cong \SL_2(8):3$, $p=7$, $\dim(V)=7,8$,
\item $G=G_2(2)\cong \PSU_3(3):2$, $p=7$, $\dim(V)=6,7$, or
\item $G=G_2(3).2$, $p=13$, $\dim(V)=14$.
\end{enumi}
\end{prop}
\fi

\end{document}